\documentclass{amsart}
\usepackage[british]{babel}
\usepackage{amsmath,amssymb,amsthm,amscd}
\usepackage{graphicx,subfigure}
\usepackage{color}
\usepackage{pb-diagram} 
\usepackage{tabularx}
\usepackage{multirow}
\usepackage{longtable}
\usepackage{hyperref}
\usepackage[foot]{amsaddr}

\newtheorem{theorem}{Theorem}[section]

\newtheorem{lemma}[theorem]{Lemma}

\theoremstyle{remark} 
\newtheorem{remark}[theorem]{Remark}

\numberwithin{equation}{section}

\newcommand{\al}{\alpha}

\newcommand{\ep}{\varepsilon}

\newcommand{\cA}{{\mathcal A}}

\newcommand{\cD}{{\mathcal D}}

\newcommand{\cO}{{\mathcal O}}

\newcommand{\cX}{{\mathcal X}}

\newcommand{\RR}{{\mathbb R}}
\newcommand{\CC}{{\mathbb C}}

\newcommand{\TT}{{\mathbb T}}
\newcommand{\ZZ}{{\mathbb Z}}

\newcommand{\PP}{{\mathbb P}}

\newcommand{\abs}[1]{|{#1}|}
\newcommand{\Abs}[1]{\left|{#1}\right|}
\newcommand{\norm}[1]{\|{#1}\|}
\newcommand{\Norm}[1]{\left\|{#1}\right\|}
\newcommand{\aver}[1]{\langle{#1}\rangle}
\newcommand{\Lie}[1]{\mathfrak{L}_{#1}}
\newcommand{\Loper}[1]{\cX_{#1}}
\newcommand{\R}[1]{\mathfrak{R}_{#1}}
\newcommand{\Dioph}[2]{\cD_{#1,#2}}

\def\to{\rightarrow}

\def\epsilon{\varepsilon}

\def\ii{\mathrm{i}}
\def\ee{\mathrm{e}}
\def\re{\mathrm{Re}}
\def\im{\mathrm{Im}}
\def\id{\mathrm{id}}
\def\dist{\mathrm{dist}}

\def\dif{ {\mbox{\rm d}} }
\def\Dif{ {\mbox{\rm D}} }
\def\Tan{ {\mbox{\rm T}} }
\def\pd{ \partial }

\def\sform{\mbox{\boldmath $\omega$}}
\def\aform{\mbox{\boldmath $\alpha$}}
\def\gform{\mbox{\boldmath $g$}}

\def\J{\mbox{\boldmath $J$}}
\def\I{\mbox{\boldmath $I$}}
\def\mani{{\mathcal{M}}}
\def\Anal{{\cA}}
\def\B{{\mathcal{B}}}
\def\U{{\mathcal{U}}}
\def\torus{{\mathcal{K}}}

\def\LieO{\mathfrak{L}}     % Lie derivative without subscript
\def\H{h}                   % Hamiltonian
\def\Tdown{H}               % Inferior block for Tc
\def\homega{\hat \omega}    % Iso-energetic frequency (ampliada)
\def\Elag{\Omega_L} % Error lagrangian
\def\Esym{E_{\mathrm{sym}}} % Error symplectic
\def\Ered{E_{\mathrm{red}}} % Error reducibility
\def\Elin{E_{\mathrm{lin}}} % Error linear system
\def\xio{\xi^{\omega}}      % Relative correction of the frequency

\def\DLieo{\Delta \LieO}
\def\DeltaK{\Delta K}
\def\Deltao{\Delta \omega}
\def\DeltaL{\Delta L}
\def\DeltaLT{\Delta {L^\top}}
\def\DeltaG{\Delta G}
\def\DeltaJ{\Delta J}
\def\DeltaJT{\Delta J^\top}
\def\DeltaGL{\Delta {G_L}}
\def\DeltatOmega{\Delta {\tilde \Omega}}
\def\DeltatOmegaL{\Delta {\tilde \Omega_L}}
\def\DeltaA{\Delta A}
\def\DeltaB{\Delta B}
\def\DeltaNO{\Delta N^0}
\def\DeltaNOT{\Delta N^{0\top}}
\def\DeltaN{\Delta N}
\def\DeltaNT{\Delta N^\top}
\def\DeltaLoperN{\Delta \Loper{N}}
\def\DeltaT{\Delta T}
\def\DeltaTc{\Delta T_c}
\def\DeltaTOI{\Delta \aver{T}^{-1}}
\def\DeltaTcOI{\Delta \aver{T_c}^{-1}}
\def\DeltaLieK{\Delta \LieO K}
\def\LieDeltaK{\LieO \Delta K}
\def\DeltaLieL{\Delta \LieO L}
\def\DeltaLieLT{\Delta \LieO L^\top}
\def\DeltaLieG{\Delta \LieO G}
\def\DeltaLieGL{\Delta \LieO {G_L}}
\def\DeltaLietOmega{\Delta \LieO {\tilde \Omega}}
\def\DeltaLietOmegaL{\Delta \LieO {\tilde \Omega_L}}
\def\DeltaLieB{\Delta \LieO B}
\def\DeltaLieA{\Delta \LieO A}
\def\DeltaLieJ{\Delta \LieO J}
\def\DeltaLieNO{\Delta \LieO N^0}
\def\DeltaLieN{\Delta \LieO N}

\def\Deltao{\Delta \omega}

\def\cteOmega{c_{\Omega,0}}
\def\cteDOmega{c_{\Omega,1}}
\def\ctetOmega{c_{\tilde\Omega,0}}
\def\cteDtOmega{c_{\tilde\Omega,1}}
\def\cteDDtOmega{c_{\tilde\Omega,2}}
\def\cteJ{c_{J,0}}
\def\cteDJ{c_{J,1}}
\def\cteDDJ{c_{J,2}}
\def\cteJT{c_{J^{\top}\hspace{-0.5mm},0}}
\def\cteDJT{c_{J^{\top}\hspace{-0.5mm},1}}
\def\cteG{c_{G,0}}
\def\cteDG{c_{G,1}}
\def\cteDDG{c_{G,2}}
\def\cteDH{c_{\H,1}}
\def\cteDp{c_{p,1}}
\def\cteDc{c_{c,1}}
\def\cteDDc{c_{c,2}}
\def\cteDpT{c_{p^\top\hspace{-0.5mm},1}}
\def\cteXH{c_{X_\H,0}}
\def\cteDXH{c_{X_\H,1}}
\def\cteDDXH{c_{X_\H,2}}
\def\cteDXHT{c_{X_\H^\top\hspace{-0.5mm},1}}

\def\cteXp{c_{X_p,0}}
\def\cteXpT{c_{X_p^\top\hspace{-0.5mm},0}}
\def\cteDXp{c_{X_p,1}}
\def\cteDXpT{c_{X_p^\top\hspace{-0.5mm},1}}
\def\cteDDXp{c_{X_p,2}}
\def\cteDDXpT{c_{X_p^\top\hspace{-0.5mm},2}}
\def\cauxT{\nu} % Auxiliary constant 1
\def\sigmaDK{\sigma_{K}}
\def\sigmaDKT{\sigma_{K^\top\hspace{-0.5mm}}}
\def\sigmaB{\sigma_B}
\def\sigmaT{\sigma_T}
\def\sigmaTc{\sigma_{T_c}}
\def\sigmao{\sigma_\omega}

\def\CE{C_{E}}
\def\CEo{C_{E^\omega}}
\def\CEc{C_{E_c}}
\def\Clag{C_{\Omega_L}}
\def\Csym{C_{\mathrm{sym}}}
\def\Cred{C_{\mathrm{red}}}
\def\Creduu{C_{\mathrm{red}}^{1,1}}
\def\Creddu{C_{\mathrm{red}}^{2,1}}
\def\Creddd{C_{\mathrm{red}}^{2,2}}
\def\Clin{C_{\mathrm{lin}}}
\def\Clino{C_{\mathrm{lin}}^\omega}

\def\Comega{{C_\omega}}
\def\COmegaK{C_{\Omega_{K}}}
\def\CT{C_T}
\def\CA{C_A}
\def\CL{C_L}

\def\CLT{C_{L^\top}}
\def\CNO{C_{N^0}}
\def\CNOT{C_{N^{0,\top}}}
\def\CN{C_N}
\def\CNT{C_{N^\top}}
\def\CLoperL{C_{\Loper{L}}}
\def\CLoperLT{C_{\Loper{L}^\top}}
\def\CLoperN{C_{\Loper{N}}}
\def\CxiNO{C_{\xi^N_0}}
\def\Cxio{C_{\xi^\omega}}
\def\CxiN{C_{\xi^N}}
\def\CxiL{C_{\xi^L}}
\def\Cxi{C_{\xi}}
\def\CGL{C_{G_L}}
\def\CtOmegaL{C_{\tilde \Omega_L}}

\def\CDeltaK{C_{\DeltaK}}
\def\CDeltao{C_{\Deltao}}
\def\CDeltaL{C_{\DeltaL}}
\def\CDeltaLT{C_{\DeltaLT}}
\def\CDeltaA{C_{\DeltaA}}
\def\CDeltaB{C_{\DeltaB}}
\def\CDeltaNO{C_{\DeltaNO}}
\def\CDeltaNOT{C_{\DeltaNOT}}
\def\CDeltaN{C_{\DeltaN}}
\def\CDeltaNT{C_{\DeltaNT}}
\def\CDeltaG{C_{\DeltaG}}
\def\CDeltaGL{C_{\DeltaGL}}
\def\CDeltatOmega{C_{\DeltatOmega}}
\def\CDeltatOmegaL{C_{\DeltatOmegaL}}
\def\CDeltaLoperN{C_{\DeltaLoperN}}
\def\CDeltaT{C_{\DeltaT}}
\def\CDeltaTc{C_{\DeltaTc}}
\def\CDeltaTOI{C_{\DeltaTOI}}
\def\CDeltaTcOI{C_{\DeltaTcOI}}
\def\CDeltaI{C_{\Delta,1}}
\def\CDeltaII{C_{\Delta,2}}
\def\CDeltaIII{C_{\Delta,3}}
\def\CDeltatot{C_{\Delta}}
\def\CDeltaJ{C_{\DeltaJ}}
\def\CDeltaJT{C_{\DeltaJT}}

\def\CLieOmegaK{C_{\LieO \Omega_K}}
\def\CLieK{C_{\LieO K}}
\def\CLieL{C_{\LieO L}}
\def\CLieLT{C_{\LieO L^\top}}
\def\CLieB{C_{\LieO B}}
\def\CLieA{C_{\LieO A}}
\def\CLieNO{C_{\LieO N^0}}
\def\CLieN{C_{\LieO N}}
\def\CLieJ{C_{\LieO J}}
\def\CLieG{C_{\LieO G}}
\def\CLietOmega{C_{\LieO \tilde \Omega}}
\def\CLiexi{C_{\LieO \xi}}
\def\CLiexiL{C_{\LieO \xi^L}}
\def\CLiexiN{C_{\LieO \xi^N}}
\def\CLieGL{C_{\LieO G_L}}
\def\CLieOmegaL{C_{\LieO \Omega_L}}
\def\CLietOmegaL{C_{\LieO \tilde \Omega_L}}

\def\CLieDeltaK{C_{\LieDeltaK}}
\def\CDeltaLieK{C_{\DeltaLieK}}
\def\CDeltaLieL{C_{\DeltaLieL}}
\def\CDeltaLieLT{C_{\DeltaLieLT}}
\def\CDeltaLieG{C_{\DeltaLieG}}
\def\CDeltaLieGL{C_{\DeltaLieGL}}
\def\CDeltaLieB{C_{\DeltaLieB}}
\def\CDeltaLieA{C_{\DeltaLieA}}
\def\CDeltaLieJ{C_{\DeltaLieJ}}
\def\CDeltaLieNO{C_{\DeltaLieNO}}
\def\CDeltaLieN{C_{\DeltaLieN}}
\def\CDeltaLietOmega{C_{\DeltaLietOmega}}
\def\CDeltaLietOmegaL{C_{\DeltaLietOmegaL}}

\def\hCDelta{\hat{C}_\Delta}

\allowdisplaybreaks

\begin{document}

\title[KAM theory for partially integrable systems]{
A-posteriori KAM theory with optimal estimates for
partially integrable systems}

\author{Alex Haro$^{\mbox{\textdagger}}$}
\address[\textdagger]{Departament de Matem\`atiques i Inform\`atica, Universitat de Barcelona,
Gran Via 585, 08007 Barcelona, Spain.}
\email{alex@maia.ub.es}

\author{Alejandro Luque$^{\mbox{\textdaggerdbl}}$}
\address[\textdaggerdbl]{Department of Mathematics, Uppsala University, 
Box 480, 751 06 Uppsala, Sweden}
\email{alejandro.luque@math.uu.se}

\begin{abstract}
In this paper we present 
a-posteriori KAM results for existence of 
$d$-dimensional
isotropic invariant tori
for 
$n$-DOF
Hamiltonian systems with additional 
$n-d$ independent
first integrals in involution.
We carry out a covariant formulation
that does not require the use of action-angle variables
nor symplectic reduction techniques.
The main advantage is that we overcome the curse of dimensionality 
avoiding the practical shortcomings produced by 
the use of reduced coordinates, which may
cause difficulties and underperformance when quantifying the
hypotheses of the KAM theorem
in such reduced coordinates.
The results include ordinary
and (generalized) iso-energetic KAM theorems.
The approach is suitable to perform numerical computations 
and computer assisted proofs.
\end{abstract}

\maketitle

\tableofcontents

\section{Introduction}

Persistence under perturbations of regular (quasi-periodic) motion
is one of the most important problems in 
Mechanics and Mathematical Physics, and has
deep implications in Celestial and Statistical Mechanics.
Classical perturbation theory experienced a breakthrough around sixty years ago,
with the work of Kolmogorov \cite{Kolmogorov54}, Arnold \cite{Arnold63a} and
Moser \cite{Moser62}, the founders of what is nowadays known as KAM theory.
They overcame the so called small divisors 
problem, that might prevent the
convergence of the series appearing in perturbative 
methods.
Since then, KAM theory has become 
a full body of knowledge 
that connects fundamental mathematical ideas,
and the literature contains eminent contributions and applications in different contexts
(e.g. unfoldings and bifurcations of invariant tori
\cite{BroerHTB90,BroerW00},
quasi-periodic solutions in
partial differential equations
\cite{BertiBP14,EliassonFK15,Poschel96},
hard implicit function theorems
\cite{Moser66b,Moser66a,Zehnder75,Zehnder76},
stability of perturbed planetary problems
\cite{Arnold63b, ChierchiaP11, Fejoz04}, 
conformally symplectic systems and the spin-orbit problem
\cite{CallejaCLa,CellettiC09},
reversible systems
\cite{Sevryuk07},
or existence of vortex tubes in the Euler equation
\cite{EncisoP15}, just to mention a few).

The importance and significance of KAM theory in 
Mathematics, Physics and Science in general
are accounted in the popular book \cite{Dumas14}. 
But, although KAM theory holds for general dynamical systems under very
mild technical assumptions, its application to concrete systems becomes
a challenging problem. Moreover, the ``threshold of validity'' of the
theory (the size of the perturbation strength for which KAM theorems can be
applied) seemed to be absurdly small in applications to physical 
systems\footnote{This source of skepticism was
pointed out by the distinguished astronomer M. H\'enon, who found a threshold of
validity for the perturbation of the order of $10^{-333}$ in early KAM theorems.}.
With the advent of computers and new developed 
methodologies, the distance between theory and practice has been shortened (see e.g. \cite{CellettiC07} 
for a illuminating historical introduction
and \cite{CellettiC88,LlaveR90} for pioneering computer-assisted applications).
But new impulses have to be made in order to make KAM theory fully applicable to realistic physical systems. 
In this respect, recent years have witnessed a revival of 
this theory.
One direction
that have experienced a lot of progress is the a-posteriori approach based on the parameterization method
\cite{Llave01,GonzalezJLV05,HaroCFLM16,Sevryuk14}.

The main signatures of the parameterization method
are the application of the a-posteriori approach and a covariant formulation, free 
of using any particular system of coordinates.
In fact, the method was baptized as KAM theory
\emph{without action-angle variables} in  \cite{GonzalezJLV05}.
Instead of performing canonical transformations, 
the strategy consists in solving 
the invariance equation for an invariant torus 
by correcting iteratively an approximately invariant one. 
If
the initial approximation is good enough
(relatively to some non-degeneracy conditions), then there is a true invariant torus nearby \cite{CellettiC07,
GonzalezJLV05} (see also \cite{CellettiC95,Llave01,
Moser66b,Russmann76b} for precedents). 
Hence one can consider non-perturbative problems in a natural way.
The approach itself implies the traditional approach (of perturbative nature)
and has been extended to different theoretical settings
\cite{CanadellH17a,
FontichLS09,
GonzalezHL13,
LuqueV11}.
A remarkable feature
of the parameterization method is that it leads to 
very fast and efficient numerical methods for the approximation of
quasi-periodic invariant tori
(e.g.~\cite{CallejaL10, CanadellH17b, FoxM14, HuguetLS12})
We refer the reader to the recent monograph~\cite{HaroCFLM16} for
detailed discussions
(beyond the KAM context)
on the numerical implementations of the method, 
examples, and a more complete bibliography. 

A recently reported success
in KAM theory is the 
design of a general methodology to perform computer assisted proofs of existence of 
Lagrangian invariant tori in a non-perturbative setting \cite{FiguerasHL17}.
The methodology has been applied to several 
low dimensional
problems obtaining almost optimal results\footnote{For example,
for the Chirikov standard map, it is proved that the invariant
curve with golden rotation persists up to $\ep\leq 0.9716$,
which corresponds to a threshold value with a defect of 0.004\%
with respect to numerical observations (e.g. \cite{Greene79,MacKay93}).}.
The program is founded on an a-posteriori theorem
with explicit estimates, whose hypotheses are checked using Fast Fourier
Transform (with interval arithmetics) in combination with a sharp control of the discretization error.
This fine control is crucial to estimate the norm of compositions
and inverses of functions, outperforming the use of symbolic manipulations.
An important consequence is that the rigorous computations are executed in a
very fast way, thus allowing to manipulate millions of Fourier modes comfortably.

One of the typical obstacles to directly apply
the above methodology (and KAM theory in general) to
realistic problems in Mechanics
is the presence of additional first integrals
in involution. 
This degeneracy implies that quasi-periodic invariant tori
appear in smooth families of lower dimensional tori.
The constraints linked to these conserved quantities
can be removed by using classical
symplectic reduction techniques
\cite{Cannas01,MarsdenW74}, 
thus obtaining a lower dimensional Hamiltonian system in a quotient
manifold, where Lagrangian tori can be computed.
This approach has
an undeniable theoretical importance (e.g. the moment map is an
object of remarkable relevance \cite{MarsdenMOPR07}) 
and has been successfully used to obtain
perturbative KAM results in significant examples \cite{ChierchiaP11,Fejoz04,PalacianSY14}.
However, the use of symplectic reduction presents serious difficulties when
applying the a-posteriori KAM approach to the reduced system. This is obvious
because, in the new set of coordinates, it may be difficult to quantify the required
control of the norms for both global objects
(e.g. Hamiltonian system or symplectic structure) and local 
objects (e.g. parameterization of the invariant torus, torsion matrix), or also,
these estimates may become sub-optimal to apply a quantitative KAM theorem.
The goal of this paper is overcoming these drawbacks.
%and avoiding the curse of dimensionality produced by the first integrals
%without using symplectic reduction. 
Instead of reducing the system, we will characterize a target lower dimensional torus,
using the original coordinates
of the problem, by constructing a geometrically adapted frame to suitably display the
linearized dynamics on the full family of invariant tori.

%The improvement of results with the a-posteriori approach with respect to traditional approach 
%in computer assisted proofs for realistic values of perturbations have been already demonstrated 
%in rather academical models \cite{FHL17}.
%Our aim is to extend the range of applications, very specially in problems from celestial mechanics 
%for which the techniques can not be applied directly from scratch, due to the presence of degeneracies coming 
%from the existence of additional first integrals. This paper is a first  step in this direction, in the setting of the 
%a-posteriori approach.

In this paper we present two  KAM theorems in a-posteriori format 
for existence of (families of) isotropic\footnote{The invariant tori we consider
in this paper are isotropic, but they are not lower dimensional invariant tori 
in strict meaning (see \cite{BroerHS96}).
The invariant tori 
are neither partially
hyperbolic nor elliptic (see the a-posteriori formulations in
\cite{FontichLS09} and \cite{LuqueV11}),
but partially parabolic tori that appear in families of invariant tori with the same 
frequency vector.}
tori in Hamiltonian systems
with first integrals in involution, avoiding the use of 
action-angle coordinates (in the spirit of \cite{GonzalezJLV05}) 
and symplectic reduction. The first theorem is an ordinary KAM theorem
(also known as \emph{\`a la Kolmogorov})
on existence of an invariant torus with a
fixed frequency vector. 
Of course, if no additional first integrals 
are present in the system, then we recover the
results in \cite{GonzalezJLV05}, with the extra bonus of providing
a geometrically improved scheme with
explicit estimates. 
Our second theorem is a 
version of 
a 
non-perturbative 
iso-energetic KAM theorem (see \cite{BroerH91,DelshamsG96} for perturbative
formulations), in which
one fixes either the energy or one of the first integrals  but modulates
the frequency vector. Actually, we present a general version that allows
us to consider any conserved quantity in involution.
Even if there are no additional first
integrals, the corollary of this result is an iso-energetic theorem
for Lagrangian tori which is a novelty in this covariant formulation.

The statements of the results are written with an eye in the applications.
Hence, we provide explicit estimates in the conditions of
our theorem, so that the hypotheses can be checked using
the computer assisted methodology in \cite{FiguerasHL17}.
In particular, another novelty of this paper is that
estimates are detailed taking advantage of the presence of additional
geometric structures in phase space other than the symplectic structure, such as a Riemannian
metric or a compatible triple, covering a gap in the literature
\cite{HaroCFLM16}. We think researchers interested in applying
these techniques to specific problems can benefit from 
these facts.
It is worth mentioning that the required control
of the first integrals is limited to estimate the norm of objects
that depend only on elementary algebraic expressions and
derivatives.
Quantitative application of these theorems
using computer assisted methods will be presented in a forthcoming work.

The paper is organized as follows.
Section 2 introduces the background and the geometric constructions. 
We present the two main theorems of this paper in 
Section 3: the ordinary KAM theorem and the generalized iso-energetic KAM theorem (with presence of first integrals). 
Some common lemmas are given in Section 4, that control the approximation of different geometric properties.
The proof of the ordinary KAM theorem is given in Section 5, while the proof of the generalized iso-energetic KAM theorem is done in Section 6.
In order to collect the long list of expressions leading to the
explicit estimates and conditions of the theorems, we include
separate tables in the appendix.

\section{Background and elementary constructions}

\subsection{Basic notation}\label{ssec:basic:notation}

We denote by $\RR^m$ and $\CC^m$ the vector spaces of $m$-dimensional vectors with components in 
$\RR$ and $\CC$, respectively, endowed with the norm
\[
|v|= \max_{i=1,\dots,m} |v_i|.
\] 
We consider the real and imaginary projections 
$\re,\im: \CC^m\to \RR^m$, and identify 
$\RR^m \simeq \im^{-1}\{0\} \subset \CC^m$.
Given $U\subset \RR^m$ and $\rho>0$, the complex strip of size $\rho$ is 
$U_\rho= \{ \theta  \in \CC^m \,:\, \re\, \theta \in U \, , \, |\im\, \theta|<\rho \}$.
Given two sets $X,Y \subset \CC^{m}$, $\dist(X,Y)$ is defined as 
$\inf\{|x-y| \,:\, x\in X \, , \, y\in Y\}$.

We denote $\RR^{n_1 \times n_2}$ and $\CC^{n_1\times n_2}$
the spaces of $n_1 \times n_2$ matrices with components in $\RR$ and $\CC$, respectively. 
We will consider the identifications $\RR^m\simeq \RR^{m\times 1}$ and $\CC^m\simeq \CC^{m\times 1}$.
We denote $I_n$ and $O_n$ the $n\times n$ identity and zero matrices, respectively. 
The $n_1\times n_2$ zero matrix is represented by $O_{n_1\times n_2}$.
Finally, we will use the notation $0_n$ to represent the column vector
$O_{n \times 1}$. 
Matrix norms in both  $\RR^{n_1 \times n_2}$ and $\CC^{n_1\times n_2}$ are the ones induced from the corresponding 
vector norms.
That is to say, for a $n_1 \times n_2$ matrix $M$, we have
\[
	|M| = \max_{i= 1,\dots,n_1} \sum_{j= 1,\dots, n_2} |M_{i,j}|.
\]
In particular, if $v$ is a 
$n_2$-dimensional vector, 
$|M v|\leq |M| |v|$. Moreover, $M^\top$ denotes the transpose of the matrix $M$, so that 
\[
 	|M^\top|= \max_{j= 1,\dots, n_2} \sum_{i= 1,\dots,n_1} |M_{i,j}|.
\]

Given an analytic function $f: \U\subset \CC^m \to \CC$, defined in an open set $\U$, 
the action of the $r$-order derivative of $f$ at a point $x\in \U$ on  
a collection of (column) vectors $v_1,\dots, v_r\in \CC^m$, 
with $v_k= (v_{1k}, \dots, v_{mk})$, is
\[
	\Dif^r f(x) [v_1,\dots,v_r] = \sum_{\ell_1,\dots,\ell_r} \frac{\partial^r f}{\partial x_{\ell_1}\dots\partial x_{\ell_r}}(x)\ v_{\ell_1 1} \cdots v_{\ell_r r}, 
\]
where the indices $\ell_1,\dots,\ell_r$ run from $1$ to $m$.

The construction is extended
to vector and matrix valued functions as 
follows: given
a matrix valued function $M:\U\subset\CC^m\to \CC^{n_1\times n_2}$
(whose components $M_{i,j}$ are analytic functions), a point $x\in \U$,
and a collection of (column) vectors $v_1,\dots, v_r\in \CC^m$, 
we obtain a $n_1\times n_2$  matrix  $\Dif^r M(x) [v_1,\dots,v_r]$ 
such that 
\[
\left(\Dif^r M(x) [v_1,\dots,v_r]\right)_{i,j} = \Dif^r M_{i,j}(x) [v_1,\dots,v_r].
\] 
For $r=1$, we will often write $\Dif M(x) [v]= \Dif^1 M(x) [v]$ for $v\in \CC^m$.

Notice that, given a function $f: \U \subset \CC^m \to \CC^n\simeq \CC^{n\times 1}$
we can think of  $\Dif f$ as a matrix function 
$\Dif f: \U \to \CC^{n\times m}$. Hence, $\Dif^1 f(x) [v]= \Dif f(x) v$ for $v\in\CC^m$.
Therefore,
we can apply the transpose to obtain a matrix function $(\Dif f)^\top$, 
which acts on $n$-dimensional vectors, while $\Dif f^\top=\Dif (f^\top)$ acts on $m$-dimensional vectors.
Hence, according to the above notation, the operators $\Dif$ and $(\cdot)^\top$ do not
commute. Therefore, in order to avoid confusion, we must pay attention to the use of parenthesis. 

A function $u : \RR^d \to \RR$ is 1-periodic if $u(\theta+e)=u(\theta)$ for all $\theta \in \RR^d$ and $e \in \ZZ^d$.
Abusing notation, we write
$u:\TT^d \to \RR$, where $\TT^d = \RR^d/\ZZ^d$ is the $d$-dimensional standard torus. 
Analogously, for $\rho>0$, a function $u:\RR^d_\rho\to \CC$
is 1-periodic if $u(\theta+e)=u(\theta)$ for all $\theta \in \RR^d_\rho$ and $e \in \ZZ^d$. 
We also abuse notation and write  $u:\TT^d_\rho \to \RR$, where 
$\TT^{d}_{\rho}=  \{\theta\in \CC^d / \ZZ^d : |\im\ \theta| < \rho\}$ is the complex strip of  $\TT^d$ of width
$\rho>0$. 
We will also write the Fourier expansion of a periodic function as
\[
u(\theta)=\sum_{k \in \ZZ^d} \hat u_k \ee^{2\pi \ii k \cdot \theta}, \qquad \hat u_k =
\int_{\TT^d} u(\theta) \ee^{-2 \pi \ii k \cdot \theta} \dif \theta\,,
\]
and introduce the notation $\aver{u}:=\hat u_0$ for the average. 
The notation in the paragraph is extended
to $n_1 \times n_2$ matrix valued periodic functions 
$M: \TT^{d}_{\rho}\to\CC^{n_1\times n_2}$, for which 
 $\hat M_k \in \CC^{n_1 \times n_2}$ denotes the Fourier coefficient of index $k\in \ZZ^d$.

\subsection{Hamiltonian systems and invariant tori}\label{ssec:inv:tor}

In this paper we consider an open set $\mani$ of $\RR^{2n}$ endowed
with a symplectic form $\sform$, 
that is, a closed ($\dif \sform =0$) non-degenerate differential 2-form on $\mani$. 
We will assume that $\sform$ is exact ($\sform=\dif \aform$ for certain 1-form $\aform$ called
\emph{action form}), so
$\mani$ is endowed with an exact symplectic structure.
The matrix representations of  $\aform$ and $\sform$ are given by the matrix valued functions
\[
\begin{array}{rcl}
a: \mani & \longrightarrow & \RR^{2n}\\
z & \longmapsto & a(z)
\,,
\end{array}
\]
and
\[
\begin{array}{rcl}
\Omega: \mani & \longrightarrow & \RR^{2n \times 2n}\\
z & \longmapsto & \Omega(z)
=(\Dif a(z))^\top-\Dif a(z)
 \,,
\end{array}
\]
respectively.
The non-degeneracy of $\sform$ is equivalent to $\det \Omega(z) \neq 0$ for all $z\in\mani$.

\begin{remark}\label{rem:proto}
The prototype example of symplectic structure is
the \emph{standard symplectic structure}
on $\mani \subset \RR^{2n}$:
$\sform_0= \sum_{i= 1}^n  {\rm d} z_{n+i}\wedge {\rm d}z_i$.
An action form for $\sform_0$ is ${\aform_0}= \sum_{i= 1}^n z_{n+i}\
{\rm d} z_i$.
The matrix representations of $\aform_0$ and $\sform_0$
are, respectively,
\[
a_0(z)= \begin{pmatrix} O_n & I_n \\ O_n & O_n\end{pmatrix} z\,, \qquad
\Omega_0= \begin{pmatrix} O_n & -I_n \\ I_n & O_n \end{pmatrix}\,.
\]
Another usual action form on $\mani\subset \RR^{2n}$ is ${\aform_0}=\tfrac{1}{2}\sum_{i= 1}^n (z_{n+i}\
{\rm d} z_i-z_i \ \dif z_{n+i})$, which is represented as
\[
a_0(z)= \frac{1}{2} \begin{pmatrix} O_n & I_n \\ -I_n & O_n\end{pmatrix} z\,, 
\]
in coordinates. 
\end{remark}

We say that a vector field
$X : \mani\to \RR^{2n}$ is \emph{symplectic} (or \emph{locally-Hamiltonian}) if $L_X \sform=0$ where $L_X$ stands for the
Lie derivative with respect to $X$. Using Cartan's magic formula, and the
fact that $\dif \sform=0$, it turns out that $X$ is symplectic if and
only if 
$i_X \sform$ is closed. 
We say that $X$ is \emph{exact
symplectic} (or \emph{Hamiltonian}) if the contraction $i_X \sform$ is 
exact, i.e., if there
exists a function $\H : \mani\to \RR$ (globally defined) such that $i_X \sform =
-\dif \H$.
In coordinates, an exact symplectic vector
field satisfies 
\[
\Omega(z) X(z)= (\Dif \H(z))^\top\,,
\qquad
\mbox{i.e.,}
\qquad
X(z)= \Omega(z)^{-1} (\Dif \H(z))^\top\,.
\]
Hence, we will use the notation $X=X_\H$.

The Poisson bracket of two functions $f$, $g$ is given by
$\{f,g\} = 
-\sform(X_f,X_g)$.
In coordinates, 
\[
\{f,g\} (z) 
       = \Dif f(z) \Omega(z)^{-1} (\Dif g(z))^\top.
\]
Then, if $\varphi_t$ is the flow of $X_\H$,
it follows that 
\[
\frac{\dif}{\dif t} (f \circ \varphi_t)  = \{f,\H\} \circ \varphi_t, 
\]
and so, $f$ is a conserved quantity if $\{f,\H\}=0$.

In this context, given a Hamiltonian vector field $X_\H$ on $\mani$ and a frequency
vector $\omega \in \RR^d$, with $2\leq d\leq n$, 
we are interested in finding
a parameterization $K : \TT^d \rightarrow \mani$ satisfying
\begin{equation}\label{eq:inv:fv}
X_\H (K(\theta)) = \Dif K(\theta) \omega \,.
\end{equation}
This means that the $d$-dimensional manifold $\torus=K(\TT^d)$ is invariant and the
internal dynamics is given by the constant vector field $\omega$. 
For obvious reasons, equation \eqref{eq:inv:fv} is called \emph{invariance equation for $K$}.
Therefore, given a parameterization $K : \TT^d \rightarrow \mani$ 
and a frequency vector $\omega\in\RR^d$, 
the \emph{error of invariance} is the periodic function
$E:\TT^d \rightarrow \RR^{2n}$ given by
\begin{equation}\label{eq:inv:err}
E(\theta) := X_\H (K(\theta)) - \Dif K(\theta) \omega \,.
\end{equation}
Roughly speaking, if we have a good enough
approximation of a $d$-dimensional invariant torus $\torus$
(that is, the error \eqref{eq:inv:err} is small enough in certain norm),
then one is interested in obtaining a true invariant torus, close to
$\torus$, satisfying \eqref{eq:inv:fv}.

\begin{remark}\label{rem:topo}
Equation \eqref{eq:inv:fv} is the infinitesimal version
of the equation
\[
\varphi_t(K(\theta)) = K(\theta+\omega t)\,,
\]
where $\varphi_t$ is the flow of $X_\H$. Accordingly, we can
study the invariance of $\torus$ using a discrete version of a KAM theorem
for symplectic maps. We refer the reader to \cite{GonzalezJLV05,
FiguerasHL17,
HaroCFLM16}
for such a-posteriori theorems, quantitative estimates, and
applications. However, notice that obtaining the required
estimates for the flow $\varphi_t$ demands to integrate the
equations of the motion up to second order variational equations.
\end{remark}

If $\omega \in \RR^d$ is \emph{nonresonant} (i.e.
if $\omega\cdot k\neq 0$ for every $k \in \ZZ^d\backslash\{0\}$)
then  $z(t)=K(\al+\omega t)$
is a quasi-periodic solution of $X_\H$ for every $\al\in \TT^d$,
and $\al$ is called the initial phase of the parameterization.
It is well known that quasi-periodicity implies additional
geometric properties of the torus (see \cite{BroerHS96,
Moser66c}).
In particular, that the torus is \emph{isotropic}.
This means that the pullback $K^*\sform$ of the symplectic form on the torus $\torus$
vanishes. In matrix notation, the representation of $K^*\sform$
at a point $\theta \in \TT^d$ is
\begin{equation} 
\label{def-OK}
\Omega_K(\theta) 
=  (\Dif K(\theta))^\top \Omega(K(\theta)) \ \Dif K(\theta)\,,
\end{equation}
and so, $\torus$ is isotropic if
$\Omega_K(\theta)=O_{d}$, $\forall \theta \in \TT^d$. If $d=n$ then $\torus$ is \emph{Lagrangian}.
Moreover, quasi-periodicity implies that
\[
\H(K(\theta)) = \aver{\H \circ K}\,, \qquad \forall \theta \in \TT^d\,,
\]
which means that the torus is contained in an energy level of the Hamiltonian.

\begin{remark}
Other topologies can be considered for the ambient manifold.
For example, we may have some information (e.g. from
normal form analysis) that allows us to construct 
tubular coordinates around a torus. In this situation,
it is interesting to look for an invariant torus inside
an ambient manifold of the form $\mani \subset \TT^d \times U$, 
with $U \subset \RR^{2n-d}$.
Notice that both $a(z)$ and $\H(z)$ are 
1-periodic in the first
$d$-variables, and, since the Poisson bracket preserves this property,
also is $X_{\H}$.
Assume that we are interested in obtaining an invariant torus
that preserves the topology of the manifold (typically
called \emph{primary torus}). Then we must choose a parameterization
$K : \TT^d \rightarrow \mani$ such that $K(\theta)-(\theta,0)$ is $1$-periodic
(that is, $K$ is \emph{homotopic to the
zero-section}). Therefore, it turns out that the error function \eqref{eq:inv:err}
is also 1-periodic, so the construction makes sense.
All expressions and formulas presented in this paper remain
valid, and one only needs to take into account that the elements
of $K$ and $\Dif K$ contain an additional term that comes from the
topology. 
We refer to \cite{HaroCFLM16} for a detailed discussion of this case.
The case of intermediate topologies 
$\mani \subset \TT^m \times U$ with $U \subset \RR^{2n-m}$
is similar.
\end{remark}

\subsection{Conserved quantities and families of invariant tori}\label{ssec:conserved}

We will assume that the vector field $X_\H$ has $n-d$ first integrals
in involution $p_1,\dots,p_{n-d}:\mani\to\RR$, that is to say:
\begin{equation}
\label{eq:Poisson:H:p}
\{\H,p_j\}= 0\,,
\quad 1\leq j\leq n-d
\,,
\end{equation}
and
\begin{equation}\label{eq:Poisson:p:p}
\{p_i,p_j\}= 0\,,
\quad 1\leq i,j\leq n-d
\,.
\end{equation}
Consequently, the 
Lie brackets of the corresponding Hamiltonian vector fields vanish and
we have:
\begin{equation}\label{eq:comm:H:p}
\Dif X_\H(z) X_{p_j}(z)= \Dif X_{p_j}(z) X_\H(z)\,,
\quad 1\leq j\leq n-d
\,,
\end{equation}
and
\begin{equation}\label{eq:comm:p:p}
\Dif X_{p_i}(z) X_{p_j}(z)= \Dif X_{p_j}(z) X_{p_i}(z)\,,
\quad 1\leq i,j\leq n-d
\,.
\end{equation}
We will encode the $n-d$ first integrals in an only function $p:\mani\to \RR^{n-d}$, so that 
the involution conditions are rephrased as
\[
\Dif \H(z) \ \Omega(z)^{-1} (\Dif p(z))^\top = 0_{n-d}^\top \,,
\]
and
\[
\Dif p(z) \ \Omega(z)^{-1} (\Dif p(z))^\top = O_{n-d}\,.
\]
Moreover, the corresponding $n-d$ Hamiltonian vector fields are the columns of the matrix function 
$X_p:\mani \to\RR^{2n \times (n-d)}$, with $(X_p)_{i,j}= (X_{p_j})_i$, and 
\[
	X_p(z)= \Omega(z)^{-1} (\Dif p(z))^\top\,.
\]
The commuting conditions are 
\begin{equation}\label{eq:comm:H:p:bis}
\Dif X_\H(z) X_{p}(z)= \Dif X_{p}(z) [X_\H(z)]
\,,
\end{equation}
and
\begin{equation}\label{eq:comm:p:p:bis}
\Dif X_{p_i}(z) X_{p}(z)= \Dif X_{p}(z) [X_{p_i}(z)]\,,
\quad 1\leq i\leq n-d
\,.
\end{equation}

The above setting implies that $p$ generates a $(n-d)$-parameter family of 
local symplectomorphisms. In particular, we introduce
\begin{equation}\label{eq:loc:action}
\Phi_s = \varphi_{s_1}^{1} \circ \cdots \circ \varphi_{s_{n-d}}^{n-d}
\end{equation}
where $s = (s_1,\ldots,s_{n-d})$ belongs to an open
neighborhood of $0$ 
in $\RR^{n-d}$,
and $\varphi_{s_i}^i$ is the flow of the Hamiltonian
vector field $X_{p_i}$.
Notice that this is a local group action of $\RR^{n-d}$ and,
by the commutativity of the flows in \eqref{eq:comm:p:p}, we have
\[
\frac{\partial \Phi_s}{\partial s_i} = X_{p_i} \circ \Phi_s\,.
\]
If the vector fields are linearly independent, 
then
the local group 
action \eqref{eq:loc:action} defines a family of local diffeomorphisms $s \mapsto \Phi_s$
which commutes with the flow of $X_\H$:
\[
\Phi_s \circ \varphi_t = \varphi_t \circ \Phi_s\,.
\]
The map $\Phi_s$ is usually called the \emph{continuous family of symmetries} of $X_\H$.
A consequence of this is the following: if $\torus=K(\TT^d)$ is invariant for
$X_\H$, with frequency vector $\omega$, then
$\torus_s = \Phi_s(\torus)$ is also invariant
for $X_\H$ with the same frequency vector:
\[
\varphi_t(\Phi_s(K(\theta))) = \Phi_s (\varphi_t(K(\theta))) =
\Phi_s(K(\theta+\omega t))\,,
\] 
and so, $\Phi_s \circ K$ is a parameterization of $\torus_s$. 

An important observation is that all invariant tori of the
family are contained in the submanifold
\[
\{z \in \mani \,:\, \H(z)=\H_0\,,\, p(z)=p_0\}\,.
\]
Hence, once the frequency $\omega$ of the torus has been fixed, we cannot
fix the values of $\H_0$ or $p_0$. 
This case is considered in Section~\ref{ssec:theo}, referred to as the 
ordinary case.
If we are interested in
obtaining an invariant torus on a target energy level $\H_0$, 
then we fix the direction of the frequency vector $\omega$ but adjust
its modulus (iso-energetic case). Similarly,
the same kind of adjustment of $\omega$ can be made if we are interested in
obtaining an invariant torus having a prefixed value
of one of the components of $p$ (iso-momentum case).
%The previous arguments can be strived for considering a particular first integral  $c: \mani \to \RR$ 
%that is in involution with $\H$ and $p$ (for instance, if $c$ is a function of $h$ and $p$).
The previous scenarios can be generalized by considering
any first integral $c : \mani \rightarrow \RR$
that commutes with $\H$ and $p=(p_1,\ldots,p_{n-d})$, that is,
\[
\Dif c(z) X_\H(z)=0\,,
\qquad
\Dif c(z) X_p(z)= 0_{n-d}^\top\,.
\]
For example, we can thing that $c$ is a function of $h$ and $p$
given by $c(z)=f_{c}(\H(z),p(z))$, where
$f_{c}:\RR \times \RR^{n-d} \rightarrow \RR$ is known
(we have the particular cases
$c(z)=h(z)$ and $c(z)=p_j(z)$
with $j \in \{1,\ldots,n-d\}$).
In Section \ref{ssec:theo:iso} we will establish sufficient
conditions to obtain an invariant torus having a prefixed value of the
target conserved quantity $c$ (generalized iso-energetic case).
We emphasize that selecting simultaneously the values of several conserved quantities is not
possible in general, but makes sense for Cantor sets of frequencies.

\subsection{Linearized dynamics and reducibility}\label{ssec:red:lin:eq}

In this section we describe the geometric construction of
a suitable symplectic frame attached to an invariant torus
$\torus$ of a Hamiltonian
system $X_\H$ with conserved quantities $p : \mani \rightarrow \RR^{n-d}$.
Indeed, given a parameterization $K:\TT^d \rightarrow \mani$ satisfying
\begin{equation}\label{eq:inv0}
X_\H(K(\theta))=\Dif K(\theta) \omega\,,
\end{equation}
and given any $m$-dimensional vector subbundle
parameterized by $V : \TT^d \rightarrow \RR^{2n \times m}$, with $1\leq m \leq 2n$,
we introduce the operator
\begin{equation}\label{eq:Loper}
\Loper{V}(\theta) := \Dif X(K(\theta)) V(\theta) - 
\Dif V(\theta) [\omega]\,,
\end{equation}
which corresponds to the infinitesimal displacement
of $V$,
and we say that a bundle is invariant under the linearized equations
if $\Loper{V}(\theta)=O_{2n \times m}$ for every $\theta \in \TT^d$.

We consider the map
$L:\TT^d \rightarrow \RR^{2 n\times n}$ given by
\[
L(\theta) =
\begin{pmatrix} \Dif K(\theta) & X_p(K(\theta)) \end{pmatrix}\,,
\]
and we will assume that $\mathrm{rank}\, L(\theta)=n$ for every $\theta \in \TT^d$.
Then, it turns out that $L(\theta)$ satisfies
\begin{equation}\label{eq:inv1}
\Loper{L}(\theta) = O_{2n \times n}\,, \qquad \forall \theta \in \TT^d\,.
\end{equation}
This invariance follows from two observations. Firstly, taking
derivatives at both sides of \eqref{eq:inv0}, we have
\[
\Dif X_\H(K(\theta)) \Dif K(\theta) = \Dif (\Dif K(\theta)) [\omega]
\,,
\]
and secondly, from the commutation rule \eqref{eq:comm:H:p}, we have:
\begin{align*}
\Dif (X_p(K(\theta))) [\omega]
  = {} & \Dif X_p(K(\theta)) [\Dif K(\theta) \omega] \\
= {} & \Dif X_p(K(\theta)) [X_\H(K(\theta))] \\
= {} & \Dif X_\H(K(\theta)) X_p(K(\theta))\,.
\end{align*}
Then, the property in \eqref{eq:inv1} follows putting together both expressions.

By similar geometric properties
(detailed computations will be presented in Section \ref{ssec:symp}),
it turns out that the subspace
$L(\theta)$ is \emph{Lagrangian}, i.e.,
we have
\[
L(\theta)^\top \Omega(K(\theta)) L(\theta) = O_{2n \times n}
\]
for every $\theta \in \TT^d$. 
Then one can use the geometric structure of the problem
to complement the above frame, thus obtaining linear coordinates on the
full tangent bundle $\Tan \mani$ that
express $\Dif X_\H \circ K$ in a simple way.
This is one of the main ingredients of recent KAM theorems
presented in different contexts and structures (see e.g. 
\cite{CallejaCLa,GonzalezJLV05,
FontichLS09,
GonzalezHL13,
LuqueV11}).
The constructions have been summarized in \cite{HaroCFLM16}
using a common framework that unifies the previous works
and emphasizes the role of the symplectic properties.
We briefly summarize this framework in Section \ref{ssec:sym:frame}.

Hence, we will obtain a map $N: \TT^d \rightarrow \RR^{2n\times n}$
such that the juxtaposed matrix
\[
P(\theta) =
\begin{pmatrix} L(\theta)  & N(\theta) \end{pmatrix}\,,
\]
satisfies $\mathrm{rank}\, P(\theta)=2n$ for every $\theta \in \TT^d$
and also
\begin{equation}\label{eq:Psym}
P(\theta)^\top \Omega(K(\theta)) P(\theta) = \Omega_0\,.
\end{equation}
In this case, we say that $P:\TT^d \rightarrow \RR^{2n\times 2n}$ is
a symplectic frame. The use of these linear coordinates on 
$\Tan_{\torus} \mani$
has several advantages. Among them, it produces a natural and geometrically
meaningful non-degeneracy condition (twist condition) in the KAM theorem,
it simplifies certain computations substantially ($P^{-1}$ can be computed
directly), but most importantly, it reduces the linearized equation
to triangular form as follows:
\[
\Dif X_\H(K(\theta)) P(\theta)
-\Dif P(\theta) [\omega]
= P(\theta) \Lambda(\theta)\,,
\]
with
\[
\Lambda(\theta)
= \begin{pmatrix}
O_n &  T(\theta) \\ 
O_n  & O_n
\end{pmatrix}
\]
and
\begin{align*}
T(\theta) = {} & N(\theta)^\top \Omega(K(\theta))
\left(
\Dif X_\H(K(\theta))
N(\theta)
-\Dif N(\theta) [\omega]
\right) \\
= {} & N(\theta)^\top \Omega(K(\theta)) \Loper{N}(\theta)\,.
\end{align*}
The matrix $T$ is usually called the \emph{torsion matrix} and plays the
role of Kolmogorov's non-degeneracy condition. 

\begin{remark}
The torsion measures the symplectic area determined by the normal bundle and its infinitesimal 
displacement.
Notice that, in the present paper, the torsion involves geometrical
and dynamical properties of both the torus and the first integrals,
and in fact, of the family of $d$-dimensional invariant tori.
\end{remark}

The above setting
allows us a approximate the solutions 
of the linearized equations around the torus
by the solutions of a triangular system which is easier to handle.
The fundamental idea is the following fact:
if $\torus$ is approximately invariant, the above geometrical
properties are still satisfied, modulo some error functions which
can be controlled in terms of the error of invariance. 
The main ingredient is the fact that (under certain assumptions) 
the frame $\theta \mapsto L(\theta)$, 
associated to a $(n-d)$-parameter family of approximately invariant tori, 
is also approximately Lagrangian. 
Hence, the linear dynamics around the torus is approximately reducible. 
This is enough to perform a quadratic scheme to correct the initial
approximation. This will be discussed in Section \ref{sec:lemmas}.

\subsection{Construction of a geometrically adapted frame}\label{ssec:sym:frame}

In this section we deal with the construction of a symplectic
frame on the bundle $\Tan_\torus\mani$ by complementing the column vectors
of a map $L:\TT^d \rightarrow \RR^{2n \times n}$ that parameterizes a Lagrangian subbundle.
To this end, we assume that we have a map $N^0:\TT^d \rightarrow \RR^{2n\times n}$
such that
\begin{equation}\label{eq:cond:CaseI}
\mathrm{rank} \begin{pmatrix} L(\theta) & N^0(\theta) \end{pmatrix} = 2n
\quad
\Leftrightarrow
\quad
\det (L(\theta)^\top \Omega(K(\theta)) N^0(\theta)) \neq 0\,,
\end{equation}
for every $\theta \in \TT^d$. Then, we complement the Lagrangian subspace
generated by $L(\theta)$ by means of a map $N:\TT^d \rightarrow \RR^{2n \times n}$
given by
\[
N(\theta) = L(\theta)A(\theta) + N^0(\theta) B(\theta)\,,
\]
where
\[
B(\theta)=-(L(\theta)^\top \Omega(K(\theta)) N^0(\theta))^{-1}
\]
and $A(\theta)$ is a solution of
\[
A(\theta)-A(\theta)^\top = -B(\theta)^\top N^0(\theta)^\top \Omega(K(\theta)) N^0(\theta) B(\theta).
\]
The solution of this equation is given by
\[
A(\theta)=-\frac{1}{2}(B(\theta)^\top N^0(\theta)^\top \Omega(K(\theta)) N^0(\theta) B(\theta))
\]
modulo the addition of any symmetric matrix.
A direct computation shows that the juxtaposed matrix $P(\theta)=(L(\theta)~N(\theta))$ satisfies \eqref{eq:Psym}.

The map $N^0$ can be obtained by directly complementing the tangent vectors of the
initial (approximately invariant) torus, and after that, be fixed along the iterative
procedure. This is called \textbf{Case I} in chapter 4 of \cite{HaroCFLM16}. 
It has the advantage of being more
general and flexible, but the it requires some extra work to obtain optimal quantitative
estimates (we refer the reader to \cite{FiguerasHL17}, where additional geometric
properties of $N^0$ are controlled). 

A natural way to construct a map $N^0$ systematically is by using additional
geometric information. 
In this paper, we assume that $\mani$ is also endowed with a Riemannian
metric $\gform$, represented in coordinates as the  positive-definite symmetric matrix
valued function $G:\mani\to\RR^{2n\times 2n}$.
Then, we define the linear isomorphism $\J:  \Tan \mani
\to \Tan\mani$ such that 
$\sform_z(\J_z u,v)=\gform_z(u,v)$,
$\forall u,v \in T_z \mani$. Observe also that $\J$ is antisymmetric with
respect to $\gform$, that is, 
$\gform_z(u,\J_z v)=-\gform_z(\J_zu,v)$,
$\forall u,v \in \Tan_z \mani$. 
The matrix representation of $\J$ is given by the matrix valued function 
$J:\mani\to \RR^{2n\times 2n}$.
Then, we have
\begin{equation}\label{eq:prop:struc1}
\Omega^\top = -\Omega\,, \qquad G^\top = G\,, \qquad J^\top \Omega=G\,,
\end{equation}
and we introduce the matrix valued function 
$\tilde \Omega:\mani\to\RR^{2n\times 2n}$, as
\[
\tilde \Omega := J^{\top} \Omega J =  G J\,,
\]
for the representation of the symplectic form in the
frame given by $J$.

Then, we choose the map $N^0$ as follows
\begin{equation}\label{eq:map:N0:CII}
N^0(\theta):= J(K(\theta)) L(\theta)
\end{equation}
and condition \eqref{eq:cond:CaseI} is equivalent to
\[
\det (L(\theta)^\top G(K(\theta)) L(\theta)) \neq 0\,\qquad \forall \theta \in \TT^d\,.
\]
Moreover, the matrices $A(\theta)$ and $B(\theta)$ are expressed as follows
\begin{align*}
A(\theta)={}& -\frac{1}{2} (B(\theta)^\top L(\theta)^\top \tilde \Omega(K(\theta)) L(\theta) B(\theta)) \,, \\
B(\theta)={}& (L(\theta)^\top G(K(\theta)) L(\theta))^{-1}\,.
\end{align*}
This is called \textbf{Case II} in \cite{HaroCFLM16}.

\begin{remark}
We want to point out that the above construction differs slightly from
the discussion in chapter 4 of \cite{HaroCFLM16},
where the linear
isomorphism $\J$ is defined as
$\sform_z(u,v)=\gform_z(u,\J_zv)$. This choice results in
the map $$N^0(\theta)= - J(K(\theta))^{-1} L(\theta)$$ rather than \eqref{eq:map:N0:CII}.
Both constructions are equivalent, but the construction
described here is geometrically more natural and produces better
quantitative estimates.
\end{remark}

There is also the special case where the isomorphism $\J$ is anti-involutive.
that is, $\J^2 = -\I$. Then, we say that the triple $(\sform,\gform,\J)$ is
compatible and that $\J$ endows $\mani$ with a complex structure. 
This is called \textbf{Case III} in chapter 4 of \cite{HaroCFLM16}.
In coordinates,
we have the following properties
\begin{equation*}
J^2= -I_{2n}, \qquad 
\Omega= J^\top \Omega J, \qquad G=  J^\top G J. 
\end{equation*}
In this situation, we have
\[
N^0(\theta) = J(K(\theta)) L(\theta)\,,
\qquad
A(\theta) = O_n\,,
\qquad
B(\theta) =  (L(\theta)^\top G(K(\theta)) L(\theta))^{-1}\,.
\]

It is important to notice that the above constructions
lead to different quantitative estimates in the KAM theorem.
Selecting the best option depends
on the particular problem under consideration.
Since \textbf{Case I} has been fully reported in the
references \cite{FiguerasHL17,HaroCFLM16}, in this paper we will focus
in obtaining sharp quantitative estimates for \textbf{Case II} and \textbf{Case III}.
Hence, we cover a gap in the literature that 
could be valuable in future studies.

\subsection{Univocal determination of an invariant torus of the family}
\label{ssec:univocal}

In this section we describe suitable strategies to avoid the
undeterminations observed in Sections \ref{ssec:inv:tor}
and \ref{ssec:conserved}. Let us recapitulate them:

\begin{itemize}
\item If $\torus = K(\TT^d)$ is a $d$-dimensional invariant torus of $X_\H$ of frequency $\omega$,
then $K^\al(\theta)=K(\theta+\al)$ also parameterizes $\torus$ for every $\al\in\TT^d$.
\item If $\torus = K(\TT^d)$ is a $d$-dimensional invariant torus of $X_\H$ of frequency $\omega$,
and we introduce $K_s = \Phi_s \circ K$ using the family of symmetries,
then $\torus_s= K_s(\TT^d)$ is also an invariant torus of frequency $\omega$ 
for every $s \in \RR^{n-d}$ in the domain of definition.
\end{itemize}

The first indetermination corresponds to the choice of the parameterization
of the invariant object, and it can be avoided simply by fixing an initial phase
of the torus. To this end, we consider a $(2n-d)$-dimensional manifold
given by the preimage of a map $Z: \RR^{2n} \rightarrow \RR^d$ and select the value of $\al$
such that
\[
Z(K^\al(0))=Z_0\,,
\]
for some $Z_0 \in \RR^d$. Notice that $Z$ must be selected in such a way
that the transversality condition
\[
\det (\Dif Z(K^\al(0)) \Dif K^\al(0)) \neq 0\,,
\]
holds in an open set of values $\al \in \RR^d$.

\begin{remark}\label{eq:fix:top}
If $\mani = \TT^d \times U$ and we are considering a non-contractible
invariant tori of the form
$K(\theta)=(K^x(\theta),K^y(\theta)) \in \mani$,
then a typical way to determine the phase univocally is to ask
the following average condition
\[
\aver{K^x-\id} = 0_d\,.
\]
In this case, we select $\al=-\aver{K^x-\id}$. Another
possibility, in the spirit described above, is to select
the transversal plane given by $Z(z)=(z_1,\ldots,z_d)$.
\end{remark}

The second indetermination corresponds to a choice of a given invariant
torus inside the $(n-d)$-parameter family described in Section \ref{ssec:conserved}.
In this case, 
we need to fix additional
$(n-d)$ conditions in order to define univocally a single torus of the family.
For example, we may assume that there is a map $q : \mani \rightarrow \RR^{n-d}$ satisfying
\begin{equation}\label{eq:cond:pq}
\Dif q(z) X_p(z) 
= \Dif q(z) \ \Omega(z)^{-1} (\Dif p(z))^\top = I_{n-d}\,,
\end{equation}
which means that $\{q_i,p_j\}=\delta_{ij}$.
For obvious reasons, $p$ and $q$ are referred to as the \emph{generalized momentum} and
the \emph{generalized conjugated position}, respectively. Then, we can
determine univocally a torus in the family by asking for the extra
equations 
\[
q \circ K (0) = q_0 \in \RR^{n-d}\,.
\]
We can recover the full family by considering the map
\[
s \longmapsto q \circ \Phi_s \circ K = q \circ K + s\,,
\]
where we used that
\begin{align*}
\frac{\pd}{\pd s_i} (q\circ \Phi_s \circ K)
= {} & 
(\Dif q \circ \Phi_s \circ K)
\pd_{s_i} (\Phi_s \circ K)
 \\
= {} & 
(\Dif q \circ \Phi_s \circ K)
(X_{p_i} \circ \Phi_s \circ K)
= e_i
\end{align*}

\begin{remark}
The above construction can be readily generalized asking the map $q$ to satisfy
\[
\det (\Dif q(z) X_p(z))\neq 0
\]
instead of \eqref{eq:cond:pq}.
\end{remark}

It is clear that both indeterminations can be addressed simultaneously
by fixing $n$ conditions. This can be done for example by asking for
a transversality condition on the Lagrangian frame $\theta \mapsto L(\theta)$
described in Section \ref{ssec:red:lin:eq}
at a given point. To this end, we denote
\[
K_{\al,s}(\theta)=\Phi_s(K(\theta+\al))\,,
\qquad \al \in \RR^d\,,
\qquad s \in \RR^{n-d}\,,
\]
we consider a map $Q:\mani \rightarrow \RR^n$,
and we ask for the condition
\[
Q(K_{\al,s}(0)) = Q_0
\]
for a given point $Q_0\in \RR^n$. 
It this situation, the transversality condition reads
\[
\det \left(\Dif Q(K_{\al,s}(0))L_{\al,s}(0) 
\right)
\neq 0
\]
where
\[
\theta \mapsto L_{\al,s}(\theta) =
\begin{pmatrix}
\Dif K_{\al,s}(\theta) & X_p(K_{\al,s}(\theta))
\end{pmatrix}
\]
is the Lagrangian frame associated with the torus $\torus_{\al,s}$. For
example, a natural choice would be
\[
Q(z)=\begin{pmatrix}
Z(z) \\
q(z)
\end{pmatrix}\,,
\]
where $Z$ is selected to fix the phase of the parameterizations and
$q$ are generalized positions associated with $p$.
Depending on the topology of the ambient space, we may
consider other choices (see Remark \ref{eq:fix:top}).

\section{A-posteriori KAM theory for partially integrable Hamiltonian systems}

In this section, we present two a-posteriori KAM theorems for $d$-dimensional 
quasi-periodic invariant tori
in Hamiltonian systems with $n$ degrees-of-freedom that have $n-d$ additional first integrals
in involution. To this end, we will assume that the frequency vector $\omega$
satisfies Diophantine conditions. Specifically, we denote the set of Diophantine vectors as
\begin{equation}\label{eq:def:Dioph}
\Dioph{\gamma}{\tau} =
\left\{
\omega \in \RR^d \, : \,
\abs{k \cdot \omega} \geq  \frac{\gamma}{|k|_1^{\tau}}
\,, 
\forall k\in\ZZ^d\backslash\{0\}
\,,
|k|_1 = \sum_{i= 1}^d |k_i|
\right\}\,,
\end{equation}
for certain $\gamma >0$ and $\tau \geq d-1$.

In Section \ref{ssec-anal-prelims} we set some basic notation regarding Banach spaces and norms 
of analytic functions.
In Section \ref{ssec:theo} we present the statement of a KAM theorem for existence (and
persistence) of $d$-dimensional invariant tori having fixed 
frequency vector $\omega \in \Dioph{\gamma}{\tau}$.
This corresponds to the so-called ordinary (\`a la Kolmogorov) KAM theorem.
In Section \ref{ssec:theo:iso} we present and adapted version of the theorem
that generalizes the iso-energetic approach. 
Section~\ref{sec:lemmas} is devoted to the 
control of approximate geometric properties, the anteroom of the proofs
of the main theorems in Sections \ref{sec:proof:KAM} and \ref{sec:proof:KAM:iso}.
We will pay special attention
in providing explicit and rather optimal bounds,
with an eye in the application of the theorems and in computer assisted proofs.
The constants have been collected in a series of tables in Appendix~\ref{ssec:consts}.

\subsection{Analytic functions and norms}\label{ssec-anal-prelims}

In this paper we work with 
real analytic functions defined in
complex neighborhoods of real domains.
We will consider the sup-norms of (matrix valued) analytic functions and their derivatives
(see the notation in Section \ref{ssec:basic:notation}). 
That is, for $f: \U\subset \CC^m \to \CC$, we consider 
\[
\norm{f}_\U= \sup_{x\in \U} |f(x)|,
\]
and
\[
\norm{\Dif^r f}_\U= \sum_{\ell_1,\dots,\ell_r} \Norm{\frac{\partial^r f}{\partial x_{\ell_1}\dots\partial x_{\ell_r}}}_\U,
\]
that could be infinite. 
For $M:\U \subset \CC^m\to \CC^{n_1\times n_2}$, we consider the norms  
\[
\norm{M}_\U= \max_{i= 1,\dots,n_1} \sum_{j= 1,\dots,n_2} \norm{M_{i,j}}_\U \,,
\]
\[
\norm{\Dif^r M}_\U= \max_{i= 1,\dots,n_1} \sum_{j= 1,\dots,n_2} \Norm{\Dif^r M_{i,j}}_\U\,,
\]
and we notice, of course, that the norms $\norm{M^\top}_\U$ and $\norm{\Dif^r M^\top}_\U$ 
are obtained simply by interchanging the role of the indices $i$ and $j$.

Let us remark that the above norms present Banach algebra-like properties.
For example, given $r$ analytic functions $v_1,\dots, v_r: \U\to \CC^m\simeq \CC^{m\times 1}$,
then the function 
$\Dif^r M [v_1,\dots, v_r]: \U\subset \CC^m\to \CC^{n_1\times n_2}$ defined as 
\[
\Dif^r M [v_1,\dots, v_r](x)= \Dif^r M(x) [v_1(x),\dots, v_r(x)]
\]
is also analytic, and we have
\begin{align*}
&\norm{\Dif^r M [v_1,\dots, v_r]}_\U 
\leq 
\max_{i=1,\ldots,n_1} \sum_{j=1}^{n_2} \norm{\Dif^r M_{i,j} [v_1,\ldots,v_r]}_\U \\
&\qquad \leq 
\max_{i=1,\ldots,n_1} \sum_{j=1}^{n_2} \Norm{\sum_{\ell_1,\ldots,\ell_r} \frac{\partial^r M_{i,j}}{\partial x_{\ell_1} \cdots \partial x_{\ell_r}} v_{\ell_1 1} \cdots v_{\ell_r r} }_\U \\
&\qquad \leq 
\max_{i=1,\ldots,n_1} \sum_{j=1}^{n_2} \sum_{\ell_1,\ldots,\ell_r} \Norm{\frac{\partial^r M_{i,j}}{\partial x_{\ell_1} \cdots \partial x_{\ell_r}} }_\U \max_{\ell=1,\ldots,r} \Norm{v_{\ell,1}}_\U 
\cdots \max_{\ell=1,\ldots,r} \Norm{v_{\ell,r}}_\U \\
&\qquad = \Norm{\Dif^r M}_\U \ \norm{v_1}_\U\cdots \norm{v_r}_\U\,.
\end{align*}
There is also a similar bound for the action of the transpose:
\begin{align*}
&\norm{(\Dif^r M [v_1,\dots, v_r])^\top}_\U 
\leq 
\max_{j=1,\ldots,n_2} \sum_{i=1}^{n_1} \norm{\Dif^r M_{i,j} [v_1,\ldots,v_r]}_\U \\
&\qquad \leq \Norm{\Dif^r M^\top}_\U \ \norm{v_1}_\U\cdots \norm{v_r}_\U\,.
\end{align*}
In addition, given $M_1: \U \subset \CC^m\to \CC^{n_1\times n_3}$ and $M_2: \U \subset \CC^m\to \CC^{n_3\times n_2}$,
we have 
\[
\norm{M_1 M_2}_\U \leq 
\norm{M_1}_\U \norm{M_2}_\U \, ,
\]
and 
\[
\norm{\Dif(M_1 M_2)}_\U \leq 
\norm{\Dif M_1}_\U \norm{M_2}_\U + \norm{M_1}_\U \norm{\Dif M_2}_\U  \, .
\]

The particular case of real-analytic periodic functions deserves some additional comments.
We denote by $\Anal(\TT^d_\rho)$ the Banach space of holomorphic functions
$u:\TT^d_\rho \to \CC$, that can be continuously extended to 
$\bar\TT^d_\rho$, and such that
$u(\TT^d) \subset \RR$ (real-analytic), endowed with the norm
\[
\norm{u}_\rho = \norm{u}_{\TT^d_\rho}= \max_{|\im\theta|\leq \rho} |u(\theta)|\,.
\] 
As usual in the analytic setting, we will use Cauchy estimates to
control the derivatives of a function. 
Given $u \in \Anal(\TT^d_\rho)$, with $\rho>0$, then for any $0<\delta<\rho$
the partial derivative $\pd u/\pd {x_\ell}$ belongs to $\Anal(\TT^d_{\rho-\delta})$
and we have the estimates
\[
\Norm{
\frac{\pd u}{\pd x_{\ell}}}_{\rho-\delta} \leq \frac{1}{\delta}\norm{u}_\rho, 
\qquad
\Norm{\Dif u}_{\rho-\delta} \leq \frac{d}{\delta}\norm{u}_\rho, 
\qquad
\Norm{(\Dif u)^\top}_{\rho-\delta} \leq \frac{1}{\delta}\norm{u}_\rho.
\]
The above definitions and estimates extend naturally to matrix valued function, that is,
given  $M: \TT^d_\rho \to \CC^{n_1\times n_2}$, with components in $ \Anal(\TT^d_\rho)$, we have
\[
\norm{\Dif M}_{\rho-\delta} 
=   \max_{i=1,\ldots,n_1} \sum_{j= 1,\dots, n_2}  \Norm{\Dif M_{i,j}}_{\rho-\delta} 
\leq \frac{d}{\delta} \norm{M}_{\rho}.
\]
A direct consequence is that $\norm{\Dif M^\top}_{\rho-\delta} \leq \frac{d}{\delta} \norm{M^\top}_{\rho}$.

As it was mentioned in Section \ref{ssec:basic:notation}, the
operators $\Dif$ and $(\cdot)^\top$ do not
commute. In particular,
given a real analytic vector function $w: \TT^d_\rho \to \CC^n\simeq \CC^{n\times 1}$, we have:
\[
	\norm{\Dif w}_{\rho-\delta} \leq \frac{d}{\delta} \norm{w}_\rho,\quad 
	\norm{\Dif w^\top}_{\rho-\delta} \leq \frac{d}{\delta} \norm{w^\top}_\rho \leq \frac{n d}{\delta} \norm{w}_\rho,\quad
	\norm{(\Dif w)^\top}_{\rho-\delta} \leq \frac{n}{\delta}\norm{w}_\rho.
\]

\subsection{Ordinary KAM theorem}\label{ssec:theo}

At this point, we are ready to state sufficient conditions to guarantee the existence of
a $d$-dimensional invariant torus with fixed frequency close to an approximately invariant one.
Notice that the hypotheses in Theorem \ref{theo:KAM} are tailored to be verified with a finite amount of computations.

The result is written simultaneously to \textbf{Case II} and \textbf{Case III}
(see Section \ref{ssec:sym:frame}). Estimates corresponding to \textbf{Case I} can be easily
obtained without any remarkable difficulty (see e.g. \cite{FiguerasHL17} for details).
Hence, given a parameterization of a torus $\torus=K(\TT^d)$ (not necessarily invariant)
and a tangent frame $L:\TT^d \rightarrow \RR^{2n \times n}$,
the normal frame $N:\TT^d \rightarrow \RR^{2n \times n}$ is constructed as follows:
\begin{equation}\label{eq:N}
N(\theta):= L(\theta) A(\theta) + N^0(\theta) B(\theta)\,,
\end{equation}
where
\begin{align}
N^0(\theta)={} & J(K(\theta)) L(\theta) \,, \label{eq:N0} \\
B(\theta)={}& (L(\theta)^\top G(K(\theta)) L(\theta))^{-1}\,, \label{eq:B} \\
A(\theta)={}& 
\begin{cases} -\displaystyle\frac{1}{2} (B(\theta)^\top L(\theta)^\top \tilde \Omega(K(\theta)) L(\theta) B(\theta)), & 
\text{if \textbf{Case II;}} \\
0, &  \text{if \textbf{Case III.}} 
\end{cases}
\, \label{eq:A} 
\end{align}
The torsion matrix $T:\TT^d\to\RR^{n\times n}$, given by 
   \begin{equation}\label{eq:T}
T(\theta) = N(\theta)^\top \Omega(K(\theta)) \Loper{N}(\theta)\,,
\end{equation}
where 
\begin{equation}
\Loper{N}(\theta)= \Dif X_\H (K(\theta)) N(\theta) -\Dif N (\theta)[\omega]\,,
\end{equation}
measures the infinitesimal twist of the normal bundle. With this geometric ingredients we are ready to state 
our main theorem, in the ordinary case.

\begin{theorem}[KAM theorem with first integrals]\label{theo:KAM}
Let us consider an exact symplectic structure $\sform=\dif \aform$ 
and a Riemannian metric $\gform$ on the
open set $\mani\subset\RR^{2n}$.
Let $\H$ be a Hamiltonian function, having $n-d$ first
integrals in involution $p=(p_1,\ldots,p_{n-d})$, with $2\leq d \leq n$,
and let $c$ be any first integral in involution with $(h,p)$. 
Let $K:\TT^d\to \mani$ be a parameterization of an approximately invariant torus
with frequency vector $\omega\in \RR^d$, and consider the tangent frame  $L:\TT^d\to\RR^{2n\times n}$ given by
\begin{equation}\label{eq:L:theo}
L(\theta)=
\begin{pmatrix} \Dif K(\theta) & X_p(K(\theta)) \end{pmatrix}\,.
\end{equation}
Then, we assume that the following hypotheses hold.
\begin{itemize}
\item [$H_1$]
The global objects can be analytically extended
to the complex domain $\B \subset \CC^{2n}$, and there
are constants 
that quantify the control of their analytic norms. 

For the geometric structures $\sform, \gform, \J, \tilde{\sform}= \J^*\sform$ in $\B$, 
the matrix functions
$\Omega,  G, J, \tilde \Omega: \B \rightarrow \CC^{2n \times 2n}$
satisfy:
\begin{align*}
&\norm{\Omega}_{\B} \leq \cteOmega\,,
&& \norm{\Dif \Omega}_{\B} \leq \cteDOmega\,, 
&&  \\
&  \norm{\tilde \Omega}_{\B} \leq \ctetOmega\,, 
&& \norm{\Dif \tilde\Omega}_{\B} \leq \cteDtOmega\,,
&& \norm{\Dif^2 \tilde \Omega}_{\B} \leq \cteDDtOmega\,, \\
&\norm{G}_{\B} \leq \cteG\,,
&& \norm{\Dif G}_{\B} \leq \cteDG\,, 
&& \norm{\Dif^2 G}_{\B} \leq \cteDDG\,,  \\
& \norm{J}_{\B} \leq \cteJ\,, 
&& \norm{\Dif J}_{\B} \leq \cteDJ\,,
&& \norm{\Dif^2 J}_{\B} \leq \cteDDJ\,, \\
& \norm{J^{\top}}_{\B} \leq \cteJT\,, && \norm{\Dif J^\top}_{\B} \leq \cteDJT\,.
&& 
\end{align*}

For the Hamiltonian 
$\H: \B \rightarrow \CC$ and its corresponding vector field
$X_\H: \B \rightarrow \CC^{2n}$, we have
:
\begin{align*}
&  \norm{\Dif \H}_{\B} \leq \cteDH\,,
&& \\
& \norm{X_\H}_{\B} \leq \cteXH\,, 
&& \norm{\Dif X_\H}_{\B} \leq \cteDXH\,, \\
& \norm{\Dif^2 X_\H}_{\B} \leq \cteDDXH \,, 
&& \norm{\Dif X_\H^\top}_{\B} \leq \cteDXHT\,. 
\end{align*}

For the first integrals 
$p:\B\to \CC^{n-d}$ and the corresponding vector fields 
$X_p: \B \rightarrow \CC^{2n \times (n-d)}$, we have:
\begin{align*}
&\norm{\Dif p}_{\B} \leq \cteDp \,,
&& \norm{\Dif p^\top}_{\B} \leq \cteDpT \,,
&& \\
& \norm{X_p}_{\B} \leq \cteXp\,,
&& \norm{\Dif X_p}_{\B} \leq \cteDXp\,,
&& \norm{\Dif^2 X_p}_{\B} \leq \cteDDXp\,, \\
& \norm{X_p^\top}_{\B} \leq \cteXpT\,, 
&& \norm{\Dif X_p^\top}_{\B} \leq \cteDXpT\,, 
&& \norm{\Dif^2 X_p^\top}_{\B} \leq \cteDDXpT\,.
\end{align*}

For the first integral $c: \B \to \CC$ we have
\[
\norm{\Dif c}_{\B} \leq \cteDc \,.
\]
\item [$H_2$] 
The parameterization $K$ is real analytic in
a complex strip $\TT^d_\rho$, with $\rho>0$, which is contained
in the global domain:
\[
\dist (K(\TT^d_\rho),\pd \B)>0.
\]
Moreover, the components of $K$ and $\Dif K$
belong to $\Anal(\TT^d_\rho)$, and
there are constants $\sigmaDK$ and $\sigmaDKT$ such that
\[
\norm{\Dif K}_{\rho} < \sigmaDK\,, \qquad
\norm{(\Dif K)^\top}_{\rho} < \sigmaDKT\,.
\]
\item [$H_3$] 
We assume that $L(\theta)$ given by \eqref{eq:L:theo} has maximum rank for every $\theta \in \bar\TT_\rho^d$. 
Moreover,  there exists a constant
$\sigmaB$ such that
\[
\norm{B}_\rho < \sigmaB,
\]
where $B(\theta)$ is given by \eqref{eq:B}.

\item [$H_4$] 
There
exists a constant $\sigmaT$ such that
\[
\abs{\aver{T}^{-1}} < \sigmaT,
\]
where $T(\theta)$ is given by \eqref{eq:T}.

\item [$H_5$] 
The frequency $\omega$ belongs to $\Dioph{\gamma}{\tau}$, given
by \eqref{eq:def:Dioph}, for certain $\gamma >0$ and $\tau \geq d-1$.
\end{itemize}
Under the above hypotheses, for each $0<\rho_\infty<\rho$ 
there exists a constant $\mathfrak{C}_1$ such that, if
the error of invariance
\begin{equation}\label{eq:invE}
E(\theta)=X_\H(K(\theta))-\Dif K(\theta) \omega, 
\end{equation}
satisfies
\begin{equation}\label{eq:KAM:HYP}
\frac{\mathfrak{C}_1 \norm{E}_\rho}{\gamma^4 \rho^{4 \tau}} < 1\,,
\end{equation}
then there exists an invariant torus
$\torus_\infty = K_\infty(\TT^d)$ with frequency $\omega$, satisfying
$K_\infty \in \Anal(\TT^{d}_{\rho_\infty})$ and
\[
\norm{\Dif K_\infty}_{\rho_\infty} < \sigmaDK\,,
\qquad
\norm{(\Dif K_\infty)^\top}_{\rho_\infty} < \sigmaDKT\,,
\qquad
\dist(K_\infty(\TT^d_{\rho_\infty}),\pd \B) > 0\,.
\]
Furthermore, the objects are close to the original ones: there exist constants $\mathfrak{C}_2$ and $\mathfrak{C}_3$ such that
\begin{equation}\label{eq:close}
\norm{K_\infty - K}_{\rho_\infty} < \frac{\mathfrak{C}_2 \norm{E}_\rho}{\gamma^2 \rho^{2\tau}}\,,
\qquad
\abs{\aver{c \circ K_\infty} - \aver{c \circ K}} < \frac{\mathfrak{C}_3 \norm{E}_\rho}{\gamma^2 \rho^{2\tau}}\,,
\end{equation}
The constants $\mathfrak{C}_1$, $\mathfrak{C}_2$ and $\mathfrak{C}_3$ are
given explicitly in 
Appendix \ref{ssec:consts}.
\end{theorem}

\begin{remark}
If $d=n$ then there are no additional first integrals
and we recover the classical KAM theorem for Lagrangian tori.
The corresponding estimates follow by taking zero the constants
$\cteDp=0$,
$\cteDpT=0$,
$\cteXp=0$,
$\cteXpT=0$,
$\cteDXp=0$,
$\cteDXpT=0$,
$\cteDDXp=0$, and
$\cteDDXpT=0$.
Thus, as a
by-product, we obtain optimal quantitative estimates for the
KAM theorem for flows stated in \cite{GonzalezJLV05}.
\end{remark}

\begin{remark}
In the \emph{canonical case}, we have $\Omega=\tilde \Omega=\Omega_0$, $G = I_{2n}$,
and $J=\Omega_0$. Hence, we have 
$\cteOmega=1$,
$\cteDOmega=0$, 
$\ctetOmega=1$, 
$\cteDtOmega=0$,
$\cteDDtOmega=0$, 
$\cteG=1$,
$\cteDG=0$, 
$\cteDDG=0$,
$\cteJ=1$,
$\cteDJ=0$,
$\cteDDJ=0$,
$\cteJT=1$, and
$\cteDJT=0$.
\end{remark}

\begin{remark}
Notice that the condition $d\geq 2$ is optimal. For $d=1$, not only the torus becomes
a periodic orbit (and the result would follow from an standard implicit theorem without
small divisors) but also $X_\H$ is completely integrable.
\end{remark}

\begin{remark}\label{rem:unicity}
The existence of a $d$-dimensional invariant torus with frequency $\omega$
implies the existence of a $(n-d)$-parameter family of invariant tori with
frequency $\omega$. The family is locally unique, meaning that if there
is an invariant torus with frequency $\omega$ close enough to the family,
then it is a member of the family. Notice also that Theorem~\ref{theo:KAM}
states the existence of the parameterization of an invariant torus, but
that we can also change the phase to obtain a new parameterization. As
mentioned in Section~\ref{ssec:univocal}, both indeterminacies (the phase
and the element of the family) could be fixed by adding $n$ extra scalar
equations to the invariance equation. 
\end{remark}

\subsection{Generalized iso-energetic KAM theorem}\label{ssec:theo:iso}

Let us consider the setting presented in Section \ref{ssec:theo},
and let us focus on the first integral $c : \mani \rightarrow \RR$
that commutes with $\H$ and $p=(p_1,\ldots,p_{n-d})$.
It is clear that if $\torus=K(\TT^d)$ is invariant under $X_h$
then $\torus$ is contained in the hypersurface $c(z)=c_0$, for $c_0\in \RR$.
In this section we are interested in finding an invariant torus
by pre-fixing such hypersurface, that is, our aim is to obtain
a parameterization satisfying
\begin{equation}\label{eq:inv:iso}
X_\H(K(\theta))=\Lie{\omega}K(\theta)\,,
\qquad
\aver{c \circ K} = c_0 \,,
\end{equation}
where $c_0\in \RR$ is fixed and we think in $\omega \in \PP \RR^d$.

The following result, which is an extension of Theorem \ref{theo:KAM} to this
generalized iso-energetic context, establishes quantitative sufficient
conditions for the existence of a solution of \eqref{eq:inv:iso} close
to an approximate one. For this reason,
we refer the reader
to Section \ref{ssec:theo} for a compendium of
the objects involved in the result and we do not
estate the common hypotheses.

\begin{theorem}[Iso-energetic KAM theorem with first
integrals]\label{theo:KAM:iso}
Let us consider the setting of Theorem \ref{theo:KAM}, 
assume that the hypotheses $H_2$ and $H_3$ hold,
and replace $H_1$, $H_4$ and $H_5$ by
\begin{itemize}
\item [$H_1'$] Assume that all estimates in $H_1$ hold. In addition,
for the first integral $c: \B \to \CC$ we have
\[
\norm{\Dif^2 c}_{\B} \leq \cteDDc \,.
\]
\item [$H_4'$] There exists a constant $\sigmaTc$ such that
\[
\abs{\aver{T_c}^{-1}} < \sigmaTc\,,
\]
where
$T_c:\TT^d_\rho \to
\CC^{(n+1)\times(n+1)}$ 
is the extended torsion matrix
\begin{equation}\label{eq:Tc}
T_c(\theta) := 
\begin{pmatrix} 
T(\theta) & \homega \\
\Dif c(K(\theta)) N(\theta) & 0 
\end{pmatrix}, 
\qquad
\homega := 
\begin{pmatrix} \omega  \\ 0_{n-d} \end{pmatrix} \,.
\end{equation}

\item [$H_5'$] Let us consider a constant $\sigmao>1$ and a frequency
vector $\omega_*$ in the set $\Dioph{\gamma}{\tau}$, given by \eqref{eq:def:Dioph},
for certain $\gamma >0$ and $\tau \geq d-1$.
Then, we assume that $\omega \in \RR^d$ is contained in the ray
\[
\Theta=\Theta(\omega_*,\sigmao)=\{ s \omega_* \in \RR^d\,:\, 1<s<\sigmao\} \subset
\Dioph{\gamma}{\tau} \,.
\]
Notice that, by definition, we have $\dist (\omega,\pd \Theta)>0$.
\end{itemize}
Under the above hypotheses, for each $0<\rho_\infty<\rho$
there exists a constant $\mathfrak{C}_1$ such that, if
the total error
\begin{equation}\label{eq:invE:iso}
E_c(\theta)=
\begin{pmatrix}
E(\theta) \\
E^\omega
\end{pmatrix}
= 
\begin{pmatrix}
X_\H(K(\theta))-\Dif K(\theta) \omega \\
\aver{c \circ K} - c_0
\end{pmatrix}
\end{equation}
satisfies
\begin{equation}\label{eq:KAM:HYP:iso}
\frac{\mathfrak{C}_1 \norm{E_c}_\rho}{\gamma^4 \rho^{4 \tau}} < 1\,,
\end{equation}
then there exists an invariant torus
$\torus_\infty = K_\infty(\TT^d)$ with frequency $\omega_\infty \in \Theta$, satisfying
$K_\infty \in \Anal(\TT^{d}_{\rho_\infty})$ and
\[
\norm{\Dif K_\infty}_{\rho_\infty} < \sigmaDK\,,
\qquad
\norm{(\Dif K_\infty)^\top}_{\rho_\infty} < \sigmaDKT\,,
\qquad
\dist(K_\infty(\TT^d_{\rho_\infty}),\pd \B) > 0\,.
\]
Furthermore, the objects are close to the initial ones: there exist constants 
$\mathfrak{C}_2$ and 
$\mathfrak{C}_3$ 
such that
\begin{equation}\label{eq:close:iso}
\norm{K_\infty - K}_{\rho_\infty} < \frac{\mathfrak{C}_2 \norm{E_c}_\rho}{\gamma^2 \rho^{2\tau}}\,,
\qquad
\abs{\omega_\infty - \omega} < \frac{\mathfrak{C}_3 \norm{E_c}_\rho}{\gamma^2 \rho^{2\tau}}\,.
\end{equation}
The constants $\mathfrak{C}_1$, $\mathfrak{C}_2$ and $\mathfrak{C}_3$ are
given explicitly in 
Appendix \ref{ssec:consts}.
\end{theorem}

\begin{remark}
If we consider the case $c(z)=h(z)$ then we recover the classical iso-energetic situation.
Notice that, if $\torus$ is invariant with frequency $\omega$, then the frame $P(\theta)$ is 
symplectic, and 
\[
\begin{split}
	\Dif h(K(\theta)) N(\theta) & = X_h(K(\theta))^\top \Omega(K(\theta)) N(\theta)= 
	-\omega^\top \Dif K(\theta)^\top \Omega(K(\theta)) N(\theta) \\ & = 
	\begin{pmatrix} \omega^\top  & 0_{n-d}^\top \end{pmatrix}= {\hat \omega}^\top.
\end{split}
\]
Hence, the extended torsion matrix for an invariant torus is
\begin{equation}
\label{def:T_isoenergetic}
T_h(\theta) := 
\begin{pmatrix} 
T(\theta) & \homega \\
\homega^\top & 0 
\end{pmatrix}\,.
\end{equation}
\end{remark}

\begin{remark}
If we consider the case $c(z)=p_j(z)$ then we have an iso-momentum situation.
In this case, if $\torus$ is invariant with frequency $\omega$, then 
\[
	\Dif p_j(K(\theta)) N(\theta)= -X_{p_j}(K(\theta))^\top \Omega(K(\theta)) N(\theta) 
	 =  e_{d+j}^\top\, ,
\]
where $e_{d+j}$ is the $(d+j)$-th canonic vector of $\RR^n$ (it has $1$ in the $(d+j)$-th component and
$0$ elsewhere).
Hence, the extended torsion matrix for an invariant torus is
\begin{equation}
\label{def:T_isomomentum}
T_{p_j}(\theta) := 
\begin{pmatrix} 
T(\theta) & \homega \\
e_{d+j}^\top & 0 
\end{pmatrix} \,.
\end{equation}
\end{remark}

\section{Some lemmas to control approximate geometric properties}
\label{sec:lemmas}

In this section we present some estimates regarding the control of
some geometric properties for an approximately invariant torus. For
the sake of clarity, we reduce the repetition of hypotheses and present
a unique setting for the whole section, consisting in the
assumptions of the KAM Theorems in Section \ref{ssec:theo} and \ref{ssec:theo:iso}.

\subsection{Estimates for cohomological equations}
Let us first introduce some useful notation regarding the so-called
\emph{cohomological equations}
that play an important role in KAM theory.
Given $\omega \in \RR^d$ and a periodic function $v$, we consider the cohomological equation
\begin{equation}\label{eq:calL}
\Lie{\omega} u = v- \aver{v}\,, \qquad
\Lie{\omega} := -\sum_{i=1}^d \omega_i \frac{\pd}{\pd \theta_i}.
\end{equation}
The notation $\Lie{}$ comes from ``left-operator''.

Let us assume that $v$ is continuous and $\omega$ is rationally independent (this implies that the
flow $t \mapsto \omega t$ is quasi-periodic).
If there exists a continuous zero-average solution of equation \eqref{eq:calL},
then it is unique and will be denoted by
$u = \R{\omega} v$.
The notation $\R{}$ comes from ``right-operator''.

Note that the formal solution of equation \eqref{eq:calL} is
immediate. Actually, if $v$ has the Fourier
expansion $v(\theta)=\sum_{k \in \ZZ^d} \hat v_k \ee^{2\pi
\ii k \cdot \theta }$ and the dynamics is quasi-periodic, then
\begin{equation}\label{eq:small:formal}
\R{\omega} v(\theta) = \sum_{k \in \ZZ^d \backslash \{0\} } \hat u_k \ee^{2\pi
\ii k \cdot  \theta}, \qquad \hat u_k = \frac{-\hat
v_k}{2\pi \ii  k \cdot \omega}.
\end{equation}
In particular, this implies that $\R{\omega} v =0$ if $v=0$.
The solutions of equation \eqref{eq:calL} differ by 
their average.

We point out that quasi-periodicity is not enough to ensure regularity of the
solutions of cohomological equations.  This is related to the effect of the
small divisors $k \cdot \omega$ in
equation \eqref{eq:small:formal}.
To deal with regularity, we
require stronger non-resonant conditions on the vector of frequencies.
In this paper, we consider the classic Diophantine conditions in $H_5$ and $H_5'$.

\begin{lemma}[R\"ussmann estimates]\label{lem:Russmann}
Let 
$\omega \in \Dioph{\gamma}{\tau}$
for some $\gamma>0$ and $\tau \geq d-1$.
Then, for any $v \in \Anal(\TT^d_\rho)$, with $\rho>0$,
there exists a unique zero-average solution
of $\Lie{\omega} u = v -\aver{v}$,
denoted by $u=\R{\omega}v$. Moreover, for any $0<\delta<\rho$ we
have that $u \in \Anal(\TT^d_{\rho-\delta})$ and the estimate
\[
\norm{u}_{\rho-\delta} \leq \frac{c_R}{\gamma \delta^\tau}
\norm{v}_\rho\,,
\]
where $c_R$ is a constant that depends on $d$, $\tau$ and possibly on $\delta$.
\end{lemma}

\begin{proof}
There is no need to reproduce here this classical result, and we refer
the reader to the original references \cite{Russmann75,Russmann76a},
where a uniform bound (independent of $\delta$) is obtained. We refer
to \cite{FiguerasHL17} for sharp non-uniform computer-assisted
estimates (in the discrete case) of the form $c_R=c_R(\delta)$.
representing a substantial
advantage in order to apply the result to particular problems. Adapting these
estimates to the continuous case is straightforward. Also, we refer to \cite{FiguerasHL18}
for a numerical quantification of these estimates and for an analysis of
the different sources of overestimation.
\end{proof}

\subsection{Approximate conserved quantities}\label{ssec:app:conserved}

If $\torus=K(\TT^d)$ is not an invariant torus with frequency $\omega$, it is clear that
a first integral in involution $c$ (such as the energy $\H$ or any of the components of $p$) is
not necessarily preserved along
$z(t) = K(\theta_0 + \omega t)$, since this is not a true trajectory. 
However, we can ``shadow''
its evolution in terms of the error of invariance.

\begin{lemma}\label{lem:cons:H:p}
Let us consider the setting of Theorem \ref{theo:KAM} or
Theorem \ref{theo:KAM:iso}. Then, for a conserved quantity $c$ 
the following estimates hold: 
\begin{align}
&
\norm{c \circ K - \aver{c \circ K}}_{\rho-\delta} \leq \frac{c_R \cteDc}{\gamma \delta^\tau}\norm{E}_\rho \,,
&&  
\label{eq:c-avgc} \\
&
\norm{\Lie{\omega}(\Dif (c \circ K))}_{\rho-\delta} \leq \frac{ d \cteDc}{\delta}\norm{E}_\rho\,, 
&& 
\norm{\Lie{\omega}(\Dif (c \circ K))^\top}_{\rho-\delta} \leq \frac{ \cteDc}{\delta}\norm{E}_\rho\,,
\label{eq:LDc} \\
&
\norm{\Dif(c \circ K) }_{\rho-2\delta} \leq \frac{c_R d \cteDc }{\gamma \delta^{\tau+1}}\norm{E}_\rho\,, 
&&
\norm{(\Dif(c \circ K))^\top }_{\rho-2\delta} \leq \frac{c_R \cteDc }{\gamma \delta^{\tau+1}}\norm{E}_\rho\,.
\label{eq:Dc}
\end{align}
In particular, 
\begin{align}
&
\norm{p \circ K - \aver{p \circ K}}_{\rho-\delta} \leq \frac{c_R \cteDp}{\gamma \delta^\tau}\norm{E}_\rho\,,
&&  
\label{eq:p-avgp} \\
&
\norm{\Lie{\omega}(\Dif (p \circ K))}_{\rho-\delta} \leq \frac{ d \cteDp}{\delta}\norm{E}_\rho \,,
&& 
\norm{\Lie{\omega}(\Dif (p \circ K))^\top}_{\rho-\delta} \leq \frac{ \cteDpT}{\delta}\norm{E}_\rho \,,
\label{eq:LDp} \\
&
\norm{\Dif(p \circ K) }_{\rho-2\delta} \leq \frac{ c_R d \cteDp }{\gamma \delta^{\tau+1}}\norm{E}_\rho\,,
&&
\norm{(\Dif(p \circ K))^\top }_{\rho-2\delta} \leq \frac{c_R \cteDpT }{\gamma \delta^{\tau+1}}\norm{E}_\rho\,.
\label{eq:Dp}
\end{align}
\end{lemma}
\begin{proof}
We will prove the result first for the conserved quantity $c$.
This case includes analogous estimates for each of the 
first integrals $h$ and
$p_i$, 
with
$i= 1,\dots, n-d$.
We apply $\Lie{\omega}(\cdot)=-\Dif (\cdot) \omega$ in the expression $c\circ K$, thus obtaining
\begin{equation}
\label{eq:LcK}
\begin{split}
\Lie{\omega}(c(K(\theta))) = {} & \Dif c(K(\theta)) \Lie{\omega} K(\theta) \\
= {} &
\Dif c(K(\theta))
\left(
E(\theta)
-X_\H(K(\theta))
\right) \\
= {} & \Dif c(K(\theta)) E(\theta)\,.
\end{split}
\end{equation}
In the second line we used the expression of the error of invariance \eqref{eq:invE}, and
in the third line we used that $c$ is in involution with $\H$. In particular, 
\[
	\Norm{\Lie{\omega}(c\circ K))}_\rho \leq \cteDc \norm{E}_\rho,
\]
where we use that $\norm{\Dif c}_\B\leq \cteDc$.
Thus, we end up with
\[
c(K(\theta))-\aver{c \circ K} = \R{\omega}(\Dif c(K(\theta)) E(\theta))\, ,
\]
and the estimate \eqref{eq:c-avgc} follows applying Lemma \ref{lem:Russmann}. 

In order to prove \eqref{eq:LDc} and \eqref{eq:Dc} we just differentiate
with respect to $\theta_\ell$, for $\ell= 1,\dots, d$, both formulae
\eqref{eq:LcK} and \eqref{eq:LDc} and apply Cauchy estimates. 
Firstly,
\[
\Norm{\frac{\partial}{\partial \theta_\ell}\Lie{\omega}(c \circ K) }_{\rho-\delta} \leq {} 
\frac{\cteDc }{\delta}\norm{E}_\rho, 
\]
so estimates in \eqref{eq:LDc} follow immediately.
Secondly,
\[
\Norm{\frac{\partial}{\partial \theta_\ell}(c \circ K) }_{\rho-2\delta} \leq {} 
\frac{c_R \cteDc }{\gamma \delta^{\tau+1}}\norm{E}_\rho, 
\]
and then estimates in \eqref{eq:Dc} follow.

In order to prove \eqref{eq:p-avgp}, \eqref{eq:LDp} and \eqref{eq:Dp} we just notice that \eqref{eq:c-avgc}, \eqref{eq:LDc} and \eqref{eq:Dc}
work for any of the first integrals $p_i$, for $i= 1,\dots, n-d$. 
Then, we
change 
the occurrences of $c$ by $p_i$ in the formulae,
with $\norm{\Dif p_i}_\B\leq c_{p_i,1}$, and 
use that 
\[
	\cteDp= \max_{i= 1,\dots,n-d} c_{p_i,1}, \qquad \cteDpT= \sum_{i= 1,\dots,n-d} c_{p_i,1}
\]
to obtain the bounds.
\end{proof}

\subsection{Approximate isotropicity of tangent vectors}

In this section we prove that if $\torus$ is approximately
invariant, then $K^*\sform$ is small and can be controlled by
the error of invariance. We refer the reader to \cite{FontichLS09,GonzalezHL13}
for similar computations, using generic constants in the estimates.

\begin{lemma}\label{lem:isotrop}
Let us consider the setting of 
Theorem \ref{theo:KAM} or
Theorem \ref{theo:KAM:iso}. 
Let us consider $\Omega_K:\TT^d \to \RR^{2n \times 2n}$, the matrix representation of
the pull-back on $\TT^d$ of the symplectic form. We have
\begin{equation}\label{eq:OmegaK:aver}
\aver{\Omega_K} = O_d\,,
\end{equation}
and
the following estimate holds:
\begin{equation}\label{eq:estOmegaK}
\norm{\Omega_K}_{\rho-2\delta} \leq \frac{\COmegaK }{\gamma \delta^{\tau+1}}\norm{E}_\rho\,,
\end{equation}
where the constant $\COmegaK$ is provided in Table \ref{tab:constants:all}.
\end{lemma}

\begin{proof}
Property \eqref{eq:OmegaK:aver} follows directly from the exact
symplectic structure, since $K^*\sform={\rm d}( K^{*}\aform)$.
In more algebraic terms,
we have
\begin{align*}
\Omega_K(\theta) = {} & (\Dif K(\theta))^\top
\left(
(\Dif a(K(\theta)))^\top-
\Dif a(K(\theta))
\right)
\Dif K(\theta) \\
= {} & (\Dif (a(K(\theta))))^\top \Dif K(\theta)
- (\Dif K(\theta))^\top \Dif(a(K(\theta)))\,.
\end{align*}
and so,
the components of $\Omega_K(\theta)$
are 
\begin{align*}
(\Omega_K)_{i,j} (\theta)
= {} & 
\sum_{m=1}^{2n} 
\left(
\frac{\partial(a_m(K(\theta))) }{\partial \theta_i}
\frac{\partial K_m(\theta)}{\partial {\theta_j}} 
- 
\frac{\partial(a_m(K(\theta))) }{\partial {\theta_j}}
\frac{\partial K_m(\theta)}{\partial {\theta_i}}
\right) \\
= {} & 
\sum_{m=1}^{2n} 
\left(
\frac{\partial}{\partial {\theta_i}} 
\left( 
a_m(K(\theta))
\frac{\partial(
K_m(\theta))}{\partial \theta_j}
\right)
- 
\frac{\partial}{\partial {\theta_j}} 
\left( 
a_m(K(\theta))
\frac{\partial(K_m(\theta))}{\partial \theta_i}
\right)\right)
\,.
\end{align*}
Hence, the components of $\Omega_K$ are sums of derivatives of periodic functions, 
and we obtain $\aver{\Omega_K}=O_d$.

Now we will use two crucial geometric properties (following Appendix A in
\cite{GonzalezHL13}).
Using the fact that $\sform$ is closed, we first obtain
the expression
\[
 \frac{\partial \Omega_{r,s}(z)}{\partial z_t} 
+\frac{\partial \Omega_{s,t}(z)}{\partial z_r} 
+\frac{\partial \Omega_{t,r}(z)}{\partial z_s} = 0,
\]
for any triplet $(r,s,t)$.
The second property is obtained by taking derivatives
at both sides of $\Omega(z) X_\H(z)= (\Dif \H(z))^\top$, obtaining
\[
\frac{\partial^2\H}{\partial z_i \partial z_j}(z)= \sum_{m=1}^{2n}
\left(\frac{\partial \Omega_{j,m}(z)}{\partial z_i} X_m(z)
+ \Omega_{j,m}(z) \frac{\partial X_m (z)}{\partial z_i} \right), 
\]
for any $i, j$, where we use the notation $X_m= (X_\H)_m$ for the components of $X_\H$.
Hence,
\[
\begin{split}
0 = {} &  \displaystyle \frac{\partial^2\H}{\partial z_j \partial z_i}(z) - \frac{\partial^2\H}{\partial z_i \partial z_j}(z) \\
  = {} & 
\displaystyle 
 \sum_{m=1}^{2n} 
\left(\frac{\partial \Omega_{i,m}(z)}{\partial z_j} X_m(z)
+ \Omega_{i,m}(z) \frac{\partial X_m (z)}{\partial z_j} \right) \\
& - 
\sum_{m=1}^{2n} \left(
\frac{\partial \Omega_{j,m}(z)}{\partial z_i} X_m(z)
+ \Omega_{j,m}(z) \frac{\partial X_m (z)}{\partial z_i} \right)   \\
  = {} & \displaystyle 
\sum_{m=1}^{2n} \left(\frac{\partial \Omega_{i,j}(z)}{\partial z_m} X_m(z) 
+ \Omega_{i,m}(z) \frac{\partial X_m (z)}{\partial z_j}
+ \Omega_{m,j}(z) \frac{\partial X_m (z)}{\partial z_i} \right) 
\end{split} 
\]
The above expressions yield the formula
\begin{equation}\label{eq:prop2killSK}
\Dif \Omega(z) [ X_\H(z) ]
+ (\Dif X_\H(z))^\top \Omega(z)
+ \Omega(z) \Dif X_\H(z)=
O_{2n}.
\end{equation}

Then, we compute the action of $\Lie{\omega}$ on $\Omega_K$, thus
obtaining
\begin{equation}\label{eq:LieOK}
\begin{split}
\Lie{\omega} \Omega_K(\theta) = {} &
\Lie{\omega} (\Dif K(\theta))^\top \Omega(K(\theta)) \Dif K(\theta)
+
(\Dif K(\theta))^\top \Lie{\omega} (\Omega(K(\theta))) \Dif K(\theta) \\
& + (\Dif K(\theta))^\top \Omega(K(\theta)) \Lie{\omega} \Dif K(\theta)\,,
\end{split}
\end{equation}
and we use the properties
(obtained from the invariance equation \eqref{eq:invE})
\begin{align*}
& \Lie{\omega} (\Dif K(\theta)) = \Dif E(\theta) - \Dif X_\H(K(\theta) \Dif K(\theta) \,,\\
& \Lie{\omega} (\Omega( K(\theta))) = \Dif \Omega(K(\theta))
[\Lie{\omega}K(\theta)] = \Dif \Omega(K(\theta))
\left[
E(\theta)
-
X_\H(K(\theta))
\right]\,,
\end{align*}
in combination with \eqref{eq:prop2killSK}, thus ending up with
\begin{align*}
\Lie{\omega} \Omega_K(\theta) = {} &
(\Dif E(\theta))^\top \Omega(K(\theta)) \Dif K(\theta)
+
(\Dif K(\theta))^\top (\Dif \Omega(K(\theta)) [E(\theta)]) \Dif K(\theta) \\
& + (\Dif K(\theta))^\top \Omega(K(\theta)) \Dif E(\theta)\,.
\end{align*}
The expression $\Lie{\omega} \Omega_K(\theta)$ is controlled using $H_1$, $H_2$, the 
Banach algebra properties and Cauchy estimates as follows
\begin{align}
\norm{\Lie{\omega} \Omega_K}_{\rho-\delta} \leq {} &
\norm{(\Dif E)^\top}_{\rho-\delta} \norm{\Omega}_{\B} \norm{\Dif K}_\rho
+
\norm{(\Dif K)^\top}_{\rho} \norm{\Dif \Omega}_\B \norm{E}_\rho \norm{\Dif K}_\rho \nonumber \\
& + \norm{(\Dif K)^\top}_\rho \norm{\Omega}_\B \norm{\Dif E}_{\rho-\delta} \nonumber \\
\leq {} & \frac{2n \cteOmega \sigmaDK 
+ \sigmaDKT \cteDOmega \sigmaDK \delta
+ d \sigmaDKT \cteOmega}{\delta} \norm{E}_\rho
=: \frac{\CLieOmegaK}{\delta} \norm{E}_\rho
\,. \label{eq:CLieOK}
\end{align}
In particular, we used that
$\norm{(\Dif \Omega \circ K ) [E]}_\B \leq \norm{\Dif \Omega }_\B \norm{E}_\rho$
(see Section \ref{ssec-anal-prelims}).
The estimate in \eqref{eq:estOmegaK} is obtained as follows
\begin{equation}\label{eq:COK}
\norm{\Omega_K}_{\rho-2 \delta} \leq \frac{c_R}{\gamma \delta^\tau} \norm{\Lie{\omega}\Omega_K}_{\rho-\delta}
\leq \frac{c_R \CLieOmegaK}{\gamma \delta^{\tau+1}} \norm{E}_\rho =:
\frac{\COmegaK}{\gamma \delta^{\tau+1}} \norm{E}_\rho\,,
\end{equation}
where 
we used Lemma \ref{lem:Russmann} and the estimate in \eqref{eq:CLieOK}.
\end{proof}

\begin{remark}
Notice that, even though the KAM theorems \ref{theo:KAM} and
\ref{theo:KAM:iso} do not require a quantitative control con the 1-form
$\aform$, the 
facts that the symplectic structure $\sform$ is exact and the vector field $X_\H$ is
globally Hamiltonian are crucial
to obtain the above result.
\end{remark}

\subsection{Approximate symplectic frame}\label{ssec:symp}

In this section we prove that if $\torus$ is approximately
invariant, then we can construct an adapted frame that is approximately symplectic.
First step is finding an adapted approximately Lagrangian bundle, that contains the tangent bundle $T_\torus \mani$.

\begin{lemma}\label{lem:Lang}
Let us consider the setting of 
Theorem \ref{theo:KAM} or
Theorem \ref{theo:KAM:iso}. Then,
the map
$L(\theta)$ given by \eqref{eq:L:theo} satisfies
\begin{equation}\label{eq:propL}
\norm{L}_\rho \leq \CL\,,
\qquad
\norm{L^\top}_\rho \leq \CLT\,,
\qquad
\aver{L^\top (\Omega \circ K) E}=0_n\,,
\end{equation}
and defines an approximately
Lagrangian bundle, i.e. the error map
\begin{equation}\label{eq:Elag}
\Elag(\theta):=L(\theta)^\top \Omega(K(\theta)) L(\theta)
\end{equation}
is small in the following sense:
\begin{equation}\label{eq:normElag}
\norm{\Elag}_{\rho-2\delta} \leq
\frac{\Clag}{\gamma \delta^{\tau+1}} \norm{E}_\rho \,.
\end{equation}
Furthermore, the objects
\begin{align}
G_L(\theta) := L(\theta)^\top G(K(\theta)) L(\theta) \,, \label{eq:def:GL} \\
\tilde \Omega_L(\theta) := L(\theta)^\top \tilde \Omega(K(\theta)) L(\theta) \,, \label{eq:def:tOmegaL}
\end{align}
are controlled as
\begin{equation}\label{eq:est:ObjL}
\norm{G_L}_\rho \leq \CGL\,, \qquad
\norm{\tilde \Omega_L}_\rho \leq \CtOmegaL\,.
\end{equation}
The above constants are given explicitly in Table \ref{tab:constants:all}.
\end{lemma}

\begin{proof}
We first obtain the property of the average in \eqref{eq:propL}
by computing
\[
L(\theta)^\top \Omega(K(\theta)) E(\theta)=
\begin{pmatrix}
(\Dif K(\theta))^\top \Omega(K(\theta)) E(\theta) \\
X_p(K(\theta))^\top \Omega(K(\theta)) E(\theta)
\end{pmatrix}\,.
\]
The upper term satisfies
\begin{align*}
(\Dif K(\theta))^\top \Omega(K(\theta)) E(\theta)
= {} & 
(\Dif K(\theta))^\top \Omega(K(\theta)) X_\H(K(\theta)) - \Omega_K(\theta) \omega \\
= {} & 
(\Dif K(\theta))^\top (\Dif \H(K(\theta)))^\top  - \Omega_K(\theta) \omega \\
= {} & 
(\Dif (\H(K(\theta))))^\top - \Omega_K(\theta) \omega \,,
\end{align*}
which has zero average, since the first term is the derivative of
a periodic function and $\Omega_K$ has zero average (see Lemma \ref{lem:isotrop}).
Moreover,
the lower term satisfies
\begin{equation}\label{eq:aux:comp}
X_p(K(\theta))^\top \Omega(K(\theta)) E(\theta)
= 
\Dif p(K(\theta)) E(\theta) =
-\Dif (p(K(\theta))) \omega \,.
\end{equation}
Hence, it is clear that $\aver{L^\top (\Omega \circ K) E}=0_n$.

We control the norm of the frame $L(\theta)$,
using $H_1$ and $H_2$, as follows
\begin{align}
\norm{L}_\rho \leq {} & \norm{\Dif K}_\rho +
\norm{X_p \circ K}_\rho \leq \sigmaDK+\cteXp =: \CL\,, \label{eq:CL}\\
\norm{L^\top}_\rho \leq {} & 
\max\{
\norm{(\Dif K)^\top}_\rho \, , \,
\norm{X_p^\top \circ K}_\rho\} \leq \max\{\sigmaDKT \, , \, \cteXpT\} =: \CLT\,. \label{eq:CLT}
\end{align}
We have obtained the 
estimates in \eqref{eq:propL}.
Using again the expression of $L(\theta)$ we obtain that the anti-symmetric matrix
\eqref{eq:Elag} is written as
\begin{align*}
\Elag (\theta) = {} &
\begin{pmatrix}
(\Dif K(\theta))^\top 
\Omega(K(\theta)) \Dif K(\theta) & 
(\Dif K(\theta))^\top 
\Omega(K(\theta)) X_p(K(\theta))
\\
X_p(K(\theta))^\top
\Omega(K(\theta)) \Dif K(\theta) & 
X_p(K(\theta))^\top
\Omega(K(\theta)) X_p(K(\theta))
\end{pmatrix} \,.
\end{align*} 
Using the expression $\Omega_K(\theta)$ in \eqref{def-OK},
performing similar computations as in \eqref{eq:aux:comp}, and using
the involution of the first integrals, we end up with
\[
\Elag (\theta) = 
\begin{pmatrix}
\Omega_K(\theta) & 
(\Dif (p(K(\theta))))^\top
\\
- \Dif (p(K(\theta))) &
O_{n-d}
\end{pmatrix}\,.
\]

Then, we have
\begin{align}
\norm{\Elag}_{\rho-2\delta} = {} & \max\{
\norm{\Omega_K}_{\rho-2\delta}+
\norm{(\Dif(p \circ K))^\top}_{\rho-2\delta} \, , \,
\norm{\Dif(p \circ K)}_{\rho-2\delta} \}\,, \nonumber \\
\leq {} & \frac{c_R \max\{\CLieOmegaK + \cteDpT \, , \, d  \cteDp \}}{\gamma \delta^{\tau+1}}
\norm{E}_\rho
=:
\frac{\Clag}{\gamma \delta^{\tau+1}}
\norm{E}_\rho \,, \label{eq:Clag}
\end{align}
where we use Lemmas \ref{lem:cons:H:p} and \ref{lem:isotrop}.
Thus, we have
obtained the estimate in \eqref{eq:normElag}.
Finally, the estimates in \eqref{eq:est:ObjL} with
\begin{align}
\norm{G_L}_\rho \leq {} & \CLT \cteG \CL
=: \CGL \,, \label{eq:CGL} \\
\norm{\tilde \Omega_L}_\rho \leq {} & \CLT \ctetOmega \CL
=: \CtOmegaL \,. \label{eq:CtOmegaL} 
\end{align}
follow directly.
\end{proof}

In the following lemma, we will see that that geometric constructions detailed in Section~\ref{ssec:sym:frame} lead, 
for an approximately invariant torus, to an approximately symplectic frame attached to the torus.

\begin{lemma}\label{lem:sympl}
Let us consider the setting of 
Theorem \ref{theo:KAM} or
Theorem \ref{theo:KAM:iso}. Then,
the map $N : \TT^d \to \RR^{2n \times n}$ given by \eqref{eq:N} satisfies
\begin{equation}\label{eq:propN}
\norm{N}_\rho \leq \CN\,,
\qquad
\norm{N^\top}_\rho \leq \CNT\,,
\end{equation}
and the map 
$P: \TT^d \to \RR^{2n \times 2n}$ 
given by
\[
P(\theta)=
\begin{pmatrix}
L(\theta) & N(\theta)
\end{pmatrix}\,,
\]
induces an approximately symplectic vector bundle isomorphism, i.e.,
the error map
\begin{equation}\label{eq:Esym}
\Esym(\theta) := P(\theta)^\top \Omega(K(\theta)) P(\theta)-\Omega_0\,,
\qquad
\Omega_0 = 
\begin{pmatrix}
O_n & -I_n \\
I_n & O_n
\end{pmatrix}\,,
\end{equation}
is small in the following sense:
\begin{equation}\label{eq:normEsym}
\norm{\Esym}_{\rho-2 \delta} \leq \frac{\Csym}{\gamma \delta^{\tau+1}}
\norm{E}_\rho\,.
\end{equation}
The above constants are given explicitly in Table \ref{tab:constants:all}.
\end{lemma}

\begin{proof}
First we control the norm of $N^0(\theta)$ in \eqref{eq:N0}, using $H_1$
and \eqref{eq:propL}, as
\begin{align}
\norm{N^0}_\rho \leq {} & \norm{J\circ K}_\rho \norm{L}_\rho
\leq \norm{J}_{\B} \norm{L}_\rho \leq \cteJ \CL =: \CNO \,, \label{eq:CNO} \\
\norm{(N^0)^\top}_\rho \leq {} & \norm{L^\top}_\rho \norm{(J\circ K)^\top}_\rho 
\leq \norm{L^\top}_\rho \norm{J^\top}_{\B} \leq \CLT \cteJT =: \CNOT \,, \label{eq:CNOT}
\end{align}
and the norm of
$A(\theta)$ in \eqref{eq:A}
as
\begin{equation}\label{eq:cA}
\norm{A}_\rho \leq 
\frac{1}{2} \norm{B^\top}_\rho \norm{L^\top (\tilde \Omega \circ K)
L}_\rho \norm{B}_\rho
\leq \frac{1}{2} \sigmaB \CtOmegaL \sigmaB =: \CA\,,
\end{equation}
where we also used $H_3$, the second estimate in \eqref{eq:est:ObjL} and
that $B(\theta)$ is symmetric.
We also control the complementary normal vectors as 
\begin{align}
\norm{N}_\rho \leq {} & 
\norm{L}_\rho \norm{A}_\rho +
\norm{N^0}_\rho \norm{B}_\rho 
\leq {}  \CL \CA + \CNO \sigmaB =: \CN \,, \label{eq:CN} \\
\norm{N^\top}_\rho \leq {} &  
\norm{A^\top}_\rho \norm{L^\top}_\rho + \norm{B^\top}_\rho \norm{(N^0)^\top}_\rho  
\leq {}  \CA \CLT + \sigmaB \CNOT =: \CNT \,, \label{eq:CNT}
\end{align}
where we used $H_1$, 
the estimates \eqref{eq:propL}, \eqref{eq:CNO}, \eqref{eq:CNOT}, and \eqref{eq:cA},
and the fact that $A(\theta)$ is anti-symmetric. Thus, we have obtained the 
estimates in \eqref{eq:propN}.

To characterize the error in the symplectic character of the frame, we
compute
\begin{equation}\label{eq:neqEsym}
\Esym(\theta)
= 
\begin{pmatrix}
L(\theta)^\top 
\Omega(K(\theta)) L(\theta) & 
L(\theta)^\top 
\Omega(K(\theta)) N(\theta) + I_n
\\
N(\theta)^\top
\Omega(K(\theta)) L(\theta) - I_n & 
N(\theta)^\top
\Omega(K(\theta)) N(\theta)
\end{pmatrix} \,,
\end{equation}
and we expand the components of this block matrix using \eqref{eq:N}, \eqref{eq:B},
and \eqref{eq:Elag}.
For example, we have
\begin{align}
L(\theta)^\top \Omega(K(\theta)) & N(\theta) \nonumber \\
= {} & 
L(\theta)^\top \Omega(K(\theta)) L(\theta) A(\theta)+
L(\theta)^\top \Omega(K(\theta)) N^0(\theta) B(\theta) \nonumber \\
= {} & \Elag(\theta) A(\theta) - L(\theta)^\top G(K(\theta)) L(\theta)B(\theta)
\nonumber \\
= {} & \Elag(\theta) A(\theta) - I_n\,, \label{eq:LON}
\end{align}
where we used that $J^\top \Omega=G$ and the definition of $B(\theta)$.
We also have
\begin{align}
N(\theta)^\top & \Omega(K(\theta)) N(\theta) \nonumber \\
= {} & 
B(\theta)^\top N^0(\theta)^\top \Omega(K(\theta)) L(\theta)
A(\theta) 
+ B(\theta)^\top N^0(\theta)^\top \Omega(K(\theta))
  N^0(\theta) B(\theta) \nonumber \\
& + A(\theta)^\top L(\theta)^\top \Omega(K(\theta)) L(\theta) A(\theta) 
+A(\theta)^\top L(\theta)^\top \Omega(K(\theta)) N^0(\theta)
B(\theta) \nonumber \\
= {} & A(\theta)^\top \Elag(\theta) A(\theta)+A(\theta)-A(\theta)^\top 
+ B(\theta)^\top L(\theta)^\top \tilde \Omega(K(\theta))
  L(\theta) B(\theta) \nonumber \\
= {} & A(\theta)^\top \Elag(\theta) A(\theta) \,. \label{eq:NON}
\end{align}

Then,
introducing the expressions \eqref{eq:Elag},
\eqref{eq:LON} and \eqref{eq:NON} into \eqref{eq:neqEsym}, we get
\[
\Esym(\theta)
= 
\begin{pmatrix}
\Elag(\theta) & \Elag(\theta) A(\theta) \\
A(\theta)^\top \Elag(\theta) &
A(\theta)^\top \Elag(\theta) A(\theta)
\end{pmatrix}\,,
\]
which is controlled as
\begin{equation}\label{eq:Csym}
\norm{\Esym}_{\rho-2\delta} \leq \frac{(1+\CA) \max\{1\, , \, \CA\} \Clag}{\gamma \delta^{\tau+1}}
\norm{E}_\rho =: \frac{\Csym}{\gamma \delta^{\tau+1}}
\norm{E}_\rho \,,
\end{equation}
thus obtaining the estimate \eqref{eq:normEsym}.
\end{proof}

\begin{remark}
The above estimates can be readily adapted to \textbf{Case III}, for which $A=0$. In this case
we have (computations are left as an exercise to the reader)
\[
N(\theta) = N^0(\theta) B(\theta) 
\,,
\qquad
\Esym(\theta)
= 
\begin{pmatrix}
\Elag(\theta) & O_n \\
O_n &
B(\theta)^\top \Elag(\theta) B(\theta)
\end{pmatrix}\,.
\]
The corresponding estimates
are given explicitly in Table \ref{tab:constants:all}.
\end{remark}

\subsection{Sharp control of the torsion matrix}

In this section we will control the torsion matrix $T(\theta)$, given in \eqref{eq:T}. To do so, we could 
use directly Cauchy estimates to control $\Lie{\omega} N(\theta)$, resulting in an additional bite in the domain and 
an additional factor $\delta$ in the denominator (among other overestimations). Hence, 
in order to improve this estimate, thus enhancing the threshold
of validity of the result, we perform a finer analysis of the expression for
$T(\theta)$. To this end, it is convenient to include here an additional smallness condition for the error of invariance
(see \eqref{eq:fake:cond} below), that later on it turns out will be rather irrelevant . 

\begin{lemma}\label{lem:twist}
Let us consider the setting of
Theorem \ref{theo:KAM} or
Theorem \ref{theo:KAM:iso},
and let us assume that
\begin{equation}\label{eq:fake:cond}
\frac{\norm{E}_\rho}{\delta}<\cauxT\,,
\end{equation}
where $\cauxT$ is an independent constant.
Then, 
the torsion matrix $T(\theta)$, given by \eqref{eq:T}, 
has components in $\Anal(\TT^d_{\rho-\delta})$ and
satisfies
the estimate
\[
\norm{T}_{\rho-\delta} \leq \CT\,,
\]
where the constant $\CT$ is provided in Table \ref{tab:constants:all}.
\end{lemma}

\begin{proof}
Recalling that $\Lie{\omega}(\cdot) = - \Dif (\cdot) \omega$ and \eqref{eq:Loper}, we
have
\begin{equation}\label{eq:line1}
\Loper{N}(\theta) = 
\Dif X_\H (K(\theta)) N(\theta) + \Lie{\omega}N (\theta)\,,
\end{equation}
where
\begin{equation}\label{eq:LieN:expanded}
\Lie{\omega}N(\theta) = 
\Lie{\omega} L(\theta) A(\theta)  
+ L(\theta) \Lie{\omega} A(\theta) 
+ \Lie{\omega} N^0(\theta) B(\theta)
+ N^0(\theta) \Lie{\omega}B(\theta) \,.
\end{equation}
Then, we must estimate the terms 
$\Lie{\omega}L(\theta)$,
$\Lie{\omega} A(\theta)$,
$\Lie{\omega} N^0(\theta)$, and
$\Lie{\omega}B(\theta)$,
that appear above. 

We start considering
\begin{equation}\label{eq:LieK}
\Lie{\omega}K(\theta)=E(\theta) - X_{\H}(K(\theta))
\end{equation}
which is controlled as
\begin{equation}\label{eq:CLieK}
\norm{\Lie{\omega}K}_{\rho} \leq \norm{E}_\rho+ \norm{X_\H \circ K}_\rho \leq \delta \cauxT + \cteXH =: \CLieK\,,
\end{equation}
where we used $H_1$ and the assumption \eqref{eq:fake:cond}. Then we consider the object
$\Lie{\omega}L(\theta)$, given by
\begin{align*}
\Lie{\omega} L(\theta) = {} &
\begin{pmatrix}
\Lie{\omega} (\Dif K(\theta)) & \Lie{\omega} (X_p(K(\theta)))
\end{pmatrix} \,.
\end{align*}
Notice that the left block in the above expression follows
by taking derivatives at both sides of \eqref{eq:LieK}, i.e.
\[
\Lie{\omega} (\Dif K(\theta)) = \Dif E(\theta) -
\Dif X_\H(K(\theta)) [\Dif K(\theta)]\,,
\]
and the right block follows from
\[
\Lie{\omega}(X_p(K(\theta)))
=\Dif X_p(K(\theta)) [\Lie{\omega}K(\theta)] \,.
\]
Then, using Cauchy estimates, 
the Hypothesis $H_1$ and $H_2$, 
the assumption \eqref{eq:fake:cond} and the estimate
\eqref{eq:CLieK}, we have
\begin{equation}\label{eq:CLieL}
\norm{\Lie{\omega}L}_{\rho-\delta} \leq  d \cauxT + \cteDXH \sigmaDK + \cteDXp \CLieK
=: \CLieL \,. 
\end{equation}
Similarly, we obtain the estimates
\begin{equation}\label{eq:CLieLT}
\norm{\Lie{\omega}L^\top}_{\rho-\delta} \leq 
\max \left\{2n \cauxT + \cteDXHT \sigmaDK \, , \, \cteDXpT \CLieK \right\} 
=: \CLieLT \,.
\end{equation}

The term $\Lie{\omega}N^0(\theta)$ is controlled using that
\begin{equation}\label{eq:LieNO}
\Lie{\omega}N^0(\theta) = \Lie{\omega} (J(K(\theta))) L(\theta)+J(K(\theta)) \Lie{\omega}L(\theta)
\end{equation}
and the chain rule. Actually, we have
\begin{align}
\norm{\Lie{\omega}(J\circ K)}_\rho \leq {} & 
    \cteDJ \CLieK =: \CLieJ \,, \label{eq:CLieJ} \\
\norm{\Lie{\omega}(G\circ K)}_\rho \leq {} & 
    \cteDG \CLieK =: \CLieG \,, \label{eq:CLieG} \\
\norm{\Lie{\omega}(\tilde \Omega \circ K)}_\rho \leq {} & 
    \cteDtOmega \CLieK =: \CLietOmega \,, \label{eq:CLietOmega}
\end{align}
and then, using \eqref{eq:CLieJ}, the expression \eqref{eq:LieNO} is controlled as
\begin{equation}\label{eq:CLieNO}
\norm{\Lie{\omega}N^0}_{\rho-\delta} \leq \CLieJ \CL + \cteJ \CLieL =: \CLieNO\,.
\end{equation}

Before controlling $\Lie{\omega}B(\theta)$ and $\Lie{\omega}A(\theta)$, we 
recall the notation for $G_L(\theta)$ and $\tilde \Omega_L(\theta)$, given by \eqref{eq:def:GL}
and \eqref{eq:def:tOmegaL} respectively,
and we control the action of $\Lie{\omega}$ on these
objects. For example, we have
\begin{align*}
\Lie{\omega} G_L(\theta) = {} & \Lie{\omega}L(\theta)^\top G(K(\theta)) L(\theta) \\
& + L(\theta)^\top \Lie{\omega}G(K(\theta)) L(\theta) 
+ L(\theta)^\top G(K(\theta)) \Lie{\omega}L(\theta)
\end{align*}
and, using \eqref{eq:CLieG}, we obtain the estimate
\begin{equation}\label{eq:CLieGL}
\norm{\Lie{\omega} G_L}_{\rho-\delta} \leq 
\CLieLT \cteG \CL + \CLT \CLieG \CL + \CLT \cteG \CLieL =:
\CLieGL \,.
\end{equation} 
Analogously, using \eqref{eq:CLietOmega}, we obtain the estimate
\begin{equation}\label{eq:CLietOmegaL}
\norm{\Lie{\omega} \tilde \Omega_L}_{\rho-\delta} \leq 
\CLieLT \ctetOmega \CL + \CLT \CLietOmega \CL + \CLT \ctetOmega \CLieL =:
\CLietOmegaL\,.
\end{equation} 

Now we obtain a suitable expression for $\Lie{\omega}B(\theta)$. To this end, we compute
\[
O_n = \Lie{\omega}(B(\theta)^{-1} B(\theta)) = \Lie{\omega}(B(\theta)^{-1}) B(\theta) + B(\theta)^{-1} \Lie{\omega}(B(\theta))\,,
\]
which, recalling \eqref{eq:B},
yields the following expression
\[
\Lie{\omega}B(\theta) 
=  -B(\theta) \Lie{\omega}(L(\theta)^\top G(K(\theta)) L(\theta)) B(\theta) 
=  -B(\theta) \Lie{\omega} G_L(\theta) B(\theta) \,.
\]
Then, 
using \eqref{eq:CLieGL},
we get the estimate
\begin{equation}\label{eq:CLieB}
\norm{\Lie{\omega} B}_{\rho-\delta} \leq (\sigmaB)^2 \CLieGL = :\CLieB \,.
\end{equation} 
Since $B(\theta)$
is symmetric, we also have $\norm{\Lie{\omega} B^\top}_{\rho-\delta} \leq \CLieB$.

Finally, using the above notation, we expand $\Lie{\omega}A(\theta)$ as
\begin{align*}
\Lie{\omega}A(\theta) 
= {} & -\frac12\Lie{\omega} B(\theta)^\top \tilde \Omega_L(\theta) B(\theta) -\frac12 B(\theta)^\top\Lie{\omega} \tilde \Omega_L(\theta)  B(\theta) \\
     & -\frac12 B(\theta)^\top \tilde \Omega_L(\theta) \Lie{\omega}B(\theta) \,,
\end{align*}
which, using the constants \eqref{eq:CtOmegaL} and \eqref{eq:CLietOmegaL}, yields the following estimate
\begin{equation}\label{eq:CLieA}
\norm{\Lie{\omega} A}_{\rho-\delta} \leq 
\CLieB \CtOmegaL \sigmaB + \frac{1}{2} (\sigmaB)^2 \CLietOmegaL =:
\CLieA\,.
\end{equation}

With the above objects, we can control \eqref{eq:LieN:expanded} as
follows
\begin{equation}\label{eq:CLieN}
\norm{\Lie{\omega} N}_{\rho-\delta} \leq  \CLieL \CA + \CL \CLieA + 
\CLieNO \sigmaB + \CNO \CLieB =:  \CLieN\,.
\end{equation}
This estimate will be used later in the proof of Lemma \ref{lem:KAM:inter:integral}.
Now, we could use \eqref{eq:CLieN} in equation \eqref{eq:line1} to
obtain
\[
\norm{\Loper{N}}_{\rho-\delta} \leq \cteDXH \CN + \CLieN\,.
\]
However, we obtain a sharper estimate observing that
\begin{align*}
\Loper{N}(\theta) = {} &
\Loper{L}(\theta)  A(\theta)
+ \Dif X_\H (K(\theta)) N^0(\theta) B(\theta) \\
& 
+  L(\theta) \Lie{\omega}A(\theta)
+ \Lie{\omega} N^0(\theta) B(\theta) + N^0(\theta) \Lie{\omega}B(\theta)\,, 
\end{align*}
where
\begin{align*}
\Loper{L}(\theta) = {} &
\Dif X_\H (K(\theta)) L(\theta) + \Lie{\omega}L (\theta) \\
= {} &
\begin{pmatrix}
\Dif E(\theta) & \Dif X_p(K(\theta)) [E(\theta)]
\end{pmatrix} \,.
\end{align*}
This last expression follows using the previous formula for
$\Lie{\omega}L(\theta)$ and the fact that the vector field $X_{\H}$
commutes with the fields $X_{p_i}$ for $1\leq i \leq n-d$.

Then, the objects $\Loper{L}(\theta)$ and $\Loper{L}(\theta)^\top$ are controlled as follows:
\begin{align}
\norm{\Loper{L}}_{\rho-\delta} \leq {} & \frac{d + \cteDXp \delta}{\delta} \norm{E}_\rho =: \frac{\CLoperL}{\delta} \norm{E}_\rho \,, \label{eq:CLoperL} \\
\norm{\Loper{L}^\top}_{\rho-\delta} \leq {} & \frac{\max\{2n\, , \, \cteDXpT \delta\}}{\delta} \norm{E}_\rho =: \frac{\CLoperLT}{\delta} \norm{E}_\rho \,, \label{eq:CLoperLT}
\end{align}
and, using again the smallness condition \eqref{eq:fake:cond}, we obtain
\begin{align}
\norm{\Loper{N}}_{\rho-\delta} \leq {} &
\CLoperL \cauxT \CA+\cteDXH \CNO \sigmaB + \CL \CLieA + 
\CLieNO \sigmaB + \CNO \CLieB \nonumber \\
 =: {} & \CLoperN \,. \label{eq:CLoperN} 
\end{align}

Finally, the torsion matrix 
satisfies
\begin{equation}\label{eq:CT}
\norm{T}_{\rho-\delta} \leq \CNT \cteOmega \CLoperN =: \CT \,,
\end{equation}
which completes the proof.
\end{proof}

\begin{remark}
The bound $\CT$ of Lemma~\ref{lem:twist} could be improved for the particular problem
at hand, since the expression for $T(\theta)$ can in many cases be obtained explicitly and may
have cancellations. 
\end{remark}

\subsection{Approximate reducibility}

A crucial step in the proofs of the KAM theorems is the resolution
of the linearized equation arising from the application of Newton
method. The resolution is based on the (approximate) reduction of
such linear system into a simpler form, in particular, block
triangular form. This is the content of the following lemma. 

\begin{lemma}\label{lem:reduc}
Let us consider the setting of 
Theorem \ref{theo:KAM} or
Theorem \ref{theo:KAM:iso}. Then, the map
$P:\TT^d \to \RR^{2n \times 2n}$, characterized in Lemma \ref{lem:sympl},
approximately reduces the linearized equation associated with the vector field
 $\Dif X_\H \circ K$ to a
block-triangular matrix, i.e. the error map
\begin{equation}\label{eq:Ered}
\Ered (\theta) :=
-\Omega_0 P(\theta)^\top \Omega(K(\theta))
\left(
\Dif X_\H(K(\theta)) P(\theta)
+\Lie{\omega}P(\theta)
\right) - \Lambda(\theta)\,,
\end{equation}
with
\begin{equation}\label{eq:Lambda}
\Lambda(\theta)
= \begin{pmatrix}
O_n &  T(\theta) \\ 
O_n  & O_n
\end{pmatrix}
\end{equation}
and $T(\theta)$ is given by \eqref{eq:T},
is small in the following sense:
\[
\norm{\Ered}_{\rho-2\delta} \leq \frac{\Cred}{\gamma \delta^{\tau+1}}
\norm{E}_\rho\,,
\]
where the constant $\Cred$ is provided in Table \ref{tab:constants:all}.
\end{lemma}
\begin{proof}
Using the notation in \eqref{eq:Loper}
we write the block components of \eqref{eq:Ered},
denoted as $\Ered^{i,j}(\theta)$, as follows:
\begin{align}
\Ered^{1,1}(\theta) = {} & 
N(\theta)^\top \Omega(K(\theta)) \Loper{L}(\theta) \,, \label{eq:Ered11} \\
\Ered^{1,2}(\theta) = {} & 
N(\theta)^\top \Omega(K(\theta)) \Loper{N}(\theta) - T(\theta) = O_n \,,
\label{eq:Ered12} \\
\Ered^{2,1}(\theta) = {} & 
-L(\theta)^\top \Omega(K(\theta)) \Loper{L}(\theta) \,, 
\label{eq:Ered21} \\
\Ered^{2,2}(\theta) = {} & 
-L(\theta)^\top \Omega(K(\theta)) \Loper{N}(\theta) \,.
\nonumber
\end{align}
Notice that we have 
used
the definition of $T(\theta)$ to see that \eqref{eq:Ered12} vanishes.
To gather a suitable expression for $\Ered^{2,2}(\theta)$, we apply $\Lie{\omega}$
at both sides of the expression obtained in \eqref{eq:LON}:
\[
\Lie{\omega}(L(\theta)^\top \Omega(K(\theta))) N(\theta)+
L(\theta)^\top \Omega(K(\theta)) \Lie{\omega}N(\theta) =
\Lie{\omega} (\Elag(\theta) A(\theta))\,.
\]
Then, 
introducing this expression into $\Ered^{2,2}(\theta)$,
using \eqref{eq:invE} and the geometric property \eqref{eq:prop2killSK}, we obtain
\begin{align}
\Ered^{2,2}(\theta) = {} & -L(\theta)^\top \Omega(K(\theta)) \Dif X_\H(K(\theta))N(\theta)
+ \Lie{\omega}(L(\theta)^\top \Omega(K(\theta))) N(\theta) \nonumber \\
 & - \Lie{\omega} (\Elag(\theta) A(\theta)) \nonumber \\
= {} & 
L(\theta)^\top (\Dif \Omega(K(\theta)) [E(\theta)]) N(\theta)
+\Loper{L}(\theta)^\top \Omega(K(\theta))N(\theta) 
- \Lie{\omega} (\Elag(\theta) A(\theta)) \,.
\label{eq:Ered22}
\end{align}

At this point, we could use Cauchy estimates in the expression
\[
\Lie{\omega} (\Elag(\theta) A(\theta)) = - \Dif (\Elag(\theta) A(\theta)) [\omega]
\]
and obtain an estimate controlled by $\norm{E}_\rho$. However, this
would give a control of the form $\norm{\Ered}_{\rho-3\delta}$, 
and we are interested in keeping the strip of analyticity $\rho-2\delta$.
For this reason, we compute
\begin{equation}\label{eq:LieElagA}
\Lie{\omega} (\Elag(\theta) A(\theta)) =
\Lie{\omega} \Elag(\theta) \ A(\theta)
+
\Elag(\theta) \ \Lie{\omega} A(\theta)\,,
\end{equation}
and consider the block components of
\begin{equation}\label{eq:LieElag}
\Lie{\omega} \Elag (\theta) 
= 
\begin{pmatrix}
\Lie{\omega}\Omega_K(\theta) & 
\Lie{\omega}(\Dif (p(K(\theta))))^\top
\\
-\Lie{\omega}(\Dif (p(K(\theta)))) &
O_{n-d}
\end{pmatrix}\,.
\end{equation}

Then, we control \eqref{eq:LieElag} as follows
\begin{align}
\norm{\Lie{\omega} \Elag}_{\rho-\delta} \leq {} & 
\max\{
\norm{\Lie{\omega} \Omega_K}_{\rho-\delta}+
\norm{\Lie{\omega}(\Dif (p \circ K))^\top}_{\rho-\delta}
\, , \,
\norm{\Lie{\omega}(\Dif (p \circ K))}_{\rho-\delta}
\} \nonumber \\
\leq {} &  
\frac{
\max \{
\CLieOmegaK + \cteDpT \, , \, d \cteDp
\}
}{\delta} 
\norm{E}_\rho =: \frac{\CLieOmegaL}{\delta} 
\norm{E}_\rho 
\label{eq:CLieOmegaL}
\end{align}
where we used estimates \eqref{eq:LDp} and \eqref{eq:Dp} from Lemma~\ref{lem:cons:H:p}.

Finally, we estimate the norms of the block components of $\Ered$
using the expressions \eqref{eq:Ered11}, \eqref{eq:Ered21}, \eqref{eq:Ered22}
and \eqref{eq:LieElagA},
and the previous estimates:
\begin{align}
\norm{\Ered^{1,1}}_{\rho-2\delta} \leq {} &
\norm{N^\top}_\rho \norm{\Omega}_\B \norm{\Loper{L}}_{\rho-\delta}
\leq \frac{\CNT \cteOmega \CLoperL}{\delta} \norm{E}_\rho =: 
\frac{\Creduu}{\delta}\norm{E}_\rho
\,, \label{eq:Creduu} \\
\norm{\Ered^{1,2}}_{\rho-2\delta} = {} & 0 \,, \nonumber \\
\norm{\Ered^{2,1}}_{\rho-2\delta} \leq {} &
\norm{L^\top}_\rho \norm{\Omega}_\B \norm{\Loper{L}}_{\rho-\delta}
\leq \frac{\CLT \cteOmega \CLoperL}{\delta} \norm{E}_\rho 
=:
\frac{\Creddu}{\delta}\norm{E}_\rho
\,, \label{eq:Creddu} \\
\norm{\Ered^{2,2}}_{\rho-2\delta} \leq {} &
\left( \CLT \cteDOmega \CN + \frac{\CLoperLT \cteOmega \CN }{\delta}
+ \frac{\CLieOmegaL \CA}{\delta} +\frac{\Clag \CLieA}{\gamma \delta^{\tau+1}} \right)\norm{E}_\rho \nonumber \\
=: {} & \frac{\Creddd}{\gamma \delta^{\tau+1}} \norm{E}_\rho
 \,.
\label{eq:Creddd}
\end{align}
Then, we end up with
\begin{equation} \label{eq:CEred}
\norm{\Ered}_{\rho-2\delta} \leq 
\frac{\max \{
\Creduu \gamma \delta^\tau
\, , \,
\Creddu \gamma \delta^\tau + \Creddd \} }{\gamma \delta^{\tau+1}}
\norm{E}_\rho 
=: \frac{\Cred}{\gamma \delta^{\tau+1}} \norm{E}_\rho\,,
\end{equation}
thus completing the proof.
\end{proof}

\section{Proof of the ordinary KAM theorem}\label{sec:proof:KAM}

In the section we present a fully detailed proof of Theorem \ref{theo:KAM}.
For convenience, we will start by outlining the scheme used
to correct the parameterization of the torus. That is,
in Section \ref{ssec:qNewton} we discuss the approximate solution of
linearized equations in the symplectic frame constructed in Section 
\ref{ssec:symp}. This establishes a quasi-Newton method to
obtain a solution of the invariance equation. In Section
\ref{ssec:iter:lemmas} we produce quantitative estimates for
the objects obtained when performing one iteration of the
previous procedure. Finally, in Section
\ref{ssec:proof:KAM} we discuss the convergence of the quasi-Newton method.

\subsection{The quasi-Newton method}\label{ssec:qNewton}

As it is usual in the a-posteriori approach to KAM theory, the argument consists
in refining $K(\theta)$ by means of a quasi-Newton method. 
Let us consider the equations associated with the invariance error
\[
E(\theta) =  X_\H(K(\theta)) + \Lie{\omega}K(\theta)\,.
\]
Then, we obtain the new parameterization
$\bar K(\theta)= K(\theta)+\DeltaK(\theta)$
by considering the linearized equation
\begin{equation}\label{eq:lin1}
\Dif X_\H (K(\theta)) \DeltaK(\theta) + \Lie{\omega}\DeltaK(\theta) 
= - E(\theta) \,,
\end{equation}
If we obtain a good enough approximation of the solution
$\DeltaK(\theta)$
of \eqref{eq:lin1},
then $\bar K(\theta)$ provides a parameterization
of an approximately invariant torus of frequency $\omega$,
with a quadratic error in terms of $E(\theta)$. 

To face the linearized equation \eqref{eq:lin1}, 
we resort to the approximately symplectic
frame $P(\theta)$, defined on the full tangent
space, which has been characterized in Section \ref{sec:lemmas}
(see Lemma \ref{lem:sympl}).
In particular, we introduce the linear change
\begin{equation}\label{eq:choice:DK}
\DeltaK (\theta) = P(\theta) \xi(\theta) \,,
\end{equation}
where $\xi(\theta)$ is the new unknown.
Taking into account this expression, 
the linearized equation becomes
\begin{equation} \label{eq:lin:1E} 
\left( \Dif X_\H (K(\theta)) P(\theta) + \Lie{\omega}P(\theta) \right) \xi(\theta)
+P(\theta) \Lie{\omega} \xi(\theta) 
=   - E(\theta) \,, 
\end{equation}
We now multiply both sides of \eqref{eq:lin:1E} by $-\Omega_0 P(\theta)^\top
\Omega(K(\theta))$, and we 
use the
geometric properties in Lemma \ref{lem:sympl} and
Lemma \ref{lem:reduc}, thus obtaining the equivalent equations:
\begin{equation} \label{eq:lin:1Enew}
\begin{split}
\left(\Lambda(\theta)+\Ered(\theta) \right) \xi(\theta)
+ {} & \left( I_{2n} - \Omega_0 \Esym(\theta)  \right) \Lie{\omega} \xi(\theta) \\ 
= {} &
\Omega_0 P(\theta)^\top \Omega(K(\theta)) E(\theta) 
\,,
\end{split}
\end{equation}
where $\Lambda(\theta)$ is the triangular matrix-valued map given
in \eqref{eq:Lambda}.

Then, it turns out that the solution of \eqref{eq:lin:1Enew}
are approximated by the solutions of a triangular system
that requires to solve two cohomological equations of
the form \eqref{eq:calL} consecutively.
Quantitative estimates for the solutions of 
such equations are obtained by applying 
R\"ussmann estimates.
This is summarized
in the following standard statement.

\begin{lemma}[Upper triangular equations]\label{lem:upperT}
Let 
$\omega \in \Dioph{\gamma}{\tau}$
and let us consider a map
$\eta= (\eta^L,\eta^N) : \TT^d \to
\RR^{2n} \simeq \RR^n\times\RR^n$, with components in $\Anal(\TT^d_\rho)$, and a 
map $T : \TT^d \rightarrow \RR^{n\times n}$,
with components in $\Anal(\TT^d_{\rho-\delta})$.
Assume that $T$ satisfies the non-degeneracy condition $\det
\aver{T} \neq 0$ and $\eta$ satisfies the compatibility condition $\aver{\eta^N}=0_n$.
Then, for any $\xi^L_0\in \RR^n$, the system of equations
\begin{equation} \label{eq:lin:last1}
\begin{pmatrix}
O_n & T(\theta) \\
O_n & O_n
\end{pmatrix}
\begin{pmatrix}
\xi^L(\theta) \\
\xi^N(\theta)
\end{pmatrix}
+
\begin{pmatrix}
\Lie{\omega} \xi^L(\theta) \\
\Lie{\omega} \xi^N(\theta)
\end{pmatrix}
=
\begin{pmatrix}
\eta^L(\theta) \\
\eta^N(\theta)
\end{pmatrix}
\end{equation}
has a solution of the form
\begin{align*}
\xi^N(\theta)={} & \xi^N_0 + \R{\omega}(\eta^N(\theta)) \,, \\
\xi^L(\theta)={} & \xi^L_0 +\R{\omega}(\eta^L(\theta) - T(\theta)
\xi^N(\theta)) \,, 
\end{align*}
where
\[
\xi^N_0= \aver{T}^{-1} \aver{\eta^L-T \R{\omega}(\eta^N)} 
\]
and $\R{\omega}$ is given by \eqref{eq:small:formal}. 
Moreover, 
we have the estimates
\begin{align*}
& |\xi_0^N| \leq \Abs{\aver{T}^{-1}}
\Big(
\norm{\eta^L}_\rho
+ \frac{c_R}{\gamma \delta^\tau} \norm{T}_{\rho-\delta} \norm{\eta^N}_\rho
\Big)\, , \\
& \norm{\xi^N}_{\rho-\delta} \leq |\xi_0^N| + \frac{c_R}{\gamma \delta^\tau}
\norm{\eta^N}_\rho\, , \\
& \norm{\xi^L}_{\rho-2\delta} \leq |\xi_0^L| + \frac{c_R}{\gamma \delta^\tau}
\Big(
\norm{\eta^L}_{\rho-\delta} + 
\norm{T}_{\rho-\delta} \norm{\xi^N}_{\rho-\delta}
\Big)\, .
\end{align*}
\end{lemma}

\begin{proof}
This triangular structure is classic in KAM theory and appears in
any Kolmogorov scheme (see e.g. \cite{BennettinGGS84,Llave01,Kolmogorov54}).
The Lemma is directly adapted from Lemma 4.14 in \cite{HaroCFLM16} and the estimates
are directly obtained using Lemma \ref{lem:Russmann}.
\end{proof}

To approximate the solutions of \eqref{eq:lin:1Enew} we will
invoke Lemma \ref{lem:upperT} taking
\begin{equation}\label{eq:eta:corr}
\eta^L(\theta) = - N(\theta)^\top \Omega(K(\theta)) E(\theta)\,,
\qquad
\eta^N(\theta) = L(\theta)^\top \Omega(K(\theta)) E(\theta)\,,
\end{equation}
and $T(\theta)$ given by \eqref{eq:T}. We recall from Lemma \ref{lem:sympl}
that the compatibility condition $\aver{\eta^N}=0_n$ is satisfied.
Note that $\aver{\xi^N}=\xi^N_0$ and
we have the freedom of choosing any
value for $\aver{\xi^L}= \xi^L_0 \in \RR^n$.
For convenience, we will select later
the solution with  $\xi_0^L=0_n$, even though
other choices can be selected according to the context (see
Remark \ref{rem:unicity}).

From Lemma \ref{lem:upperT} we read that 
$\norm{\xi}_{\rho-2\delta}=\cO(\norm{E}_\rho)$ and,
using the geometric properties characterized in Section \ref{sec:lemmas},
we have $\norm{\Ered}_{\rho-2\delta}=\cO(\norm{E}_\rho)$ and $\norm{\Esym}_{\rho-2\delta}=\cO(\norm{E}_\rho)$.
From these estimates, we conclude that the solution of equation
\eqref{eq:lin:1Enew}
is approximated by the solution of
the cohomological equation
\begin{equation}\label{eq:mycohoxi}
\Lambda(\theta) \xi(\theta) + \Lie{\omega} \xi(\theta)=\eta(\theta)\,.
\end{equation}
This, together with other estimates, will be suitably quantified in the next section.

\subsection{One step of the iterative procedure}\label{ssec:iter:lemmas}

In this section we apply one correction of the quasi-Newton method
described in Section \ref{ssec:qNewton} and we obtain sharp quantitative
estimates
for the new approximately invariant torus and related objects. We
set sufficient conditions to preserve the control of the previous
estimates.

\begin{lemma} [The Iterative Lemma in the ordinary case] \label{lem:KAM:inter:integral}
Let us consider the same setting and hypotheses of Theorem \ref{theo:KAM},
and a constant $\cauxT>0$. 
Then, there exist constants
$\CDeltaK$, 
$\CDeltaB$, 
$\CDeltaTOI$ and
$\CE$
such that 
if the inequalities
\begin{equation}\label{eq:cond1:K:iter}
\frac{\hCDelta \norm{E}_\rho}{\gamma^2 \delta^{2\tau+1}} < 1
\qquad\qquad
\frac{\CE \norm{E}_\rho}{\gamma^4 \delta^{4\tau}} < 1
\end{equation}
hold for some $0<\delta< \rho$, where
\begin{equation}\label{eq:mathfrak1}
\begin{split}
\hCDelta := \max \bigg\{ & \frac{\gamma^2 \delta^{2\tau}}{\cauxT}
\, , \, 
2 \Csym \gamma \delta^{\tau} 
\, , \,
\frac{ d \CDeltaK }{\sigmaDK - \norm{\Dif K}_\rho} 
\, , \,
\frac{ 2n  \CDeltaK }{\sigmaDKT - \norm{(\Dif K)^\top}_\rho}
\, , \,
\\
&  \frac{\CDeltaB}{\sigmaB - \norm{B}_\rho} 
\, , \,
 \frac{\CDeltaTOI }{\sigmaT - \abs{\aver{T}^{-1}}} 
\, , \,
 \frac{\CDeltaK \delta}{\dist (K(\TT^d_\rho),\partial B)}
 \bigg\} \,,
\end{split}
\end{equation}
then we have an approximate torus of the same frequency $\omega$
given by $\bar K=K+\DeltaK$, with components in $\Anal(\TT^d_{\rho-2\delta})$, that defines new objects $\bar B$
and $\bar T$ (obtained replacing $K$ by $\bar K$) satisfying
\begin{align}
& \norm{\Dif \bar K}_{\rho-3\delta} < \sigmaDK \,, \label{eq:DK:iter1} \\
& \norm{(\Dif \bar K)^\top}_{\rho-3\delta} < \sigmaDKT \,, \label{eq:DKT:iter1} \\
& \norm{\bar B}_{\rho-3\delta} < \sigmaB \,, \label{eq:B:iter1} \\
& \abs{\aver{\bar T}^{-1}} < \sigmaT \,, \label{eq:T:iter1} \\
& \dist(\bar K(\TT^d_{\rho-2\delta}),\partial \B) >0 \,, \label{eq:distB:iter1} 
\end{align}
and
\begin{align}
& \norm{\bar K-K}_{\rho-2 \delta} < \frac{\CDeltaK}{\gamma^2 \delta^{2\tau}}
\norm{E}_\rho\,, \label{eq:est:DeltaK} \\
& \norm{\bar B-B}_{\rho-3\delta} < \frac{\CDeltaB}{\gamma^2 \delta^{2 \tau+1}}
\norm{E}_\rho\,, \label{eq:est:DeltaB} \\
& \abs{\aver{\bar T}^{-1}-\aver{T}^{-1} } < \frac{\CDeltaTOI}{\gamma^2 \delta^{2 \tau+1}}
\norm{E}_\rho\,, \label{eq:est:DeltaT}
\end{align}
The new error of invariance is given by
\[
\bar E(\theta) =  X_\H (\bar K(\theta)) + \Lie{\omega} \bar K(\theta) \,,
\]
and satisfies
\begin{equation}\label{eq:E:iter1}
\norm{\bar E}_{\rho-2\delta} < \frac{\CE}{\gamma^4
\delta^{4\tau}}\norm{E}_\rho^2\,.
\end{equation}
The above constants are collected in Table~\ref{tab:constants:all:2}.
\end{lemma}

\begin{proof} 
This result requires rather cumbersome computations, so we divide the proof
into several steps.

\bigskip

%%%%%%%%%%%%%%%%%%%%%%%%%%%%%%%%%%%%%%%%%%%%%%%%%%%%%%%%%%
% Step 1: Control of the new parameterization
%%%%%%%%%%%%%%%%%%%%%%%%%%%%%%%%%%%%%%%%%%%%%%%%%%%%%%%%%%

\paragraph
{\emph{Step 1:
Control of the new parameterization}.}
We start by considering the new parameterization $\bar K (\theta)= K(\theta) + \DeltaK(\theta)$ obtained 
from the 
system \eqref{eq:lin:last1}, 
with $\eta(\theta)$ given by \eqref{eq:eta:corr}. We choose the solution that satisfies $\xi_0^L=0$.
Using the estimates obtained in Section \ref{sec:lemmas} we have
\[
\norm{\eta^L}_\rho \leq \CNT \cteOmega \norm{E}_\rho\,,
\qquad
\norm{\eta^N}_\rho \leq \CLT \cteOmega \norm{E}_\rho\,.
\]
In order to invoke Lemma \ref{lem:twist}
(we must fulfill condition \eqref{eq:fake:cond})
we have included the inequality
\begin{equation}\label{eq:ingredient:iter:1}
\frac{\norm{E}_\rho}{\delta} < \cauxT
\end{equation}
into Hypothesis \eqref{eq:cond1:K:iter} (this corresponds to the first term in \eqref{eq:mathfrak1}).
Hence, combining 
Lemma \ref{lem:twist}
and 
Lemma \ref{lem:upperT},
we obtain estimates for the solution
of the cohomological equations
(we recall that $\xi_0^L=0_n$)
\begin{align}
& \abs{\xi^N_0} \leq \abs{\aver{T}^{-1}}
\Big(
\norm{\eta^L}_{\rho} + \frac{c_R}{\gamma \delta^\tau} \norm{T}_{\rho-\delta} \norm{\eta^N}_\rho
\Big) \nonumber \\
& \qquad \leq \sigmaT \Big(
\CNT \cteOmega + \frac{c_R}{\gamma \delta^\tau} \CT \CLT \cteOmega
\Big)
\norm{E}_\rho =: \frac{\CxiNO}{\gamma \delta^\tau} \norm{E}_\rho\,, \label{eq:CxiNO} \\
& \norm{\xi^N}_{\rho-\delta} \leq \abs{\xi^N_0}+ \frac{c_R}{\gamma \delta^\tau} \norm{\eta^N}_\rho \nonumber \\
& \qquad \leq 
\frac{\CxiNO}{\gamma \delta^\tau} \norm{E}_\rho 
+ \frac{c_R}{\gamma \delta^\tau}\CLT \cteOmega \norm{E}_\rho
 =: 
\frac{\CxiN}{\gamma \delta^\tau} \norm{E}_\rho \,, \label{eq:CxiN} \\
& \norm{\xi^L}_{\rho-2\delta} \leq \abs{\xi_0^L} + \frac{c_R}{\gamma \delta^\tau}
\Big(
\norm{\eta^L}_{\rho} + \norm{T}_{\rho-\delta} \norm{\xi^N}_{\rho-\delta}
\Big) \nonumber \\
& \qquad \leq \frac{c_R}{\gamma \delta^\tau} \Big(
\CNT \cteOmega + \CT \frac{\CxiN}{\gamma \delta^\tau}
\Big)
\norm{E}_\rho =: \frac{\CxiL}{\gamma^2 \delta^{2\tau}} \norm{E}_\rho\,. \label{eq:CxiL} 
\end{align}
The norm of the full vector $\xi(\theta)$, which satisfies \eqref{eq:mycohoxi}, is controlled as
\begin{align}
\norm{\xi}_{\rho-2\delta} \leq {} &
\max\{ \norm{\xi^L}_{\rho-2\delta}\, , \, \norm{\xi^N}_{\rho-\delta}\} \nonumber \\
\leq {} &
\frac{\max\{\CxiL \, , \, \CxiN  \gamma \delta^\tau\}}{\gamma^2 \delta^{2\tau}} \norm{E}_\rho
=: \frac{\Cxi}{\gamma^2 \delta^{2\tau}} \norm{E}_\rho\,. \label{eq:Cxi}
\end{align}

The new parameterization $\bar K(\theta)$ and the related objects are
controlled using standard computations. Estimate \eqref{eq:est:DeltaK} follows
directly from 
\[
\bar K(\theta)-K(\theta) = \DeltaK(\theta) = P(\theta) \xi(\theta) = L(\theta) \xi^L(\theta) + N (\theta) \xi^N(\theta)\,,
\]
that is, using estimates in \eqref{eq:CxiN} and \eqref{eq:CxiL},
we obtain
\begin{equation}\label{eq:CDeltaK}
\norm{\bar K - K}_{\rho-2\delta} 
\leq \frac{\CL \CxiL + \CN \CxiN \gamma \delta^\tau}{\gamma^2 \delta^{2\tau}} \norm{E}_\rho =:
\frac{\CDeltaK}{\gamma^2 \delta^{2\tau}} \norm{E}_\rho\,.
\end{equation}

To complete this step, we check that $\bar K(\theta)$ remains inside the
domain $\B$ where the global objects are defined. This is important
because we need to estimate the new error $\bar E(\theta)$ before
controlling the remaining geometrical objects
For this, we observe that 
\begin{align}
\dist(\bar K(\TT^d_{\rho-2\delta}),\partial \B) \geq {} &  \dist(K(\TT^d_\rho),\partial \B) - \norm{\DeltaK}_{\rho-2\delta} \nonumber \\
\geq {} &  \dist(K(\TT^d_\rho),\partial \B) - \frac{\CDeltaK}{\gamma^2 \delta^{2\tau}} \norm{E}_\rho>0\,, \label{eq:ingredient:iter:2}
\end{align}
where the last inequality follows from
Hypothesis \eqref{eq:cond1:K:iter} (this corresponds to the seventh term in \eqref{eq:mathfrak1}).
We have obtained the control in \eqref{eq:distB:iter1}.
\bigskip

%%%%%%%%%%%%%%%%%%%%%%%%%%%%%%%%%%%%%%%%%%%%%%%%%%%%%%%%%%
% Step 2: Control of the new error of invariance
%%%%%%%%%%%%%%%%%%%%%%%%%%%%%%%%%%%%%%%%%%%%%%%%%%%%%%%%%%

\paragraph
{\emph{Step 2: Control of the new error of invariance}.}

To control the error of invariance of the corrected parameterization $\bar K$, 
we first consider the error in the solution of the linearized equation 
\eqref{eq:lin:1Enew}, that is, we control the quadratic terms
that are neglected when considering the equation \eqref{eq:mycohoxi}:
\[
\Elin (\theta) = \Ered(\theta) \xi(\theta) - \Omega_0 \Esym(\theta) \Lie{\omega}\xi(\theta)\,.
\]
The term $\Lie{\omega} \xi(\theta)$ is controlled using that $\xi(\theta)$ is precisely the solution
of the cohomological equation \eqref{eq:mycohoxi} :
\begin{align}
\norm{\Lie{\omega} \xi^N}_{\rho}
= {} & 
\norm{\eta^N}_{\rho}
\leq \CLT \cteOmega \norm{E}_\rho =: \CLiexiN \norm{E}_\rho \,,
\label{eq:CLiexiN}
\\
\norm{\Lie{\omega} \xi^L}_{\rho-\delta}
={} &
\norm{\eta^L - T \xi^N}_{\rho-\delta}
\leq \left(\CNT \cteOmega + \CT \frac{\CxiN}{\gamma \delta^\tau}\right)\norm{E}_\rho \nonumber \\
=: {} & \frac{\CLiexiL}{\gamma \delta^{\tau}} \norm{E}_\rho \,,
 \label{eq:CLiexiL} \\
\norm{\Lie{\omega} \xi}_{\rho-\delta} 
 \leq {} & \max\left(\frac{\CLiexiL}{\gamma \delta^{\tau}}, \CLiexiN \right) \norm{E}_\rho =: \frac{\CLiexi}{\gamma \delta^{\tau}} \norm{E}_\rho\,.
 \label{eq:CLiexi}
\end{align}

Hence, we control $\Elin(\theta)$ by 
\begin{align}
\norm{\Elin}_{\rho-2\delta} \leq {} & 
\norm{\Ered}_{\rho-2\delta} \norm{\xi}_{\rho-2\delta} +
\cteOmega \norm{\Esym}_{\rho-2\delta} \norm{\Lie{\omega}\xi}_{\rho-2 \delta} \nonumber \\
\leq {} &
\frac{\Cred}{\gamma \delta^{\tau+1}}
\frac{\Cxi}{\gamma^2 \delta^{2 \tau}}
\norm{E}_\rho^2
+
\cteOmega
\frac{\Csym}{\gamma \delta^{\tau+1}}
\frac{\CLiexi}{\gamma \delta^{\tau}}
\norm{E}_\rho^2
=: \frac{\Clin}{\gamma^3 \delta^{3\tau+1}} \norm{E}_\rho^2\,.
\label{eq:Clin}
\end{align}
We remark that this last estimate can be improved
by considering the components of $\xi(\theta)=(\xi^L(\theta),\xi^N(\theta))$
separately, thus obtaining a divisor $\gamma^2 \delta^{2\tau+1}$ in \eqref{eq:Clin}. 
Nevertheless, this improvement is irrelevant for practical purposes.

After performing the correction, the error of invariance associated with
the new parameterization is given by
\begin{equation}\label{eq:new:E:comp}
\begin{split}
\bar E(\theta) = {} & X_\H(K(\theta)+\DeltaK(\theta)) + \Lie{\omega}K(\theta) + \Lie{\omega}\DeltaK(\theta) \\
= {} & X_{\H}(K(\theta))+\Dif X_\H (K(\theta)) \DeltaK(\theta) + \Lie{\omega}K(\theta) + \Lie{\omega}\DeltaK(\theta)
+\Delta^2 X(\theta) \\
= {} & \Dif X_\H (K(\theta)) \DeltaK(\theta) + \Lie{\omega}\DeltaK(\theta) + E(\theta)
+\Delta^2 X(\theta) \\
= {} & \left( \Dif X_\H (K(\theta)) P(\theta) + \Lie{\omega}P(\theta) \right) \xi(\theta)
+P(\theta) \Lie{\omega} \xi(\theta) +  E(\theta) + \Delta^2X(\theta) \\
= {} & (-\Omega_0 P(\theta)^\top \Omega(K(\theta)))^{-1} \Elin(\theta) + \Delta^2X(\theta) \\
= {} & P(\theta) (I_{2n}-\Omega_0 \Esym(\theta))^{-1} \Elin(\theta)+\Delta^2 X(\theta)\,,
\end{split}
\end{equation}
where
\begin{equation}\label{eq:Delta2X}
\begin{split}
\Delta^2 X(\theta) = {} & X_\H(K(\theta)+\DeltaK(\theta))-X_\H(K(\theta))-\Dif X_\H(K(\theta)) \DeltaK(\theta) \\
= {} & \int_0^1 (1-t) \Dif^2 X_\H(K(\theta)+t \DeltaK(\theta)) [\DeltaK(\theta),\DeltaK(\theta)] \dif t\,,
\end{split}
\end{equation}
and we used \eqref{eq:Esym}.
Notice that the above error function is well defined, due to the computations
in \eqref{eq:ingredient:iter:2}, and we
estimate its norm as follows
\[
\norm{\bar E}_{\rho-2\delta} \leq \norm{P}_{\rho-2\delta} \norm{(I-\Omega_0 \Esym)^{-1}}_{\rho-2\delta} \norm{\Elin}_{\rho-2\delta}+\norm{\Delta^2 X}_{\rho-2\delta}\,.
\]
Then, using a Neumann series argument, we obtain
\[
 \norm{(I-\Omega_0 \Esym)^{-1}}_{\rho-2\delta} \leq \frac{1}{1-\norm{\Omega_0 \Esym}_{\rho-2\delta}} <2\,,
\]
where we used the inequality
\[
\frac{\Csym}{\gamma \delta^{\tau+1}}\norm{E}_\rho \leq \frac{1}{2}\ ,
\]
that corresponds to the second term in \eqref{eq:mathfrak1}
(Hypothesis \eqref{eq:cond1:K:iter}).
Putting together the above estimates, and applying the mean value theorem to control $\Delta^2 X(\theta)$, we obtain
\begin{equation}\label{eq:CE}
\norm{\bar E}_{\rho-2\delta} \leq \left(\frac{ 2(\CL+\CN)\Clin}{\gamma^3 \delta^{3\tau+1}} + \frac{1}{2} \cteDDXH \frac{(\CDeltaK)^2}{\gamma^4 \delta^{4\tau}} \right)
\norm{E}_\rho^2 =:
\frac{\CE}{\gamma^4 \delta^{4\tau}} \norm{E}_\rho^2\,.
\end{equation}
We have obtained the estimate \eqref{eq:E:iter1}.
Notice that the second assumption in \eqref{eq:cond1:K:iter} 
and \eqref{eq:ingredient:iter:1}
imply that
\begin{equation}\label{eq:barEvsE}
\norm{\bar E}_{\rho-2\delta} < \norm{E}_\rho< \delta \cauxT\,.
\end{equation}
This will be used in Step 6.

\bigskip

%%%%%%%%%%%%%%%%%%%%%%%%%%%%%%%%%%%%%%%%%%%%%%%%%%%%%%%%%%
% Step 3: Control of the new frame L(theta)
%%%%%%%%%%%%%%%%%%%%%%%%%%%%%%%%%%%%%%%%%%%%%%%%%%%%%%%%%%

\paragraph
{\emph{Step 3: Control of the new frame $L(\theta)$}.}

Combining \eqref{eq:CDeltaK} with Cauchy estimates, we obtain the control \eqref{eq:DK:iter1}:
\begin{equation} \label{eq:ingredient:iter:4}
\norm{\Dif \bar K}_{\rho-3\delta} \leq 
\norm{\Dif K}_{\rho} +
\norm{\Dif \DeltaK}_{\rho-3\delta} \leq 
\norm{\Dif K}_{\rho} +
\frac{d \CDeltaK}{\gamma^2 \delta^{2\tau+1}} \norm{E}_\rho < \sigmaDK \,,
\end{equation}
where the last inequality follows from
Hypothesis \eqref{eq:cond1:K:iter} (this corresponds to the third term in \eqref{eq:mathfrak1}).
The control \eqref{eq:DKT:iter1}
on the transposed object is analogous
\begin{equation} \label{eq:ingredient:iter:5}
\norm{(\Dif \bar K)^\top}_{\rho-3\delta} \leq 
\norm{(\Dif K)^\top}_{\rho} +
\frac{2n \CDeltaK}{\gamma^2 \delta^{2\tau+1}} \norm{E}_\rho < \sigmaDKT \,,
\end{equation}
where the last inequality follows from
Hypothesis \eqref{eq:cond1:K:iter} (this corresponds to the fourth term in \eqref{eq:mathfrak1}).

After obtaining the estimates \eqref{eq:ingredient:iter:4} and \eqref{eq:ingredient:iter:5},
it is clear that
\[
\norm{\bar L}_{\rho-3\delta} \leq \CL\,,
\qquad
\norm{\bar L^\top}_{\rho-3\delta} \leq \CLT\,.
\]
Indeed, we can control the norm of the corresponding corrections
using Cauchy estimates, the mean value theorem and estimate \eqref{eq:CDeltaK}:
\begin{align}
\norm{\bar L-L}_{\rho-3 \delta} \leq {} & 
\frac{d \CDeltaK}{\gamma^2 \delta^{2\tau+1}} \norm{E}_\rho 
+  
\frac{\cteDXp \CDeltaK}{\gamma^2 \delta^{2\tau}} \norm{E}_\rho =: \frac{\CDeltaL}{\gamma^2 \delta^{2\tau+1}} \norm{E}_\rho\,, \label{eq:CDeltaL} \\
\norm{\bar L^\top-L^\top}_{\rho-3 \delta} \leq {} & 
\frac{\CDeltaK \max\{ 2n \, , \, \cteDXpT \delta\}}{\gamma^2 \delta^{2\tau+1}} \norm{E}_\rho 
=: \frac{\CDeltaLT}{\gamma^2 \delta^{2\tau+1}} \norm{E}_\rho\,. \label{eq:CDeltaLT}
\end{align}

\bigskip

%%%%%%%%%%%%%%%%%%%%%%%%%%%%%%%%%%%%%%%%%%%%%%%%%%%%%%%%%%
% Step 4: Control of the new transversality condition
%%%%%%%%%%%%%%%%%%%%%%%%%%%%%%%%%%%%%%%%%%%%%%%%%%%%%%%%%%

\paragraph
{\emph{Step 4: Control of the new transversality condition}.}
To control $\bar B$ we use Lemma \ref{lem:aux} taking
\begin{align*}
M(\theta) = {} & G_L(\theta) = L(\theta)^\top G(K(\theta)) L(\theta) \,,\\
\bar M(\theta) = {} & G_{\bar L}(\theta) = \bar L(\theta)^\top G(\bar K(\theta)) \bar L(\theta) \,,
\end{align*}
where we have used
the notation introduced in \eqref{eq:def:GL}.
First, we compute
\begin{align}
\norm{G \circ \bar K-G \circ K}_{\rho-2\delta} \leq {} & \norm{\Dif G}_{\B} \norm{\bar K - K}_{\rho-2\delta}
\leq \frac{\cteDG \CDeltaK}{\gamma^2 \delta^{2\tau}} \norm{E}_\rho \nonumber \\
=: {} & \frac{\CDeltaG}{\gamma^2 \delta^{2\tau}} \norm{E}_\rho\,, \label{eq:CDeltaG}
\end{align}
and
\begin{align}
\norm{G_{\bar L}-G_L}_{\rho-3\delta} 
\leq {} & \norm{\bar L^\top (G \circ \bar K) \bar L - \bar L^\top (G \circ \bar K) L}_{\rho-3\delta}  \nonumber\\
        & +\norm{\bar L^\top (G \circ \bar K) L - \bar L^\top (G \circ K) L}_{\rho-3\delta} \nonumber \\
        &+\norm{\bar L^\top (G \circ K) L - L^\top (G \circ K) L}_{\rho-3\delta} \nonumber \\
\leq {} & \norm{\bar L^\top}_\rho \norm{G}_\B \norm{\bar L- L}_{\rho-3\delta}  \nonumber\\
        & + \norm{\bar L^\top}_\rho \norm{G \circ \bar K-G \circ K}_{\rho-3\delta} \norm{L}_\rho \nonumber \\
        &+\norm{\bar L^\top - L^\top}_{\rho-3\delta}\norm{G}_\B \norm{L}_\rho \nonumber \\ 
\leq {} & \frac{\CLT \cteG \CDeltaL + \CLT \CDeltaG \CL \delta + \CDeltaLT \cteG \CL }{\gamma^2 \delta^{2\tau+1}} \norm{E}_\rho \nonumber \\
  =: {} & \frac{\CDeltaGL}{\gamma^2 \delta^{2\tau+1}} \norm{E}_\rho \,. \label{eq:CDeltaGL} 
\end{align}
Then, we introduce the constant
\[
\CDeltaB := 2(\sigmaB)^2 \CDeltaGL
\]
and check the condition \eqref{eq:lem:aux} in Lemma \ref{lem:aux}:
\begin{align}
\frac{2 (\sigmaB)^2 \norm{G_{\bar L}-G_L}_{\rho-3\delta}}{\sigmaB-\norm{B}_\rho} \leq {} &
\frac{2 (\sigmaB)^2 \CDeltaGL}{\sigmaB-\norm{B}_\rho} \frac{\norm{E}_\rho}{\gamma^2 \delta^{2\tau+1}} 
\nonumber \\
= {} & 
\frac{\CDeltaB}{\sigmaB-\norm{B}_\rho} \frac{\norm{E}_\rho}{\gamma^2 \delta^{2\tau+1}}
< 1\,,
\label{eq:ingredient:iter:6}
\end{align}
where the last inequality follows from
Hypothesis \eqref{eq:cond1:K:iter} (this corresponds to the fifth term in \eqref{eq:mathfrak1}).
Hence, by invoking Lemma \ref{lem:aux}, we conclude that
\begin{equation}\label{eq:CDeltaB}
\norm{\bar B}_{\rho-3\delta} < \sigmaB\,, \qquad
\norm{\bar B-B}_{\rho-3\delta} \leq \frac{2 (\sigmaB)^2 \CDeltaGL}{\gamma^2 \delta^{2\tau+1}}
\norm{E}_\rho 
= \frac{\CDeltaB}{\gamma^2 \delta^{2\tau+1}}
\norm{E}_\rho\,,
\end{equation}
and so, we obtain the estimates \eqref{eq:B:iter1} and \eqref{eq:est:DeltaB} on the new object.

\bigskip

%%%%%%%%%%%%%%%%%%%%%%%%%%%%%%%%%%%%%%%%%%%%%%%%%%%%%%%%%%
% Step 5: Control of the new normal frame
%%%%%%%%%%%%%%%%%%%%%%%%%%%%%%%%%%%%%%%%%%%%%%%%%%%%%%%%%%

\paragraph
{\emph{Step 5: Control of the new frame $N(\theta)$}.}
To control the new adapted normal frame $\bar N(\theta)$, it is convenient to
recall the following notation:
\begin{align*}
N(\theta) = {} &  L(\theta) A(\theta) + N^0(\theta) B(\theta)\,, \\
N^0(\theta) = {} & J(K(\theta)) L(\theta) \,,\\
A(\theta) = {} & -\tfrac{1}{2} (B(\theta)^\top L(\theta)^\top \tilde \Omega(K(\theta))
L(\theta) B(\theta)) \,, \\
B(\theta) = {} & (L(\theta)^\top G(K(\theta)) L(\theta))^{-1}\,,
\end{align*}
where, as usual, the new objects $\bar N(\theta)$, $\bar A(\theta)$, $\bar N^0(\theta)$ and $\bar B(\theta)$
are obtained by replacing $K(\theta)$ by $\bar K(\theta)$. Note that the object
$\bar B(\theta)$ has been controlled in Step~4.

Now, we recall the notation introduced in \eqref{eq:def:tOmegaL} and
reproduce the computations in \eqref{eq:CDeltaG} and \eqref{eq:CDeltaGL} for the matrix functions 
\begin{align*}
\tilde \Omega_L(\theta) = {} & L(\theta)^\top \tilde \Omega(K(\theta)) L(\theta) \,,\\
\tilde \Omega_{\bar L}(\theta) = {} & \bar L(\theta)^\top \tilde  \Omega(\bar K(\theta)) \bar L(\theta) \,,
\end{align*}
thus obtaining
\begin {equation} \label{eq:CDeltatOmega}
\norm{\tilde \Omega \circ \bar K-\tilde \Omega \circ K}_{\rho -2\delta} \leq \frac{\cteDtOmega \CDeltaK}{\gamma^2 \delta^{2\tau}} \norm{E}_\rho =:\frac{\CDeltatOmega}{\gamma^2 \delta^{2\tau}} \norm{E}_\rho \,,
\end{equation}
and
\begin{align}
\norm{\tilde \Omega_{\bar L}-\tilde \Omega_L}_{\rho-3\delta} 
\leq {} & \frac{\CLT \ctetOmega \CDeltaL + \CLT \CDeltatOmega \CL \delta + \CDeltaLT \ctetOmega \CL }{\gamma^2 \delta^{2\tau+1}} \norm{E}_\rho \nonumber \\
  =: {} & \frac{\CDeltatOmegaL}{\gamma^2 \delta^{2\tau+1}} \norm{E}_\rho \,. \label{eq:CDeltatOmegaL}
\end{align}

Now, we control the matrix $\bar A(\theta)$ as follows
\begin{align}
\norm{\bar A-A}_{\rho-3\delta} \leq {} &
\frac{1}{2}
\norm{
\bar B^\top \tilde \Omega_{\bar L} \bar B
- \bar B^\top \tilde \Omega_{\bar L} B
}_{\rho-3\delta} \nonumber \\
& + \frac{1}{2}
\norm{
\bar B^\top \tilde \Omega_{\bar L} B
- \bar B^\top \tilde \Omega_{L} B
}_{\rho-3\delta} + \frac{1}{2}
\norm{
\bar B^\top \tilde \Omega_{L} B
- B^\top \tilde \Omega_{L} B
}_{\rho-3\delta} \nonumber \\
\leq {} &
\sigmaB \CtOmegaL \norm{\bar B-B}_{\rho-3 \delta}
+\frac{1}{2} (\sigmaB)^2 \norm{\tilde \Omega_{\bar L}-\tilde \Omega_L}_{\rho-3 \delta} \nonumber \\
\leq {} & \frac{\sigmaB \CtOmegaL \CDeltaB + \frac{1}{2} (\sigmaB)^2 \CDeltatOmegaL}{\gamma^2 \delta^{2\tau+1}} \norm{E}_\rho
=: \frac{\CDeltaA}{\gamma^2 \delta^{2\tau+1}}\norm{E}_\rho \,, \label{eq:CDeltaA}
\end{align}
where we used that $B(\theta)^\top=B(\theta)$, and the constants \eqref{eq:CtOmegaL}, \eqref{eq:CDeltaB}
and \eqref{eq:CDeltatOmegaL}.
The same control holds for
$\norm{\bar A^\top-A^\top}_{\rho-3\delta}$, since we have that $A(\theta)^\top=-A(\theta)$.
We notice that $\CDeltaA=0$ in \textbf{Case III}.

Analogous computations yield to
\begin{align}
\norm{J \circ \bar K-J \circ K}_{\rho -2\delta} \leq {} & \frac{\cteDJ \CDeltaK}{\gamma^2 \delta^{2\tau}} \norm{E}_\rho =:\frac{\CDeltaJ}{\gamma^2 \delta^{2\tau}} \norm{E}_\rho \,, \label{eq:CDeltaJ} \\
\norm{(J \circ \bar K-J \circ K)^\top}_{\rho -2\delta} \leq {} & \frac{\cteDJT \CDeltaK}{\gamma^2 \delta^{2\tau}} \norm{E}_\rho =:\frac{\CDeltaJT}{\gamma^2 \delta^{2\tau}} \norm{E}_\rho \,, \label{eq:CDeltaJT} \\
\norm{\bar N^0-N^0}_{\rho-3\delta} \leq {} &
\norm{J}_\B \norm{\bar L-L}_{\rho-3\delta} + \norm{J \circ \bar K- J \circ K}_{\rho-2 \delta} \norm{L}_\rho \nonumber \\
\leq {} & \frac{\cteJ \CDeltaL + \CDeltaJ \CL \delta}{\gamma^2 \delta^{2\tau+1}} \norm{E}_\rho \nonumber \\
=: {} &  \frac{\CDeltaNO}{\gamma^2 \delta^{2\tau+1}} \norm{E}_\rho \,, \label{eq:CDeltaNO} \\
\norm{(\bar N^0)^\top-(N^0)^\top}_{\rho-3\delta} 
\leq {} & \frac{\CDeltaLT \cteJT + \CLT \CDeltaJT \delta }{\gamma^2 \delta^{2\tau+1}} \norm{E}_\rho \nonumber \\
=: {} & \frac{\CDeltaNOT}{\gamma^2 \delta^{2\tau+1}} \norm{E}_\rho \,. \label{eq:CDeltaNOT}
\end{align}

Finally, we control the correction of the adapted normal frame:
\begin{align}
\|\bar N-N &\|_{\rho-3\delta} \nonumber \\
& \leq \norm{\bar L \bar A-\bar L A}_{\rho-3 \delta}+
\norm{\bar L A- L A}_{\rho-3 \delta} \nonumber \\
& \hphantom{\leq}+ \norm{\bar N^0 \bar B-\bar N^0 B}_{\rho-3 \delta}+
\norm{\bar N^0 B- N^0 B}_{\rho-3 \delta} \nonumber \\
& \leq \frac{\CL \CDeltaA + \CDeltaL \CA + \CNO \CDeltaB + \CDeltaNO \sigmaB}{\gamma^2 \delta^{2\tau+1}} \norm{E}_\rho 
\nonumber \\
& =: \frac{\CDeltaN}{\gamma^2 \delta^{2\tau+1}} \norm{E}_\rho \,, \label{eq:CDeltaN}
\\
\|\bar N^\top-N^\top&\|_{\rho-3\delta} \nonumber \\
& \leq \frac{\CA \CDeltaLT + \CDeltaA \CLT + \CDeltaB \CNOT + \sigmaB \CDeltaNOT}{\gamma^2 \delta^{2\tau+1}} \norm{E}_\rho \nonumber \\
& =: \frac{\CDeltaNT}{\gamma^2 \delta^{2\tau+1}} \norm{E}_\rho \,. \label{eq:CDeltaNT}
\end{align}

\bigskip

%%%%%%%%%%%%%%%%%%%%%%%%%%%%%%%%%%%%%%%%%%%%%%%%%%%%%%%%%%
% Step 6: Control of the action of left operators
%%%%%%%%%%%%%%%%%%%%%%%%%%%%%%%%%%%%%%%%%%%%%%%%%%%%%%%%%%

\paragraph
{\emph{Step 6: Control of the action of the left operator}.}
It is worth mentioning that an important effort is made
to obtain optimal estimates for the twist condition. As it was illustrated in
the proof of Lemma \ref{lem:twist}, 
improved estimates are obtained by avoiding
the use of Cauchy estimates when controlling the action of $\Lie{\omega}$
on different objects and their corrections.

Using 
the assumption \eqref{eq:ingredient:iter:1}
and the control \eqref{eq:barEvsE} of the new error of invariance,
we preserve the control for the new objects
$\Lie{\omega}\bar K(\theta)$, $\Lie{\omega}\bar L(\theta)$ and $\Lie{\omega}\bar L(\theta)^\top$.
Indeed, using \eqref{eq:barEvsE},
we have
\begin{align*}
\norm{\Lie{\omega} \bar K}_{\rho-2\delta} \leq {} & \norm{\bar E}_\rho+ \norm{X_\H \circ \bar K}_\rho \leq \delta \cauxT + \cteXH = \CLieK\,, \\
\norm{\Lie{\omega} \bar L}_{\rho-3\delta} \leq {} & 
d \cauxT + \cteDXH \sigmaDK + \cteDXp \CLieK
= \CLieL \,, \\
\norm{\Lie{\omega} \bar L^\top}_{\rho-3\delta} \leq {} &
\max \left\{2n \cauxT + \cteDXHT \sigmaDK \, , \, \cteDXpT \CLieK \right\} 
= \CLieLT \,.
\end{align*}

Then, we also control the action of $\Lie{\omega}$ on the correction of the torus, using that
\begin{align*}
& \Lie{\omega} \bar K(\theta) - \Lie{\omega} K(\theta) =
\Lie{\omega} \DeltaK(\theta) \\
& \qquad = 
\Lie{\omega} L(\theta) \xi^L(\theta)
+L(\theta) \Lie{\omega} \xi^L(\theta)
+\Lie{\omega} N(\theta) \xi^N(\theta)
+N(\theta) \Lie{\omega}\xi^N(\theta)
\end{align*}
and, recalling previous estimates, we obtain:
\begin{align}
\norm{\Lie{\omega} \bar K - \Lie{\omega} K}_{\rho-2\delta} \leq {} &
\left(
\frac{\CLieL \CxiL}{\gamma^2 \delta^{2\tau}} +
\frac{\CL \CLiexiL}{\gamma \delta^\tau} + \frac{\CLieN \CxiN}{\gamma \delta^{\tau}}
+\CN \CLiexiN
\right)\norm{E}_\rho
\nonumber \\
  =:{} & \frac{\CDeltaLieK}{\gamma^2 \delta^{2\tau}} \norm{E}_\rho\,. 
\label{eq:CDeltaLieK}
\end{align}

The action of $\Lie{\omega}$ on the correction of $L(\theta)$ is similar. On the one hand, we
have
\begin{align}
\norm{\Lie{\omega} \bar L-\Lie{\omega} L}_{\rho -3\delta}
\leq {} & 
\norm{\Lie{\omega}( \Dif \bar K - \Dif K)}_{\rho-3\delta}
+\norm{\Lie{\omega} (X_p \circ \bar K - X_p \circ K)}_{\rho-3\delta}\nonumber \\
\leq {} & 
\norm{ \Dif (\Lie{\omega}\bar K - \Lie{\omega} K)}_{\rho-3\delta} \nonumber \\
& +\norm{(\Dif X_p \circ \bar K) [\Lie{\omega} \bar K] - 
(\Dif X_p \circ \bar K) [\Lie{\omega}K]}_{\rho-3\delta}  \nonumber \\
& +\norm{(\Dif X_p \circ \bar K) [\Lie{\omega} K] - 
(\Dif X_p \circ K) [\Lie{\omega}K]}_{\rho-3\delta}  \nonumber \\
\leq {} & \frac{d \CDeltaLieK}{\gamma^2 \delta^{2\tau+1}} \norm{E}_\rho
+ \frac{\cteDXp \CDeltaLieK}{\gamma^2 \delta^{2\tau}} \norm{E}_\rho
+ \frac{\cteDDXp \CDeltaK \CLieK}{\gamma^2 \delta^{2\tau}} \norm{E}_\rho \nonumber \\
=: {} & \frac{\CDeltaLieL}{\gamma^2 \delta^{2\tau+1}} \norm{E}_\rho\,, \label{eq:CDeltaLieL}
\end{align}
and on the other hand, we have
\begin{align}
&\norm{\Lie{\omega} \bar L^\top-\Lie{\omega} L^\top}_{\rho -3\delta} \nonumber \\
& \qquad\qquad \leq
\max\{\norm{\Lie{\omega}( \Dif \bar K - \Dif K)^\top}_{\rho-3\delta}
\, , \,
\norm{\Lie{\omega} (X_p \circ \bar K - X_p \circ K)^\top}_{\rho-3\delta}\} \nonumber \\
& \qquad\qquad \leq 
\max
\left \{\frac{2n \CDeltaLieK}{\gamma^2 \delta^{2\tau+1}} \norm{E}_\rho
\, , \,
\frac{\cteDXpT \CDeltaLieK}{\gamma^2 \delta^{2\tau}} \norm{E}_\rho
+ \frac{\cteDDXpT \CDeltaK \CLieK}{\gamma^2 \delta^{2\tau}} \norm{E}_\rho \right\} \nonumber \\
& \qquad\qquad =: \frac{\CDeltaLieLT}{\gamma^2 \delta^{2\tau+1}} \norm{E}_\rho\,. \label{eq:CDeltaLieLT}
\end{align}

To control the action of $\Lie{\omega}$ on the
correction of the matrix $B(\theta)$ that provides the transversality condition,
we first consider
the correction 
\begin{align}
\|\Lie{\omega}(G \circ \bar K) - & \Lie{\omega}(G\circ K)\|_{\rho-2\delta} 
\leq \norm{\Dif G}_{\B} \norm{\Lie{\omega}\bar K -\Lie{\omega}K}_{\rho-2\delta} \nonumber \\
&~+ \norm{\Dif^2 G}_{\B} \norm{\bar K-K}_{\rho-2\delta} \norm{\Lie{\omega} K}_{\rho-\delta} \nonumber \\
&\leq \frac{\cteDG \CDeltaLieK + \cteDDG \CDeltaK \CLieK}{\gamma^2 \delta^{2\tau}} \norm{E}_\rho =: \frac{\CDeltaLieG}{\gamma^2 \delta^{2\tau}} \norm{E}_\rho\,, \label{eq:CDeltaLieG}
\end{align}
and we control the correction of the adapted metric
\begin{equation}\label{eq:eqCDeltaLieGL}
\begin{split}
\Lie{\omega} G_{\bar L}(\theta) - \Lie{\omega} G_L(\theta)
= {} &   \Lie{\omega} \bar L^\top(\theta) G(\bar K(\theta)) \bar L(\theta) 
- \Lie{\omega}L^\top(\theta) G (K(\theta)) L(\theta) \\
& + \bar L^\top(\theta) \Lie{\omega} G (\bar K(\theta)) \bar L (\theta) 
- L^\top(\theta) \Lie{\omega} G(K(\theta)) L(\theta) \\
& +\bar L^\top(\theta) G(\bar K(\theta)) \Lie{\omega}\bar L(\theta)  
- L^\top(\theta)  G (K(\theta)) \Lie{\omega}L(\theta)
\end{split}
\end{equation}
as follows
\begin{align}
\|\Lie{\omega} G_{\bar L} & - \Lie{\omega} G_L\|_{\rho-3\delta} \nonumber \\
& \leq
\frac{\CLieLT \cteG \CDeltaL + \CLieLT \cteDG \CDeltaK \CL \delta + \CDeltaLieLT \cteG \CL}{\gamma^2 \delta^{2\tau+1}}\norm{E}_\rho \nonumber \\
& \hphantom{\leq} + 
\frac{\CLT \cteDG \CLieK \CDeltaL + \CLT \CDeltaLieG \CL \delta + \CDeltaLT \cteDG \CLieK \CL}{\gamma^2 \delta^{2\tau+1}}\norm{E}_\rho \nonumber \\
& \hphantom{\leq} + 
\frac{\CLT \cteG \CDeltaLieL + \CLT \cteDG \CDeltaK \CLieL \delta + \CDeltaLT \cteG \CLieL}{\gamma^2 \delta^{2\tau+1}}\norm{E}_\rho \nonumber \\
& =: \frac{\CDeltaLieGL}{\gamma^2 \delta^{2\tau+1}}\norm{E}_\rho\,. \label{eq:CDeltaLieGL}
\end{align}
Moreover, the following estimates (borrowed from Lemma \ref{lem:twist}) will be also useful
\[
\norm{\Lie{\omega} G_L}_{\rho-\delta} 
\leq
\CLieGL\,,
\qquad
\norm{\Lie{\omega} G_{\bar L}}_{\rho-3\delta} 
\leq
\CLieGL\,,
\]
Again, it is clear that  this control is
preserved for the corrected objects.
Using the previous estimates, we have
\begin{align}
& \norm{\Lie{\omega}\bar B -\Lie{\omega} B}_{\rho-3\delta} \leq 
\norm{\bar B \Lie{\omega} G_{\bar L} \bar B - \bar B \Lie{\omega} G_{\bar L} B}_{\rho-3 \delta} \nonumber \\
& \qquad + \norm{\bar B \Lie{\omega} G_{\bar L} B - \bar B \Lie{\omega} G_L B}_{\rho-3 \delta} 
+ \norm{\bar B \Lie{\omega} G_{L} B - B \Lie{\omega} G_L B}_{\rho-3 \delta} \nonumber \\
& \qquad \leq
\frac{
2 \sigmaB \CLieGL \CDeltaB + (\sigmaB)^2 \CDeltaLieGL
}{\gamma^2 \delta^{2\tau+1}} \norm{E}_\rho =:
\frac{
\CDeltaLieB
}{\gamma^2 \delta^{2\tau+1}} \norm{E}_\rho\,. \label{eq:CDeltaLieB}
\end{align}

By repeating the computations \eqref{eq:CDeltaLieG},
\eqref{eq:eqCDeltaLieGL} and \eqref{eq:CDeltaLieGL},
mutatis mutandis, we obtain the estimates
\begin{align}
\|\Lie{\omega}(\tilde \Omega \circ \bar K) - & \Lie{\omega}(\tilde \Omega \circ K)\|_{\rho-2\delta} 
\nonumber \\
&\leq \frac{\cteDtOmega \CDeltaLieK + \cteDDtOmega \CDeltaK \CLieK}{\gamma^2 \delta^{2\tau}} \norm{E}_\rho =: \frac{\CDeltaLietOmega}{\gamma^2 \delta^{2\tau}} \norm{E}_\rho\,, \label{eq:CDeltaLietOmega}
\end{align}
and
\begin{align}
\|\Lie{\omega} \tilde \Omega_{\bar L} & -\Lie{\omega} \tilde\Omega_L\|_{\rho-3\delta} \nonumber \\
& \leq
\frac{\CLieLT \ctetOmega \CDeltaL + \CLieLT \cteDtOmega \CDeltaK \CL \delta + \CDeltaLieLT \ctetOmega \CL}{\gamma^2 \delta^{2\tau+1}}\norm{E}_\rho \nonumber \\
& \hphantom{\leq} + 
\frac{\CLT \cteDtOmega \CLieK \CDeltaL + \CLT \CDeltaLietOmega \CL \delta + \CDeltaLT \cteDtOmega \CLieK \CL}{\gamma^2 \delta^{2\tau+1}}\norm{E}_\rho \nonumber \\
& \hphantom{\leq} + 
\frac{\CLT \ctetOmega \CDeltaLieL + \CLT \cteDtOmega \CDeltaK \CLieL \delta + \CDeltaLT \ctetOmega \CLieL}{\gamma^2 \delta^{2\tau+1}}\norm{E}_\rho \nonumber \\
& =: \frac{\CDeltaLietOmegaL}{\gamma^2 \delta^{2\tau+1}}\norm{E}_\rho\,. \label{eq:CDeltaLietOmegaL}
\end{align}
We also recall the following controls
\[
\norm{\Lie{\omega} \tilde \Omega_L}_{\rho-\delta} 
\leq
\CLietOmegaL\,.
\qquad
\norm{\Lie{\omega} \tilde \Omega_{\bar L}}_{\rho-3\delta} 
\leq
\CLietOmegaL\,.
\]
Finally, we obtain
\begin{align}
\|\Lie{\omega}\bar A & -\Lie{\omega} A\|_{\rho-3\delta} \nonumber \\
& \leq \frac{\CLieB \CtOmegaL \CDeltaB + \CLieB \CDeltatOmegaL \sigmaB + \CDeltaLieB \CtOmegaL \sigmaB}{\gamma^2 \delta^{2\tau+1}} \norm{E}_\rho \nonumber \\
& \hphantom{\leq} + \frac{1}{2}
\frac{\sigmaB \CLietOmegaL \CDeltaB + (\sigmaB)^2 \CDeltaLietOmegaL + \CDeltaB \CLietOmegaL \sigmaB}{\gamma^2 \delta^{2\tau+1}} \norm{E}_\rho \nonumber \\
& =:
\frac{
\CDeltaLieA
}{\gamma^2 \delta^{2\tau+1}} \norm{E}_\rho\,, \label{eq:CDeltaLieA} \\
\|\Lie{\omega}(J \circ \bar K) & -\Lie{\omega} (J\circ K)\|_{\rho-2\delta} \nonumber \\
& \leq \frac{\cteDJ \CDeltaLieL + \cteDDJ \CDeltaK \CLieK}{\gamma^2 \delta^{2\tau}} \norm{E}_\rho =: \frac{
\CDeltaLieJ
}{\gamma^2 \delta^{2\tau}} \norm{E}_\rho\,,
\label{eq:CDeltaLieJ} \\
\|\Lie{\omega}\bar N^0 & -\Lie{\omega} N^0\|_{\rho-3\delta} \nonumber \\
& \leq \frac{\CLieJ \CDeltaL + \CDeltaLieJ \CL\delta + \cteJ \CDeltaLieL + \CDeltaJ \CLieL \delta}{\gamma^2 \delta^{2\tau+1}} \norm{E}_\rho \nonumber \\
&  =: \frac{
\CDeltaLieNO
}{\gamma^2 \delta^{2\tau+1}} \norm{E}_\rho\,,
\label{eq:CDeltaLieNO} \\
\|\Lie{\omega}\bar N & -\Lie{\omega} N\|_{\rho-3\delta} \nonumber \\
& \leq \frac{\CLieL \CDeltaA + \CDeltaLieL \CA + \CL \CDeltaLieA + \CDeltaL \CLieA}{\gamma^2 \delta^{2\tau+1}} \norm{E}_\rho \nonumber \\
& \hphantom{\leq} + \frac{\CLieNO \CDeltaB + \CDeltaLieNO \sigmaB + \CNO \CDeltaLieB + \CDeltaNO \CLieB}{\gamma^2 \delta^{2\tau+1}} \norm{E}_\rho \nonumber \\
&  =: \frac{
\CDeltaLieN
}{\gamma^2 \delta^{2\tau+1}} \norm{E}_\rho\,,
\label{eq:CDeltaLieN}
\end{align}
where we used the constant \eqref{eq:CDeltaLieB}.

\bigskip

%%%%%%%%%%%%%%%%%%%%%%%%%%%%%%%%%%%%%%%%%%%%%%%%%%%%%%%%%%
% Step 7: Control of the new torsion condition
%%%%%%%%%%%%%%%%%%%%%%%%%%%%%%%%%%%%%%%%%%%%%%%%%%%%%%%%%%

\paragraph
{\emph{Step 7: Control of the new torsion condition}.}
Notice that this
step could be replaced by the use of Cauchy estimates, thus obtaining
a much pessimistic control. 
Now we control the non-degeneracy (twist) condition
associated to the new torsion matrix
\[
\bar T(\theta) = \bar N^\top \Omega(\bar K(\theta)) \Loper {\bar N}(\theta)\,,
\]
where
\[
\Loper{\bar N}(\theta)=\Dif X_\H (\bar K(\theta)) \bar N(\theta) + \Lie{\omega} \bar N(\theta)\,,
\]
is the infinitesimal displacement of the normal subbundle for the linearized dynamics.
At this point, intermediate
computations will be skipped for convenience. Such
details are left to the reader (they are analogous to previous computations).

We start by controlling the correction of the displacement. To this end, we observe that
\begin{align*}
\Loper{\bar N}(\theta) - \Loper{N}(\theta) = {} & 
\Dif X_{\H}(\bar K(\theta)) \bar N(\theta) -
\Dif X_{\H}(K(\theta)) N(\theta)
\\
& + \Lie{\omega} \bar N(\theta) - \Lie{\omega} N(\theta) \,,
\end{align*}
and we readily obtain
\begin{align}
\norm{\Loper{\bar N} - \Loper{N}}_{\rho-3\delta} \leq {} &
\frac{\cteDXH \CDeltaN + \cteDDXH \CDeltaK \CN \delta + \CDeltaLieN}{\gamma^2 \delta^{2\tau+1}} \norm{E}_\rho \nonumber \\
=: {} & \frac{\CDeltaLoperN}{\gamma^2 \delta^{2\tau+1}} \norm{E}_\rho \,.
\label{eq:CDeltaLoperN}
\end{align}
Now, we can control the correction of the torsion matrix
\begin{align}
& \norm{\bar T -T}_{\rho-3\delta} \leq 
\norm{\bar N^\top}_{\rho-3\delta} \norm{\Omega \circ \bar K}_{\rho-2\delta}
\norm{
\Loper{\bar N} - \Loper{N}
}_{\rho-3\delta} \nonumber \\
& \qquad \hphantom{\leq} +
\norm{\bar N^\top}_{\rho-3\delta} \norm{\Omega \circ \bar K - \Omega \circ K }_{\rho-2\delta}
\norm{\Loper{N}}_{\rho-\delta}  \nonumber \\
& \qquad  \hphantom{\leq}+
\norm{\bar N^\top-N^\top}_{\rho-3\delta} \norm{\Omega \circ K}_{\rho}
\norm{\Loper{N}}_{\rho-\delta} \nonumber \\
& \qquad \leq
\frac{\CNT \cteOmega \CDeltaLoperN + \CNT \cteDOmega \CDeltaK \CLoperN \delta + \CDeltaNT \cteOmega \CLoperN)}{\gamma^2 \delta^{2\tau+1}} \norm{E}_\rho \nonumber \\
& \qquad
=: \frac{\CDeltaT}{\gamma^2 \delta^{2\tau+1}} \norm{E}_\rho\,. \label{eq:CDeltaT}
\end{align}
Then,
we introduce the constant
\[
\CDeltaTOI := 2 (\sigmaT)^2 \CDeltaT
\]
and check the condition \eqref{eq:lem:aux} in Lemma \ref{lem:aux}:
\begin{align}
\frac{2 (\sigmaT)^2 |\aver{\bar T} - \aver{T}|}{\sigmaT-|\aver{T}^{-1}|} \leq {} &
\frac{2 (\sigmaT)^2 \norm{\bar T - T}_{\rho-3\delta}}{\sigmaT-|\aver{T}^{-1}|} \leq 
\frac{2 (\sigmaT)^2 \CDeltaT}{\sigmaT-|\aver{T}^{-1}|} \frac{\norm{E}_\rho}{\gamma^2 \delta^{2\tau+1}}
\nonumber \\
= {} & 
\frac{\CDeltaTOI}{\sigmaB-\norm{B}_\rho} \frac{\norm{E}_\rho}{\gamma^2 \delta^{2\tau+1}}
 < 1\,,
\label{eq:ingredient:iter:8}
\end{align}
where the last inequality follows from
Hypothesis \eqref{eq:cond1:K:iter} (this corresponds to the sixth term in \eqref{eq:mathfrak1}).
Hence, by invoking Lemma \ref{lem:aux}, we conclude that
\begin{equation}\label{eq:CDeltaTOI}
|\aver{\bar T}^{-1}| <\sigmaT\,, \qquad
|\aver{\bar T}^{-1} - \aver{T}^{-1}| \leq \frac{2 (\sigmaT)^2 \CDeltaT}{\gamma^2 \delta^{2\tau+1}}
\norm{E}_\rho = \frac{\CDeltaTOI}{\gamma^2 \delta^{2\tau+1}}
\norm{E}_\rho\,,
\end{equation}
and so, we obtain the estimates \eqref{eq:B:iter1} and \eqref{eq:est:DeltaB} on the new object.
\end{proof}

\subsection{Convergence of the iterative process}\label{ssec:proof:KAM}

Now we are ready to proof our first KAM theorem with conserved quantities.
Once the quadratic procedure has been
established in Section \ref{ssec:iter:lemmas},
proving the convergence of the scheme follows standard arguments.
Nevertheless, the required computations will be carefully detailed
since we are interested in providing explicit conditions for the KAM theorem.

\begin{proof}
[Proof of Theorem \ref{theo:KAM}]
Let us consider the approximate torus $K_0:=K$ with initial error $E_0:=E$. We
also introduce $B_0:=B$ and $T_0:=T$ associated with the initial approximation.
By applying Lemma \ref{lem:KAM:inter:integral} recursively, we obtain new objects
$K_s=\bar K_{s-1}$,
$E_s=\bar E_{s-1}$,
$B_s=\bar B_{s-1}$ and
$T_s=\bar T_{s-1}$. The domain of analyticity
of these objects is reduced at every step. To characterize this fact, we
introduce parameters $a_1>1$, $a_2>1$, $a_3=3 \frac{a_1}{a_1-1} \frac{a_2}{a_2-1}$ and
define
\[
\rho_0=\rho, \quad
\delta_0=\frac{\rho_0}{a_3}, \quad
\rho_s=\rho_{s-1} - 3 \delta_{s-1}, \quad
\delta_s= \frac{\delta_0}{a_1^s}, \quad
\rho_\infty = \lim_{s\to \infty} \rho_s = \frac{\rho_0}{a_2}\,.
\]
We can select the above parameters (together with the
parameter $\cauxT$) to optimize the convergence of the KAM
process for a particular problem (see \cite{FiguerasHL17}). 

Let us assume that we have successfully applied $s$ times Lemma \ref{lem:KAM:inter:integral}
(the Iterative Lemma), and 
let $K_s$, $E_s$, $B_s$ and $T_s$ be the objects at the $s$-step of the quasi-Newton method.
We observe that condition \eqref{eq:cond1:K:iter} is required at every step, but the construction
has been performed in such a way that we control 
$\norm{\Dif K_s}_{\rho_s}$,
$\norm{(\Dif K_s)^\top}_{\rho_s}$,
$\norm{B_s}_{\rho_s}$,
$\dist(K_s(\TT^d_{\rho_s}),\partial \B)$,
and
$\abs{\aver{T_s}^{-1}}$ uniformly with respect to $s$,
so the constants that appear in Lemma \ref{lem:KAM:inter:integral}
(which are obtained in Table \ref{tab:constants:all} and
Table \ref{tab:constants:all:2})
are taken to be the same for all steps by considering the worst
value of $\delta_s$, that is, $\delta_0 = \rho_0/a_3$.

The first computation is tracking the sequence $E_s$ of invariance errors:
\begin{equation}\label{eq:conv:geom}
\begin{split}
\norm{E_s}_{\rho_s} < {} & 
\frac{\CE}{\gamma^4 \delta_{s-1}^{4\tau}} \norm{E_{s-1}}_{\rho_{s-1}}^2 =
\frac{\CE a_1^{4\tau(s-1)}}{\gamma^4 \delta_0^{4\tau}} \norm{E_{s-1}}_{\rho_{s-1}}^2 \\
< {} & \left( \frac{a_1^{4 \tau} \CE \norm{E_0}_{\rho_0}}{\gamma^4 \delta_0^{4\tau}}
\right)^{2^s-1} a_1^{-4\tau s} \norm{E_0}_{\rho_0} 
< a_1^{-4\tau s} \norm{E_0}_{\rho_0}
\,, 
\end{split}
\end{equation}
where we used the sums
$1+2+\ldots+2^{s-1}=2^s-1$, and
$1(s-1)+2(s-2)+2^2(s-3)\ldots+2^{s-2} 1 = 2^s-s-1$.
Notice that we also used the inequality
\begin{equation}\label{eq:theo:conv:1}
\frac{a_1^{4 \tau} \CE \norm{E_0}_{\rho_0}}{\gamma^4 \delta_0^{4\tau}} <1\,,
\end{equation}
which is included in \eqref{eq:KAM:HYP}. Now,
using expression \eqref{eq:conv:geom}, we check the
Hypothesis \eqref{eq:cond1:K:iter} of the iterative Lemma,
so that we can perform the step $s+1$. 
The required sufficient condition will be included in
the hypothesis \eqref{eq:KAM:HYP} of the KAM theorem.

The right inequality in \eqref{eq:cond1:K:iter} reads:
\[
\frac{\CE \norm{E_s}_{\rho_s}}{\gamma^4 \delta_s^{4\tau}} \leq
\frac{\CE a_1^{-4\tau s} \norm{E_0}_{\rho_0}}{\gamma^4 \delta_s^{4\tau}} 
= 
\frac{\CE \norm{E_0}_{\rho_0}}{\gamma^4 \delta_0^{4\tau}} 
\leq a_1^{-4\tau}
< 1 \,,
\]
where we used \eqref{eq:conv:geom} and \eqref{eq:theo:conv:1}.

The left inequality in \eqref{eq:cond1:K:iter} has several terms
(which correspond to the different components in \eqref{eq:mathfrak1}).
The first of them, using again \eqref{eq:conv:geom}, is given by
\begin{equation}\label{eq:theo:conv:2}
\frac{\norm{E_s}_{\rho_s}}{\delta_s} \leq \frac{a_1^s a_1^{-4\tau s} \norm{E_0}_{\rho_0}}{\delta_0}
\leq \frac{\norm{E_0}_{\rho_0}}{\delta_0} < \cauxT\,.
\end{equation}
We used that $\tau \geq d-1 \geq 1$, so that $1-4\tau<0$.
The last inequality in \eqref{eq:theo:conv:2} is included in
\eqref{eq:KAM:HYP}. The second term is guaranteed by
performing the following computation
\begin{equation}\label{eq:theo:conv:3}
\frac{2 \Csym \norm{E_s}_{\rho_s}}{\gamma \delta_s^{\tau+1}}
\leq 
\frac{2 \Csym a_1^{s(\tau+1)} a_1^{-4\tau s} \norm{E_0}_{\rho_0}}{\gamma \delta_0^{\tau+1}}
\leq 
\frac{2 \Csym \norm{E_0}_{\rho_0}}{\gamma \delta_0^{\tau+1}}
< 1 \,,
\end{equation}
where we used \eqref{eq:conv:geom}, the fact that $1-3\tau<0$,
and we have included the last inequality in \eqref{eq:KAM:HYP}.
The remaining conditions are similar. We only need to pay attention to the
fact that they involve the objects
$\Dif K_s$, $(\Dif K_s)^\top$, $B_s$ and $\aver{T_s}^{-1}$,
at the $s^{\mathrm{th}}$ step. Hence, we have to relate these conditions
to the corresponding initial objects 
$\Dif K_0$, $(\Dif K_0)^\top$, $B_0$ and $\aver{T_0}^{-1}$.
For example, the third term in \eqref{eq:cond1:K:iter} reads
\[
\left(\frac{ d\CDeltaK}{\sigmaDK-\norm{\Dif K_s}_{\rho_s}}\right) \frac{\norm{E_s}_{\rho_s}}{\gamma^2 \delta_s^{2\tau+1}}< 1\,,
\]
and it is checked by performing the following computations
\begin{align}
\norm{\Dif K_s}_{\rho_s} + & \frac{d \CDeltaK\norm{E_s}_{\rho_s}}{\gamma^2 \delta_s^{2\tau+1}} 
<  
\norm{\Dif K_0}_{\rho_0} + \sum_{j=0}^s \frac{d \CDeltaK \norm{E_j}_{\rho_j}}{\gamma^2 \delta_j^{2\tau+1}} \nonumber \\
< {} &  
\norm{\Dif K_0}_{\rho_0} + \sum_{j=0}^\infty \frac{d \CDeltaK
a_1^{(1-2\tau)j}}{\gamma^2 \delta_0^{2\tau+1}}
\norm{E_0}_{\rho_0} \nonumber
\\
= {} &  
\norm{\Dif K_0}_{\rho_0} + \frac{d
\CDeltaK \norm{E_0}_{\rho_0}}{\gamma^2 \delta_0^{2\tau+1}}
\left(
\frac{1}{1-a_1^{1-2\tau}}
\right) < \sigmaDK\,. \label{eq:theo:conv:4}
\end{align}
Again, the last
inequality is included into \eqref{eq:KAM:HYP}.
The fourth, fifth and sixth terms in \eqref{eq:cond1:K:iter}
(associated to $(\Dif K_s)^\top$, $B_{s}$ and $\aver{T_s}^{-1}$, respectively)
follow by reproducing the same computations. Finally,
the seventh term in \eqref{eq:cond1:K:iter} is checked as follows
\begin{align}
\dist (K_s(\TT^d_{\rho_s}),\partial \B) - & \frac{\CDeltaK}{\gamma^2 \delta_s^{2\tau}} \norm{E_s}_{\rho_s}> 
\dist (K_0(\TT^d_{\rho_0}),\partial \B) - \sum_{j=0}^\infty \frac{\CDeltaK a_1^{-2\tau j} \norm{E_0}_{\rho_0}}{\gamma^2 \delta_0^{2\tau}} 
\nonumber \\
& = \dist (K_0(\TT^d_{\rho_0}),\partial \B) - \frac{\CDeltaK \norm{E_0}_{\rho_0}}{\gamma^2 \delta_0^{2\tau}}
\frac{1}{1-a_1^{-2\tau}} > 0\,,
\label{eq:check:dist}
\end{align}
where the last inequality is included into \eqref{eq:KAM:HYP}.

Having guaranteed all hypothesis of Lemma \ref{lem:KAM:inter:integral},
we collect the inequalities
\eqref{eq:theo:conv:1},
\eqref{eq:theo:conv:2},
\eqref{eq:theo:conv:3},
\eqref{eq:theo:conv:4} and
\eqref{eq:check:dist} that are included into
hypothesis \eqref{eq:KAM:HYP}. This follows
by introducing the constant $\mathfrak{C}_1$ as
\begin{equation}\label{eq:cte:mathfrakC1}
\mathfrak{C}_1:=\max
\left\{
(a_1 a_3)^{4\tau} \CE \, , \, (a_3)^{2\tau+1} \gamma^2 \rho^{2\tau-1} \CDeltatot
\right\}
\end{equation}
where
\begin{align}
\CDeltaI := {} &
\max
\Bigg\{
\frac{d \CDeltaK}{\sigmaDK-\norm{\Dif K}_{\rho}} \, , \,
\frac{2n \CDeltaK}{\sigmaDKT-\norm{(\Dif K)^\top}_{\rho}} \, , \label{eq:CDeltaI}
\\
& \qquad\qquad\qquad\qquad  \frac{\CDeltaB}{\sigmaB-\norm{B}_{\rho}} \, , \,
\frac{\CDeltaTOI}{\sigmaT-\abs{\aver{T}^{-1}}}
\Bigg\} \,, \nonumber \\
\CDeltaII := {} & \frac{\CDeltaK \delta}{\dist(K(\TT^d_{\rho}),\partial \B)}\,,
\label{eq:CDeltaII} \\
\CDeltatot := {} & \max
\Bigg\{
\frac{\gamma^2 \delta^{2\tau}}{\cauxT} \, , \,
2 \Csym \gamma \delta^\tau \, , \,
\frac{\CDeltaI}{1-a_1^{1-2\tau}} \, , \, 
\frac{\CDeltaII}{1-a_1^{-2\tau}}
\Bigg\}
\label{eq:CDeltatot} \,.
\end{align}
Note that we recovered the original notation $K=K_0$, $B=B_0$, $T=T_0$, $\rho=\rho_0$
and $\delta=\delta_0$ for the original objects.

Therefore, by induction, we can apply the iterative process infinitely many times.
Indeed, we have
\[
\norm{E_s}_{\rho_s} < a_1^{-4 \tau s} \norm{E}_{\rho} \longrightarrow  0\
\quad \mbox{when} \quad s\rightarrow 0\,
\]
so the iterative scheme converges to a true quasi-periodic torus $K_\infty$.
As a result of the output of Lemma \ref{lem:KAM:inter:integral}, this object
satisfies
$K_\infty \in \Anal(\TT^{d}_{\rho_\infty})$ and
\[
\norm{\Dif K_\infty}_{\rho_\infty} < \sigmaDK\,,
\qquad
\norm{(\Dif K_\infty)^\top}_{\rho_\infty} < \sigmaDKT\,,
\qquad
\dist(K_\infty(\TT^d),\pd \B) > 0\,.
\]
Furthermore, we control the 
the limit objects
by repeating the computations in \eqref{eq:check:dist} as follows
\begin{align}
&\norm{K_\infty - K}_{\rho_\infty} \leq {} \sum_{j=0}^\infty
\norm{K_{j+1} - K_j}_{\rho_{j+1}} < \frac{\CDeltaK \norm{E}_{\rho}}{\gamma^2 \delta^{2\tau}}
\frac{1}{1-a_1^{-2\tau}} =:  
\frac{\mathfrak{C}_2 \norm{E}_{\rho}}{\gamma^2 \rho^{2\tau}} \label{eq:Cmathfrak2}\,, \\
&\abs{\aver{c \circ K_\infty} - \aver{c \circ K}} \leq {}
\norm{\Dif c}_\B \norm{K_\infty - K}_{\rho_\infty} < 
\frac{\cteDc \mathfrak{C}_2 \norm{E}_{\rho}}{\gamma^2 \rho^{2\tau}} =:
\frac{\mathfrak{C}_3 \norm{E}_{\rho}}{\gamma^2 \rho^{2\tau}} \label{eq:Cmathfrak3}\,, 
\end{align}
thus obtaining the estimates in \eqref{eq:close}.
This completes the proof of the ordinary KAM theorem.
\end{proof}

\section{Proof of the generalized iso-energetic KAM theorem}\label{sec:proof:KAM:iso}

In the section we present a proof of Theorem~\ref{theo:KAM:iso}
following the same structure of the proof of Theorem~\ref{theo:KAM} in
Section \ref{sec:proof:KAM}.
We will emphasize the differences between both proofs, and the computations that 
are analogous will be conveniently
omitted for the sake of brevity.
In Section \ref{ssec:qNewton:iso} we discuss the approximate solution of
linearized equations in the symplectic frame constructed in Section
\ref{sec:lemmas}. In Section
\ref{ssec:iter:lemmas:iso} we produce quantitative estimates for
the objects obtained when performing one iteration of the
previous procedure. Finally, in Section
\ref{ssec:proof:KAM:iso} we discuss the convergence of the quasi-Newton method.

\subsection{The quasi-Newton method}\label{ssec:qNewton:iso}

The proof of Theorem~\ref{theo:KAM:iso} consists again on
refining $K(\theta)$ and $\omega$ by means of a quasi-Newton method.
In this case, the total error is associated with
the invariance error and the
target energy level:
\[
E_c(\theta)=
\begin{pmatrix}
E(\theta) \\
E^\omega
\end{pmatrix}
= 
\begin{pmatrix}
X_\H(K(\theta)) + \Lie{\omega} K(\theta) \\
\aver{c \circ K} - c_0
\end{pmatrix}\, .
\]
Then, we look for a corrected parameterization
$\bar K(\theta)= K(\theta)+\DeltaK(\theta)$
and a corrected frequency $\bar \omega= \omega + \Deltao$
by considering the linearized system
\begin{equation}\label{eq:lin1:iso}
\begin{split}
& \Dif X_\H (K(\theta)) \DeltaK(\theta) + \Lie{\omega}\DeltaK(\theta) +
\Lie{\Deltao} K(\theta)
= - E(\theta) \,, \\
& \aver{(\Dif c \circ K) \DeltaK } =  - E^\omega \,.
\end{split}
\end{equation}
If we obtain a good enough approximation of the solution
$(\DeltaK(\theta),\Deltao)$
of \eqref{eq:lin1:iso},
then $\bar K(\theta)$ provides a parameterization
of an approximately invariant torus of frequency $\bar \omega$,
with a quadratic error in terms of $E_c(\theta)=(E(\theta),E^\omega)$. 

To face the linearized equations \eqref{eq:lin1:iso}, 
we introduce again the linear change
\begin{equation}\label{eq:choice:DK:iso}
\DeltaK (\theta) = P(\theta) \xi(\theta) \,,
\end{equation}
where $P(\theta)$ is the approximately symplectic frame
characterized in Lemma \ref{lem:sympl}. 
In addition, to ensure Diophantine properties for $\bar
\omega$, we select a parallel correction of the frequency:
\begin{equation}\label{eq:choice:Do:iso}
\Deltao = - \omega \, \xio\,,
\end{equation}
where $\xio$ is a real number.
This guarantees the solvability of system \eqref{eq:lin1:iso}
along the iterative procedure. The following notation
for the new unknowns will be useful
\[
\xi_c(\theta)=
\begin{pmatrix}
\xi(\theta) \\
\xi^\omega
\end{pmatrix}
\,,
\qquad
\xi(\theta)=
\begin{pmatrix}
\xi^L(\theta) \\
\xi^N(\theta)
\end{pmatrix}
\,.
\]

Then, taking into account the expressions \eqref{eq:choice:DK:iso}
and \eqref{eq:choice:Do:iso}, the system of equations in \eqref{eq:lin1:iso} becomes
\begin{align}
& \left( \Dif X_\H (K(\theta)) P(\theta) + \Lie{\omega}P(\theta) \right) \xi(\theta)
+P(\theta) \Lie{\omega} \xi(\theta) 
- \Lie{\omega} K(\theta) \xio
=   - E(\theta) \,, 
\label{eq:lin:1E:iso} \\
& \aver{(\Dif c \circ K) P \xi} =  - E^\omega \,.
\nonumber
\end{align}
We now multiply both sides of \eqref{eq:lin:1E:iso} by $-\Omega_0 P(\theta)^\top
\Omega(K(\theta))$, and we 
use the
geometric properties in Lemma \ref{lem:sympl} and
Lemma \ref{lem:reduc}, thus obtaining the equivalent equations:
\begin{align}
& \left(\Lambda(\theta)+\Ered(\theta) \right) \xi(\theta)
+ \left( I_{2n} - \Omega_0 \Esym(\theta)  \right) \Lie{\omega} \xi(\theta) 
\label{eq:lin:1Enew:iso} \\
\
&\qquad\qquad + \Omega_0 P(\theta)^\top \Omega(K(\theta)) \Lie{\omega} K(\theta) \xio
=  
\Omega_0 P(\theta)^\top \Omega(K(\theta)) E(\theta) \,, \nonumber \\
& \aver{(\Dif c \circ K) N \xi^N } + \aver{(\Dif c \circ K)L \xi^L }
=  - E^\omega \,.\label{eq:lin:1Eh:new:iso} 
\end{align}

We observe that
\[
\begin{split}
\Omega_0 P(\theta)^\top \Omega(K(\theta)) \Lie{\omega} K(\theta)
& = -\Omega_0 P(\theta)^\top \Omega(K(\theta)) L(\theta) \homega \\
& = 
\begin{pmatrix}
\homega \\
0_n
\end{pmatrix}
+
\begin{pmatrix}
A(\theta)^\top \Elag(\theta)\homega \\
-\Elag(\theta)\homega
\end{pmatrix}\, ,
\end{split}
\]
where we used the notation for $\homega$ introduced 
in the statement of the theorem
(e.g. equation \eqref{eq:Tc}). Moreover,
recalling the computations in Lemma \ref{lem:cons:H:p},
we obtain
\begin{align*}
\Dif c(K(\theta)) L(\theta) = {} &
\begin{pmatrix}
\Dif c(K(\theta)) \Dif K(\theta) & \Dif c(K(\theta)) X_p(K(\theta))
\end{pmatrix} \\
= {} &
\begin{pmatrix}
\Dif (c(K(\theta))) & 0_{n-d}^\top
\end{pmatrix} \\
= {} &
\begin{pmatrix}
\Dif (\R{\omega}(\Dif c(K(\theta))E(\theta))) & 0_{n-d}^\top
\end{pmatrix} \,.
\end{align*}

From the above expressions, we conclude that equations
\eqref{eq:lin:1Enew:iso} and \eqref{eq:lin:1Eh:new:iso}, are
approximated by a triangular system
that requires to solve two cohomological equations of the form
\eqref{eq:calL} consecutively. Quantitative estimates for the solutions of
such equations are provided in the following statement.

\begin{lemma}[Upper triangular equations for the iso-energetic case]\label{lem:upperT:iso}
Let $\omega \in \Dioph{\gamma}{\tau}$, $\eta^\omega \in \RR$, and let us
consider the map $\eta= (\eta^L,\eta^N) : \TT^d \to
\RR^{2n}\simeq \RR^n\times\RR^n$, with components in $\Anal(\TT^d_\rho)$, and the maps 
$T : \TT^d \rightarrow \RR^{n\times n}$ and $\Tdown : \TT^d \to \RR^{1\times n}$, 
with components in $\Anal(\TT^d_{\rho-\delta})$.
Let us introduce the notation
\[
T_c(\theta) = 
\begin{pmatrix} 
T(\theta) & \homega \\
\Tdown(\theta) & 0 
\end{pmatrix}, 
\qquad
\homega = 
\begin{pmatrix} \omega  \\ 0_{n-d} \end{pmatrix}\,.
\]
Let us assume that $T_c$ satisfies the non-degeneracy condition $\det
\aver{T_c} \neq 0$, and $\eta$ satisfies the compatibility condition
$\aver{\eta^N}=0_n$. Then, for any $\xi^L_0\in \RR^n$,
the system of equations
\begin{equation}\label{eq:syst:iso:peich}
\begin{split}
\begin{pmatrix}
O_n & T(\theta) \\
O_n & O_n
\end{pmatrix}
\begin{pmatrix}
\xi^L(\theta) \\
\xi^N(\theta)
\end{pmatrix}
+
\begin{pmatrix}
\Lie{\omega} \xi^L(\theta) \\
\Lie{\omega} \xi^N(\theta)
\end{pmatrix}
+
\begin{pmatrix}
\homega \xio \\
0_n
\end{pmatrix}
= {} & 
\begin{pmatrix}
\eta^L(\theta) \\
\eta^N(\theta)
\end{pmatrix} \\
\aver{\Tdown \xi^N} = {} & \eta^\omega\,, 
\end{split}
\end{equation}
has a solution given by
\begin{align}
\xi^N(\theta)={} & \xi^N_0 + \R{\omega}(\eta^N(\theta)) \,, \label{eq:xiy:iso} \\
\xi^L(\theta)={} & \xi^L_0 + \R{\omega}(\eta^L(\theta) - T(\theta)
\xi^N(\theta)) \,,\label{eq:xix:iso}
\end{align}
where
\begin{equation}\label{eq:averxiy:iso}
\begin{pmatrix}
\xi^N_0 \\
\xio
\end{pmatrix}
=
\aver{T_c}^{-1}
\begin{pmatrix}
\aver{\eta^L-T \R{\omega}(\eta^N)} \\
\aver{\eta^\omega- \Tdown \R{\omega}(\eta^N)}
\end{pmatrix}
\end{equation}
and $\R{\omega}$ is given by \eqref{eq:small:formal}.

Moreover, we have the estimates
\begin{align*}
& |\xi_0^N|,  |\xio| \leq \Abs{\aver{T_c}^{-1}} \max \left\{
\norm{\eta^L}_\rho
+ \frac{c_R}{\gamma \delta^\tau}
\norm{T}_{\rho-\delta}
\norm{\eta^N}_\rho
\, , \, \right.\\
& \qquad\qquad\qquad\qquad\qquad\qquad \left.
|\eta^\omega| + \frac{c_R }{\gamma \delta^\tau} \norm{\Tdown}_{\rho-\delta} \norm{\eta^N}_\rho\right\}\,, \\
& \norm{\xi^N}_{\rho-\delta} \leq |\xi_0^N| + \frac{c_R}{\gamma \delta^\tau}
\norm{\eta^N}_\rho \,, \\
& \norm{\xi^L}_{\rho-2\delta} \leq |\xi_0^L| + \frac{c_R}{\gamma \delta^\tau}
\left(
\norm{\eta^L}_{\rho-\delta}+\norm{T}_{\rho-\delta}
\norm{\xi^N}_{\rho-\delta}
\right) \,.
\end{align*}
\end{lemma}

\begin{proof}
It is analogous to Lemma \ref{lem:upperT}. After solving
$\xi^N = \xi_0^N + \R{\omega}(\eta^N)$, $\xi_0^N \in \RR^n$ from the triangular
system, we observe that
\[
\aver{\Tdown \xi^N} = 
\aver{\Tdown} \xi_0^N + 
\aver{\Tdown \R{\omega} (\eta^N)} \,,
\]
so that the last equation in \eqref{eq:syst:iso:peich} becomes
\[
\aver{\Tdown} \xi_0^N =
\eta^\omega-\aver{\Tdown \R{\omega}(\eta^N)} 
=\aver{\eta^\omega-\Tdown \R{\omega}(\eta^N)} 
\,.
\]
This equation, together with the compatibility condition
required to solve the equation for $\xi^L$, yields
to a linear system which
can be solved using that $\det \aver{T_c} \neq 0$,
thus obtaining \eqref{eq:averxiy:iso}.
The estimates are obtained using Lemma \ref{lem:Russmann}.
\end{proof}

To approximate the solutions of 
\eqref{eq:lin:1Enew:iso}--\eqref{eq:lin:1Eh:new:iso},
we will invoke Lemma \ref{lem:upperT:iso} taking
\begin{align}
& \eta^L(\theta) = - N(\theta)^\top \Omega(K(\theta)) E(\theta)\,,
\label{eq:etaL:corr:iso} \\
& \eta^N(\theta) = L(\theta)^\top \Omega(K(\theta)) E(\theta)\,,
\label{eq:etaN:corr:iso} \\
& \eta^\omega = -E^\omega\,,
\label{eq:etao:corr:iso} \\
& \Tdown(\theta) = \Dif c(K(\theta)) N(\theta)\,,
\label{eq:Tdown:corr:iso}
\end{align}
and $T(\theta)$ given by \eqref{eq:T}. 
We recall from Lemma \ref{lem:sympl}
that the compatibility condition $\aver{\eta^N}=0_n$ is satisfied.
We will select the solution that satisfies
$\xi_0^L=\aver{\xi^L}=0_n$, even though
other choices can be selected according to the context (see
Remark~\ref{rem:unicity}). 

Then, from Lemma \ref{lem:upperT:iso} we have
$\norm{\xi}_{\rho-2\delta}=\cO(\norm{E_c}_\rho)$ and
$|\xi^\omega|=\cO(\norm{E_c}_\rho)$, 
and
using the geometric properties characterized in Section \ref{sec:lemmas},
we have
$\norm{\Ered}_{\rho-2\delta}=\cO(\norm{E}_\rho)$
and 
$\norm{\Esym}_{\rho-2\delta}=\cO(\norm{E}_\rho)$.
Hence, the error produced when approximating the solutions of
\eqref{eq:lin:1Enew:iso}--\eqref{eq:lin:1Eh:new:iso}
using the solutions of
\eqref{eq:syst:iso:peich} will be controlled by $\cO(\norm{E_c}_\rho^2)$.
This, together with other estimates, will be suitably quantified in the next section.

\subsection{One step of the iterative procedure}\label{ssec:iter:lemmas:iso}

In this section we apply one correction of the quasi-Newton method 
described in Section \ref{ssec:qNewton:iso} and we obtain sharp quantitative
estimates 
for the new approximately invariant torus and related objects. We 
set sufficient conditions to preserve the control of the previous 
estimates.

\begin{lemma} [The Iterative Lemma in the iso-energetic case] \label{lem:KAM:inter:integral:iso}
Let us consider the same setting and hypotheses of Theorem \ref{theo:KAM:iso},
and a constant $\cauxT>0$.
Then, there exist constants
$\CDeltaK$,
$\CDeltao$,
$\CDeltaB$,
$\CDeltaTcOI$ and
$\CEc$
such that
if the inequalities
\begin{equation}\label{eq:cond1:K:iter:iso}
\frac{\hCDelta \norm{E_c}_\rho}{\gamma^2 \delta^{2\tau+1}} < 1
\qquad\qquad
\frac{\CEc \norm{E_c}_\rho}{\gamma^4 \delta^{4\tau}} < 1
\end{equation}
hold for some $0<\delta< \rho$, where
\begin{equation}\label{eq:mathfrak1:iso}
\begin{split}
\hCDelta := {} & \max \bigg\{\frac{\gamma^2 \delta^{2\tau}}{\cauxT}
\, , \, 
2 \Csym \gamma \delta^{\tau} 
\, , \,
\frac{ d \CDeltaK }{\sigmaDK - \norm{\Dif K}_\rho} 
\, , \,
\frac{ 2n  \CDeltaK }{\sigmaDKT - \norm{(\Dif K)^\top}_\rho}
\, , \,
\\
&  \frac{\CDeltaB}{\sigmaB - \norm{B}_\rho} 
\, , \,
 \frac{\CDeltaTcOI}{\sigmaTc - \abs{\aver{T_c}^{-1}}} 
\, , \,
 \frac{\CDeltaK \delta}{\dist (K(\TT^d_\rho),\partial B)}
\, , \,
 \frac{\CDeltao \gamma \delta^{\tau+1}}{\dist (\omega,\partial \Theta)}
 \bigg\} \,,
\end{split}
\end{equation}
then we have an approximate torus of frequency $\bar \omega=\omega+\Deltao$
given by $\bar K=K+\DeltaK$, with components in $\Anal(\TT^d_{\rho-2\delta})$, that defines new objects $\bar B$
and $\bar T_c$ (obtained replacing $K$ by $\bar K$) satisfying
\begin{align}
& \norm{\Dif \bar K}_{\rho-3\delta} < \sigmaDK \,, \label{eq:DK:iter1:iso} \\
& \norm{(\Dif \bar K)^\top}_{\rho-3\delta} < \sigmaDKT \,, \label{eq:DKT:iter1:iso} \\
& \norm{\bar B}_{\rho-3\delta} < \sigmaB \,, \label{eq:B:iter1:iso} \\
& \abs{\aver{\bar T_c}^{-1}} < \sigmaTc \,, \label{eq:T:iter1:iso} \\
& \dist(\bar K(\TT^d_{\rho-2\delta}),\partial \B) >0 \,, \label{eq:distB:iter1:iso}\\
& \dist(\bar \omega,\partial \Theta) >0 \,, \label{eq:disto:iter1:iso} 
\end{align}
and
\begin{align}
& \norm{\bar K-K}_{\rho-2 \delta} < \frac{\CDeltaK}{\gamma^2 \delta^{2\tau}}
\norm{E_c}_\rho\,, \label{eq:est:DeltaK:iso} \\
& \abs{\bar \omega-\omega} < \frac{\CDeltao}{\gamma \delta^{\tau}}
\norm{E_c}_\rho\,, \label{eq:est:Deltao:iso} \\
& \norm{\bar B-B}_{\rho-3\delta} < \frac{\CDeltaB}{\gamma^2 \delta^{2 \tau+1}}
\norm{E_c}_\rho\,, \label{eq:est:DeltaB:iso} \\
& \abs{\aver{\bar T_c}^{-1}-\aver{T_c}^{-1} } < \frac{\CDeltaTcOI}{\gamma^2 \delta^{2 \tau+1}}
\norm{E_c}_\rho\,, \label{eq:est:DeltaT:iso}
\end{align}
The new total error is given by
\[
\bar E_c(\theta)=
\begin{pmatrix}
\bar E(\theta) \\
\bar E^\omega
\end{pmatrix}
= 
\begin{pmatrix}
X_\H(\bar K(\theta)) + \Lie{\bar \omega} \bar K(\theta) \\
\aver{c \circ \bar K} - c_0
\end{pmatrix} \,,
\]
and satisfies
\begin{equation}\label{eq:E:iter1:iso}
\norm{\bar E_c}_{\rho-2\delta} < \frac{\CEc}{\gamma^4
\delta^{4\tau}}\norm{E_c}_\rho^2\,.
\end{equation}
The above constants are collected in Table~\ref{tab:constants:all:2:iso}.
\end{lemma}

\begin{proof}
The proof of this result is parallel to the ordinary situation.
On the one hand,
those
constants that must be changed in this result (e.g. $\CxiNO$)
will be redefined using the symbol ``$:=$'' and
will be included in Table \ref{tab:constants:all:2:iso}.
On the other hand,
those constants that are not
redefined (e.g. $\CDeltaK$) will have the same expression that
in the proof of Lemma \ref{lem:KAM:inter:integral} and the
reader is referred to Table \ref{tab:constants:all:2} for the
corresponding label in the text.

\bigskip

\paragraph
{\emph{Step 1: Control of the new parameterization}.}
We start by considering the new objects $\bar K(\theta)=K(\theta) + \DeltaK(\theta)$
and $\bar \omega = \omega+\Deltao$, obtained 
from the expressions \eqref{eq:choice:DK:iso}
and \eqref{eq:choice:Do:iso}, using the solutions of
the system
\eqref{eq:syst:iso:peich} taking the objects
\eqref{eq:etaL:corr:iso}
\eqref{eq:etaN:corr:iso}
\eqref{eq:etao:corr:iso} and
\eqref{eq:Tdown:corr:iso}.
We have
\[
\norm{\eta^L}_\rho \leq \CNT \cteOmega \norm{E}_\rho \,,
\qquad
\norm{\eta^N}_\rho \leq \CLT \cteOmega \norm{E}_\rho \,,
\qquad
\abs{\eta^\omega} \leq \abs{E^\omega} \,,
\]
and
\[
\norm{\Tdown}_{\rho-\delta} =
\norm{(\Dif c\circ K) N}_{\rho-\delta} \leq \norm{\Dif c}_\B \norm{N}_\rho \leq \cteDc \CN\,.
\]
Notice that
\[
\norm{E_c}_\rho = \max \{\norm{E}_\rho,|E^\omega|\}\,,
\]
so, from now on, we will use that $\norm{E}_\rho \leq \norm{E_c}_\rho$ and
$|E^\omega| \leq \norm{E_c}_\rho$.

In order to invoke Lemma \ref{lem:twist}
(we must fulfill condition \eqref{eq:fake:cond})
we include the inequality
\begin{equation}\label{eq:ingredient:iter:1:iso}
\frac{\norm{E_c}_\rho}{\delta} < \cauxT
\end{equation}
into Hypothesis \eqref{eq:cond1:K:iter:iso} (this corresponds to the first term in \eqref{eq:mathfrak1:iso}).
Hence, combining
Lemma \ref{lem:twist}
and
Lemma \ref{lem:upperT:iso},
we obtain estimates for the solution
of the cohomological equations
(we recall that $\xi_0^L=0_n$)
\begin{align}
|\xi_0^N|,|\xi^\omega| \leq {} &
\sigmaTc \max \Big\{\CNT \cteOmega + \frac{c_R}{\gamma \delta^\tau}
\CT \CLT \cteOmega \, , \, 1 + \frac{c_R}{\gamma \delta^\tau} \cteDc \CN \CLT
\cteOmega \Big\}\norm{E_c}_\rho \nonumber \\
& =: \frac{\CxiNO}{\gamma \delta^\tau}\norm{E_c}_\rho
=: \frac{\Cxio}{\gamma \delta^\tau}\norm{E_c}_\rho\,, \label{eq:CxiNO:iso}
\end{align}
and we observe that $\xi^L(\theta)$ and $\xi^N(\theta)$ are controlled
as in Lemma \ref{lem:KAM:inter:integral}, and so
are the objects $\xi(\theta)$, $\bar K(\theta)$ and $\bar K(\theta)-K(\theta)$,
thus obtaining the estimate in \eqref{eq:est:DeltaK:iso}.
Observing that 
\begin{equation}\label{eq:Comega}
\abs{\omega_*} < \abs{\omega} < \sigmao \abs{\omega_*} =: \Comega\,,
\end{equation}
we get the estimate \eqref{eq:est:Deltao:iso} as follows:
\begin{equation}\label{eq:CDeltao}
\abs{\bar \omega-\omega} = \abs{\omega \xi^\omega}< \frac{\Comega \Cxio}{\gamma \delta^\tau}
\norm{E_c}_\rho
=: \frac{\CDeltao}{\gamma \delta^\tau}
\norm{E_c}_\rho
\,.
\end{equation}

To obtain \eqref{eq:distB:iter1:iso}, we repeat the
computations in \eqref{eq:ingredient:iter:2}
using Hypothesis \eqref{eq:cond1:K:iter:iso} (this corresponds to
the seventh term in \eqref{eq:mathfrak1:iso}). Similarly,
we obtain \eqref{eq:disto:iter1:iso}:
\[
\dist (\bar\omega,\pd\Theta) \geq \dist(\omega,\pd\Theta) - \abs{\Deltao} \geq
\dist (\omega,\pd\Theta) - \frac{\CDeltao}{\gamma \delta^\tau} \norm{E_c}_\rho >0\,,
\]
where we used Hypothesis \eqref{eq:cond1:K:iter:iso} (this corresponds to
the eight term in \eqref{eq:mathfrak1:iso}). A direct consequence of the fact that 
the new frequency $\bar\omega$ is strictly contained in $\Theta$, and then
$\abs{\omega_*} < \abs{(1-\xi^\omega) \omega} < \sigmao \abs{\omega_*}$, is that 
we have
\begin{equation}\label{eq:prop:omega}
\frac{1}{\sigmao} < {1-\xi^\omega} < \sigmao\,,
\qquad
\frac{1}{\sigmao} < \frac{1}{1-\xi^\omega} < \sigmao\,.
\end{equation}

\bigskip

\paragraph
{\emph{Step 2: Control of the new error of invariance}.}

To control the error of invariance of the corrected parameterization $\bar K$, we
first consider the error in the solution of the linearized equation
\eqref{eq:lin1:iso}, that is
\begin{align*}
\Elin (\theta) := {} & \Ered(\theta) \xi(\theta) - \Omega_0 \Esym(\theta) \Lie{\omega}\xi(\theta)
+ 
\begin{pmatrix}
A(\theta)^\top \Elag(\theta) \\
-\Elag(\theta)
\end{pmatrix}
\hat\omega \xi^\omega
\,, \\
\Elin^\omega := {} &
\aver{
\begin{pmatrix}
\Dif ( \R{\omega}((\Dif c \circ K) E)) & 0_{n-d}^\top)
\end{pmatrix}
\xi^L} \,.
\end{align*}
First, we control $\Lie{\omega}\xi(\theta)$ in a similar fashion
as in Step 2 of the ordinary case,  
but using now that $\xi(\theta)$ is the solution of 
\eqref{eq:syst:iso:peich}. The only difference is that 
\begin{align}
\nonumber
\norm{\Lie{\omega} \xi^L}_{\rho-\delta}
={} &
\norm{\eta^L - T \xi^N - \hat\omega \xi^\omega}_{\rho-\delta} \\
\leq & \left(\CNT \cteOmega + \CT \frac{\CxiN}{\gamma \delta^\tau} + \Comega \frac{\Cxio}{\gamma \delta^\tau} \right)\norm{E}_\rho 
=: \frac{\CLiexiL}{\gamma \delta^{\tau}} \norm{E}_\rho \,,
 \label{eq:CLiexiL:iso}
\end{align}

Then, we control $\Elin(\theta)$ and $\Elin^\omega$ by
\begin{align}
\norm{\Elin}_{\rho-2\delta} \leq {} & 
\norm{\Ered}_{\rho-2\delta} \norm{\xi}_{\rho-2\delta} +
\cteOmega \norm{\Esym}_{\rho-2\delta} \norm{\Lie{\omega}\xi}_{\rho-2 \delta} \nonumber \\
& + \max \{ \norm{A}_{\rho-2\delta} \, , \, 1\} \norm{\Elag}_{\rho-2 \delta} \abs{\hat \omega} \abs{\xi^\omega} 
\nonumber \\
\leq {} &
\frac{\Cred}{\gamma \delta^{\tau+1}}
\frac{\Cxi}{\gamma^2 \delta^{2 \tau}}
\norm{E_c}_\rho^2
+
\cteOmega
\frac{\Csym}{\gamma \delta^{\tau+1}}
\frac{\CLiexi}{\gamma \delta^{\tau}}
\norm{E_c}_\rho^2 \nonumber  \\
& + \frac{\max\{ \CA \, , \, 1\} \Clag \Comega \CxiNO}{\gamma^2 \delta^{2\tau+1}}
\norm{E_c}_\rho^2
=: \frac{\Clin}{\gamma^3 \delta^{3\tau+1}} \norm{E_c}_\rho^2\,,
\label{eq:Clin:iso} \\
\abs{\Elin^\omega} \leq {} & \frac{d c_R \cteDc}{\gamma \delta^{\tau+1}} \norm{E_c}_\rho
\frac{\CxiL}{\gamma^2\delta^{2\tau}} \norm{E_c}_\rho =: \frac{\Clino}{\gamma^3 \delta^{3 \tau+1}}\norm{E_c}_\rho^2 \,.
\label{eq:Clino:iso}
\end{align}

After performing the correction, the 
total error associated with
the new parameterization is given by
(the computation is analogous to \eqref{eq:new:E:comp})
\[
\bar E_c(\theta) =
\begin{pmatrix}
\bar E(\theta) \\
\bar E^\omega
\end{pmatrix} =
\begin{pmatrix}
P(\theta) (I-\Omega_0 \Esym(\theta))^{-1} \Elin(\theta)+\Delta^2 X(\theta) + \Lie{\Deltao}\DeltaK(\theta) \\
\Elin^\omega + \aver{ \Delta^2 c(\theta)}
\end{pmatrix} \,,
\]
where $\Delta^2 X(\theta)$ is given by \eqref{eq:Delta2X} and
\begin{align*}
\Delta^2 c(\theta) = {} & \int_0^1 (1-t) \Dif^2 c(K(\theta)+t \DeltaK(\theta)) [\DeltaK(\theta),\DeltaK(\theta)] \dif t\,.
\end{align*}
Now we observe that
\[
\Lie{\Delta\omega} \DeltaK(\theta) = -\xio \Lie{\omega} \DeltaK(\theta) \,,
\]
where $\Lie{\omega} \DeltaK(\theta)$ is
controlled using the expression
\[
\Lie{\omega} \DeltaK(\theta)
= 
\Lie{\omega} L(\theta) \xi^L(\theta)
+L(\theta) \Lie{\omega} \xi^L(\theta)
+\Lie{\omega} N(\theta) \xi^N(\theta)
+N(\theta) \Lie{\omega}\xi^N(\theta) \,,
\]
thus obtaining
\begin{align}
\norm{\Lie{\omega} \Delta K}_{\rho-2\delta} \leq {} &
\left(
\frac{\CLieL \CxiL}{\gamma^2 \delta^{2\tau}} +
\frac{\CL \CLiexiL}{\gamma \delta^\tau} + \frac{\CLieN \CxiN}{\gamma \delta^{\tau}}
+\CN \CLiexiN
\right)\norm{E}_\rho
\nonumber \\
  =:{} & \frac{\CLieDeltaK}{\gamma^2 \delta^{2\tau}} \norm{E}_\rho\,. 
\label{eq:CLieDeltaK}
\end{align}
Hence, we have:
\begin{align}
\norm{\bar E}_{\rho-2\delta} \leq {} &\left(\frac{ 2(\CL+\CN)\Clin}{\gamma^3 \delta^{3\tau+1}} 
+ \frac{1}{2} \cteDDXH \frac{(\CDeltaK)^2}{\gamma^4 \delta^{4\tau}} 
+ \frac{ \Cxio \CLieDeltaK}{\gamma^3 \delta^{3\tau}}
\right)
\norm{E_c}_\rho^2  \nonumber \\ 
=: {} & \frac{\CE}{\gamma^4 \delta^{4\tau}} \norm{E_c}_\rho^2\,,
\label{eq:CE:iso}
\end{align}
where we used the second term in \eqref{eq:mathfrak1:iso}
(Hypothesis \eqref{eq:cond1:K:iter:iso}).
Moreover, we get 
\begin{align}
|\bar E^\omega| \leq & \left( \frac{\Clino}{\gamma^3 \delta^{3 \tau+1}} + \frac{1}{2} \cteDDc \frac{(\CDeltaK)^2}{\gamma^4 \delta^{4\tau}} \right) \norm{E_c}_\rho^2 =: \frac{\CEo}{\gamma^4 \delta^{4\tau}} \norm{E_c}_\rho^2\,
\label{eq:CEo:iso}
\end{align}
and, finally, 
\begin{equation}\label{eq:CEc:iso}
\norm{\bar E_c}_{\rho-2\delta} \leq 
\frac{\max \left\{
\CE \, , \, \CEo
\right\}}{\gamma^4 \delta^{4\tau}} \norm{E_c}_\rho^2 =:
\frac{\CEc}{\gamma^4 \delta^{4\tau}} \norm{E_c}_\rho^2\,.
\end{equation}
We have obtained the estimate \eqref{eq:E:iter1:iso}.
Notice that the second assumption in \eqref{eq:cond1:K:iter:iso}
and \eqref{eq:ingredient:iter:1:iso}
imply that
\begin{equation}\label{eq:barEvsE:iso}
\norm{\bar E_c}_{\rho-2\delta} < \norm{E_c}_\rho< \delta \cauxT\,.
\end{equation}

\bigskip

\paragraph
{\emph{Step 3, 4 and 5}.}
All arguments and computations presented in these steps
depend only on the invariance equation. Hence,
the control of the new frames $L(\theta)$, $N(\theta)$ and the new
transversality condition is exactly the same as in Lemma \ref{lem:KAM:inter:integral},
but replacing $\norm{E}_\rho$ by $\norm{E_c}_\rho$.
Specifically, we obtain the estimates
\eqref{eq:DK:iter1:iso} and \eqref{eq:DKT:iter1:iso}
using Hypothesis \eqref{eq:cond1:K:iter:iso} (they correspond to the
third and fourth term in \eqref{eq:mathfrak1:iso}, respectively).
We obtain the estimates \eqref{eq:B:iter1:iso} and \eqref{eq:est:DeltaB:iso}
following the computations in \eqref{eq:ingredient:iter:6} and \eqref{eq:CDeltaB}
(using the fifth term in \eqref{eq:mathfrak1:iso}).

\bigskip

\paragraph
{\emph{Step 6: Control of the action of left operator}.}
Notice that the action of $\Lie{\omega}$ is affected by
the change of the frequency, since now the natural operator to control
is $\Lie{\bar \omega}$. 
From now on, given any object $X$,
we introduce the operator
\begin{equation}\label{eq:overline:operator}
\DLieo X (\theta) := 
\Lie{\bar \omega} \bar X(\theta) - \Lie{\omega} X(\theta)
\end{equation}
for the convenience of notation.

The control of $\Lie{\bar \omega} \bar K$ is straightforward, since
\[
\norm{\Lie{\bar \omega} \bar K}_{\rho-2\delta} \leq
\norm{\bar E_c}_{\rho-2\delta}+\norm{X_\H \circ \bar K}_{\rho-2\delta}
\leq
\delta \cauxT + \cteXH = \CLieK\,,
\]
and similarly we obtain
$\norm{\Lie{\bar \omega} \bar L}_{\rho-3\delta}\leq \CLieL$
and
$\norm{\Lie{\bar \omega} \bar L^\top}_{\rho-3\delta}\leq \CLieLT$.
However, to control increments of the form
$\norm{\DLieo K}_{\rho-2\delta}$
we need to include an additional term.  
More specifically, 
\begin{align*}
\DLieo K(\theta)
= {} & \Lie{\Delta\omega}\bar K(\theta) + \Lie{\omega} \Delta K(\theta) 
= {}  -\frac{\xi^\omega}{1-\xi^\omega} \Lie{\bar\omega} \bar K(\theta) + \Lie{\omega} \Delta K(\theta)\,,
\end{align*}
where we use that $\bar\omega= (1-\xi^\omega)\omega$, and, from bounds \eqref{eq:prop:omega} and \eqref{eq:CLieDeltaK},
\begin{align}
\norm{\DLieo K}_{\rho-2\delta} \leq {} & 
\left(\frac{\sigmao \Cxio \CLieK}{\gamma \delta^\tau} + \frac{\CLieDeltaK}{\gamma^2 \delta^{2\tau}}\right)   \norm{E_c}_\rho =: \frac{\CDeltaLieK}{\gamma^2 \delta^{2\tau}} \norm{E_c}_\rho\,. 
\label{eq:CDeltaLieKiso}
\end{align}

We now observe that this is the only estimate that must be updated, since
it is the only place where cohomological equations play a role. For
example, we have
\begin{align*}
\DLieo L(\theta)
= {} &
\begin{pmatrix}
\DLieo \Dif K(\theta)
&
\DLieo(X_p \circ K)
\end{pmatrix} \\
= {} &
\begin{pmatrix}
\Dif (\DLieo K(\theta))
&
\Dif X_p \circ \bar K\, \DLieo K + (\Dif X_p \circ \bar K - \Dif X_p \circ K)\, \Lie{\omega} K
\end{pmatrix}
\end{align*}
and this expression yields formally to the same estimate in~\eqref{eq:CDeltaLieL},
but using the constants $\CDeltaLieK$, $\CDeltaK$ and $\CLieK$ defined in
this section (and replacing $E$ by $E_c$). This also affects to the control
of the objects 
$\DLieo L^\top$,
$\DLieo (G \circ K)$,
$\DLieo G_L$,
$\DLieo (\Omega \circ K)$,
$\DLieo \tilde \Omega_L$,
$\DLieo A$,
$\DLieo (J \circ K)$,
$\DLieo N^0$ and
$\DLieo N$.

\bigskip

\paragraph
{\emph{Step 7: Control of the new torsion condition}.}
Now we consider the control of the extended torsion matrix $T_c(\theta)$
and the corresponding non-degeneracy condition. First, we observe
that the upper-left block $T(\theta)$ is controlled as in
Lemma \ref{lem:KAM:inter:integral}. 
Thus, we control the extended torsion as
\begin{align}
\norm{\bar T_c - T_c}_{\rho-3\delta}
\leq {} & \max
\left\{
\frac{\CDeltaT}{\gamma^2 \delta^{2\tau+1}} + \frac{\CDeltao}{\gamma \delta^\tau}
\, , \,
\frac{\cteDc \CDeltaN}{\gamma^2 \delta^{2\tau+1}}+
\frac{\cteDDc \CDeltaK \CN }{\gamma^2 \delta^{2\tau}}
\right\}
\norm{E_c}_\rho \nonumber \\
=: {} & \frac{\CDeltaTc}{\gamma^2 \delta^{2\tau+1}} \norm{E_c}_\rho\,. \label{eq:CDeltaTc}
\end{align}
Finally, we obtain the estimates \eqref{eq:T:iter1:iso} and \eqref{eq:est:DeltaT:iso}
by adapting the computations in \eqref{eq:ingredient:iter:8} and \eqref{eq:CDeltaTOI}.
We use the second term in \eqref{eq:mathfrak1:iso}
(Hypothesis \eqref{eq:cond1:K:iter:iso}) to get
the estimate
\begin{equation}\label{eq:CDeltaTcOI}
|\aver{\bar T_c}^{-1} - \aver{T_c}^{-1}| \leq \frac{2 (\sigmaTc)^2 \CDeltaTc}{\gamma^2 \delta^{2\tau+1}}
\norm{E}_\rho =: \frac{\CDeltaTcOI}{\gamma^2 \delta^{2\tau+1}}
\norm{E}_\rho\,.
\end{equation}
This completes the proof of the lemma.
\end{proof}

\subsection{Convergence of the iterative process}\label{ssec:proof:KAM:iso}

Now we are ready to proof our second KAM theorem with conserved quantities.
Again, we comment the differences with respect to Theorem \ref{theo:KAM}
and omit the common parts.

\begin{proof}
[Proof of Theorem \ref{theo:KAM:iso}]
Let us consider the approximate torus $K_0:=K$ with frequency $\omega_0:=\omega$ and
with initial errors 
$E_0:=E$ and $E_0^\omega = E^\omega$. We
also introduce $B_0:=B$, $T_0:=T$, $T_{c,0}=T_c$ and $E_{c,0}=E_c$
associated with the initial approximation. We reproduce the
iterative construction in the proof of Theorem \ref{theo:KAM}, but
applying Lemma \ref{lem:KAM:inter:integral:iso} recursively, and
taking intro account the evolution of the error $E_{c,s}$ at
the $s$-step of the quasi-Newton method. 

Computations are the same mutatis mutandis. In this case, we
need to consider additional computations regarding the 
correction of the frequency. 
In particular, the eight term in \eqref{eq:cond1:K:iter:iso} is
checked as follows
\[
\dist (\omega_s,\partial \Theta) - \frac{\CDeltao \norm{E_{c,s}}_{\rho_s}}{\gamma \delta_s^{\tau}} > 
\dist (\omega_0,\partial \Theta) - \frac{\CDeltao \norm{E_{c,0}}_{\rho_0}}{\gamma \delta_0^{\tau}}
\frac{1}{1-a_1^{-3\tau}} > 0\,,
\]
where the last inequality is included into \eqref{eq:KAM:HYP:iso}.

Having guaranteed all hypothesis of Lemma \ref{lem:KAM:inter:integral:iso},
we collect the assumptions by introducing the constant $\mathfrak{C}_1$ as
\begin{equation}\label{eq:cte:mathfrakC1:iso}
\mathfrak{C}_1:=\max
\left\{
(a_1 a_3)^{4\tau} \CEc \, , \, (a_3)^{2\tau+1} \gamma^2 \rho^{2\tau-1} \CDeltatot
\right\}
\end{equation}
where
\begin{align}
\CDeltaI := {} &
\max
\Bigg\{
\frac{d \CDeltaK}{\sigmaDK-\norm{\Dif K}_{\rho}} \, , \,
\frac{2n \CDeltaK}{\sigmaDKT-\norm{(\Dif K)^\top}_{\rho}} \, , \label{eq:CDeltaI:iso}
\\
& \qquad\qquad\qquad\qquad  \frac{\CDeltaB}{\sigmaB-\norm{B}_{\rho}} \, , \,
\frac{\CDeltaTcOI}{\sigmaTc-\abs{\aver{T_{c}}^{-1}}}
\Bigg\} \,, \nonumber \\
\CDeltaII := {} & \frac{\CDeltaK \delta}{\dist(K(\TT^d_{\rho}),\partial \B)}\,,
\label{eq:CDeltaII:iso} \\
\CDeltaIII := {} & \frac{\CDeltao \gamma \delta^{\tau+1}}{\dist(\omega,\partial \Theta)}\,,
\label{eq:CDeltaIII:iso} \\
\CDeltatot := {} & \max
\Bigg\{
\frac{\gamma^2 \delta^{2\tau}}{\cauxT} \, , \,
2 \Csym \gamma \delta^\tau \, , \,
\frac{\CDeltaI}{1-a_1^{1-2\tau}} \, , \, 
\frac{\CDeltaII}{1-a_1^{-2\tau}} \, , \,
\frac{\CDeltaIII}{1-a_1^{-3\tau}}
\Bigg\}
\label{eq:CDeltatot:iso} \,.
\end{align}
Note that we recovered the original notation $K=K_0$, $\omega=\omega_0$, $B=B_0$, $T_c=T_{c,0}$, $\rho=\rho_0$
and $\delta=\delta_0$ for the original objects.

Therefore, by induction, we can apply the iterative process infinitely many times.
Indeed, we have
\[
\norm{E_{c,s}}_{\rho_s} < a_1^{-4 \tau s} \norm{E_c}_{\rho} \longrightarrow  0\
\quad \mbox{when} \quad s\rightarrow 0\,
\]
so the iterative scheme converges to a true quasi-periodic torus $K_\infty$ with frequency
$\omega_\infty$.
As a result of the output of Lemma \ref{lem:KAM:inter:integral:iso}, 
these objects satisfy
$K_\infty \in \Anal(\TT^{d}_{\rho_\infty})$, $\omega_\infty\in \Theta$ and
\[
\norm{\Dif K_\infty}_{\rho_\infty} < \sigmaDK\,,
\qquad
\norm{(\Dif K_\infty)^\top}_{\rho_\infty} < \sigmaDKT\,,
\qquad
\dist(K_\infty(\TT^d),\pd \B) > 0\,.
\]
Furthermore, we control the
the limit objects as follows:
\begin{align}
\norm{K_\infty - K}_{\rho_\infty} < \frac{\CDeltaK \norm{E_c}_{\rho}}{\gamma^2 \delta^{2\tau}}
\frac{1}{1-a_1^{-2\tau}} 
=: {} &  \frac{\mathfrak{C}_2 \norm{E_c}_{\rho}}{\gamma^2 \rho^{2\tau}}\,,
\label{eq:Cmathfrak2:iso} \\
\abs{\omega_\infty - \omega} <
\frac{\CDeltao \norm{E_c}_{\rho}}{\gamma \delta^{\tau}} 
\frac{1}{1-a_1^{-3\tau}}
=: {} &
\frac{\mathfrak{C}_3 \norm{E_c}_{\rho}}{\gamma \rho^{\tau}} \,,
\label{eq:Cmathfrak3:iso}
\end{align}
thus obtaining the estimates in \eqref{eq:close:iso}.
This completes the proof of the generalized iso-energetic KAM theorem.
\end{proof}

\section{Acknowledgements}
A. H. acknowledges the Spanish grants MTM2015-67724-P
(MINECO/FEDER), MDM-2014-0445 (MINECO) and 2014 SGR 1145 (AGAUR),
and the European Union's Horizon 2020 research and innovation programme
MSCA 734557.
A. L. acknowledges the Spanish grants 
MTM2016-76072-P (MINECO/FEDER) and
SEV-2015-0554 (MINECO, Severo Ochoa Programme for Centres of Excellence
in R\&D), and the ERC Starting Grant~335079. We are also grateful to
D. Peralta-Salas and
J.-Ll. Figueras for fruitful discussions.

\bibliographystyle{plain}
\bibliography{references}

\appendix

\section{An auxiliary lemma to control the inverse of a matrix}

To prove Lemmas \ref{lem:KAM:inter:integral} and
\ref{lem:KAM:inter:integral:iso}, we control the correction
of inverses of matrices several times using Neumann series
This affects to the estimates in 
\eqref{eq:est:DeltaB}, \eqref{eq:est:DeltaT},
\eqref{eq:est:DeltaB:iso} and \eqref{eq:est:DeltaT:iso}.
For convenience, we present the following auxiliary result separately.
Notice that the result is presented for matrices but
it is directly extended for matrix-valued maps with the
corresponding norm (see Section \ref{ssec-anal-prelims}).

\begin{lemma}\label{lem:aux}
Let $M \in \CC^{n \times n}$ be an invertible matrix
satisfying $\abs{M^{-1}} < \sigma$. Assume that $\bar M \in \CC^{n \times n}$
satisfies
\begin{equation}\label{eq:lem:aux}
\frac{2 \sigma^2 \abs{\bar M-M}}{\sigma-\abs{M^{-1}}} \leq 1\,.
\end{equation}
Then, we have that $\bar M$ is invertible and
\[
\abs{\bar M^{-1}-M^{-1}} < 2 \sigma^2 \abs{\bar M-M}\,,
\qquad
\abs{\bar M^{-1}} < \sigma\,.
\]
\end{lemma}

\begin{proof}
A direct computation shows that
\begin{equation}\label{eq:Neu1}
\bar M^{-1} = (I_n + M^{-1} (\bar M-M))^{-1} M^{-1}\,.
\end{equation}
By hypothesis \eqref{eq:lem:aux} we obtain
\begin{equation}\label{eq:Neu2}
\abs{M^{-1}}\abs{\bar M-M} < \frac{\sigma^2 \abs{\bar M-M}}{\sigma-\abs{M^{-1}}} < \frac{1}{2}\,.
\end{equation}
Then a Neumann series argument in \eqref{eq:Neu1}, using \eqref{eq:Neu2} and $\abs{M^{-1}} < \sigma$, yields
the estimate
\[
\abs{\bar M^{-1}-M^{-1}} \leq \frac{\abs{M^{-1}}^2 \abs{\bar M-M}}{1-\abs{M^{-1}} \abs{\bar M-M}} <
2 \sigma^2 \abs{\bar M-M}\,.
\]
Finally, we conclude that
\begin{align*}
\abs{\bar M^{-1}} \leq  {} & \abs{M^{-1}} + \abs{\bar M^{-1}-M^{-1}} \leq \abs{M^{-1}} +
 2 \sigma^2 \abs{\bar M-M} \\
< {} & \abs{M^{-1}} + \sigma - \abs{M^{-1}} = \sigma\,,
\end{align*}
where we used again \eqref{eq:lem:aux}.
\end{proof}

\section{Compendium of constants involved in the KAM theorem}\label{ssec:consts}

In this appendix we collect the recipes to compute all constants
involved in the different estimates presented in the paper. Keeping
track of these constants is crucial to apply the presented KAM
theorems in particular problems and for concrete values of
parameters. In addition, we think that the labels included in the tables 
will be of valuable help assisting the reading
of the paper. Thus,
the reader can find the place where a particular object is estimated.

Given an object $X :\TT^n_{\rho_*} \to \CC^{n_1 \times n_2}$, the following tables
code an estimate of the form
\[
\norm{X}_{\rho_*} \leq\frac{C_X}{\gamma^{a_*} \delta^{b_*}} \norm{E_*}_\rho^{c_*}\,.
\]
Notice that the strip $\rho_*$ and
the exponents $a_*, b_*, c_*$ can be tracked following the corresponding label;
and $E_*$ is the target error ($E_*=E$ for Theorem \ref{theo:KAM}
and $E_*=E_c$ for Theorem \ref{theo:KAM:iso}).

Let us remark that, as it becomes clear in the proof, the numbers
$a_1$, $a_2$, $a_3$ and $\cauxT$ are independent parameters, that can
be selected in order to optimize the applicability of the theorems
depending on the particular problem at hand.

\begin{remark}
Table \ref{tab:constants:all} corresponds to the geometric
construction that is common to both theorems. The constants
associated to the ordinary KAM theorem \ref{theo:KAM}
are presented in Tables
\ref{tab:constants:all:2} and
\ref{tab:constants:all:3}.
The constants associated to the iso-energetic KAM theorem \ref{theo:KAM:iso}
are presented in Tables
\ref{tab:constants:all:2:iso} and
\ref{tab:constants:all:3:iso}.
To reduce the length of the tables, in the iso-energetic situation
we have omitted those constants that have the same formula that in
the ordinary case. 
\end{remark}

\newpage

\bgroup
{\scriptsize
\def\arraystretch{1.5}
\begin{longtable}{|l l l l|}
\caption{Constants introduced in Section \ref{sec:lemmas}. 
Constants with $^*$ in the label are $0$ in {\bf Case III}. 
We denote by $\chi_0$ the
characteristic function of the set $\{0\}$.} 
\label{tab:constants:all} \\
\hline
Object & Constant & Label & Result\\
\hline
\hline
$\Lie{\omega}\Omega_K$ &
$\CLieOmegaK = 2 n \cteOmega \sigmaDK+ \sigmaDKT \cteDOmega \sigmaDK \delta+ d\sigmaDKT \cteOmega$ & 
\eqref{eq:CLieOK} & Lemma \ref{lem:isotrop}\\
\hline
$\Omega_K$ & $\COmegaK = c_R \CLieOmegaK$ & \eqref{eq:COK} & Lemma \ref{lem:isotrop}\\
\hline
$L$ & $\CL =  \sigmaDK+\cteXp$ & \eqref{eq:CL} & Lemma \ref{lem:Lang} \\
\hline
$L^\top$ & $\CLT =   \max\{\sigmaDKT \ , \ \cteXpT\}$ & \eqref{eq:CLT} & Lemma \ref{lem:Lang} \\
\hline
$\Elag$ & $\Clag = c_R \max\{ \CLieOmegaK + \cteDpT \ , \ d \cteDp \}$ & \eqref{eq:Clag} & Lemma \ref{lem:Lang}\\
\hline
$G_L$ & $\CGL =  \CLT \cteG \CL$ & \eqref{eq:CGL} & Lemma \ref{lem:Lang} \\
\hline
$\tilde \Omega_L$ & $\CtOmegaL = \CLT \ctetOmega \CL$ & \eqref{eq:CtOmegaL} & Lemma \ref{lem:Lang} \\
\hline
$N^0$ & $\CNO = \cteJ \CL$ & \eqref{eq:CNO} & Lemma \ref{lem:sympl} \\
\hline
$(N^0)^\top$ & $\CNOT = \CLT \cteJT$ & \eqref{eq:CNOT} & Lemma \ref{lem:sympl} \\
\hline
$A$ & $\CA = \tfrac{1}{2} (\sigmaB)^2 \, \CtOmegaL$
 & \eqref{eq:cA}$^*$ & Lemma \ref{lem:sympl} \\
\hline
$N$ & $\CN = \CL \CA + \CNO \sigmaB$ & \eqref{eq:CN} & Lemma \ref{lem:sympl} \\
\hline
$N^\top$ & $\CNT = \CA \CLT + \sigmaB \CNOT$ & \eqref{eq:CNT} & Lemma \ref{lem:sympl} \\
\hline
$\Esym$ & 
$\Csym =  (1+\CA) \max\{1,\CA +  (\sigmaB)^2 \chi_0(C_A) \} \, \CtOmegaL$
 & \eqref{eq:Csym} & Lemma \ref{lem:sympl} \\
\hline
$\Lie{\omega}K$ & $\CLieK = \delta \cauxT + \cteXH$ & \eqref{eq:CLieK} & Lemma \ref{lem:twist} \\
\hline
$\Lie{\omega}L$ & $\CLieL = d \cauxT + \cteDXH \sigmaDK + \cteDXp \CLieK$ & \eqref{eq:CLieL} & Lemma \ref{lem:twist} \\
\hline
$\Lie{\omega}L^\top$ & $\CLieLT = \max \{2n \cauxT + \cteDXHT \sigmaDK \, , \, \cteDXpT \CLieK \}$ & \eqref{eq:CLieLT} & Lemma \ref{lem:twist} \\
\hline
$\Lie{\omega} (J\circ K)$ & $\CLieJ = \cteDJ \CLieK$ & \eqref{eq:CLieJ} & Lemma \ref{lem:twist} \\
\hline
$\Lie{\omega} (G\circ K)$ & $\CLieG = \cteDG \CLieK$ & \eqref{eq:CLieG} & Lemma \ref{lem:twist} \\
\hline
$\Lie{\omega} (\tilde \Omega\circ K)$ & $\CLietOmega = \cteDtOmega \CLieK$ & \eqref{eq:CLietOmega} & Lemma \ref{lem:twist} \\
\hline
$\Lie{\omega} N^0$ & $\CLieNO = \CLieJ \CL+\cteJ \CLieL$ & \eqref{eq:CLieNO} & Lemma \ref{lem:twist} \\
\hline
$\Lie{\omega} G_L$ & 
$\CLieGL = \CLieLT \cteG \CL + \CLT \CLieG \CL + \CLT \cteG 
\CLieL$ & \eqref{eq:CLieGL} & Lemma \ref{lem:twist} \\
\hline
$\Lie{\omega} \tilde \Omega_L$ & 
$\CLietOmegaL = \CLieLT \ctetOmega \CL + \CLT \CLietOmega \CL + \CLT \ctetOmega 
\CLieL$ & \eqref{eq:CLietOmegaL} & Lemma \ref{lem:twist} \\
\hline
$\Lie{\omega} B$ & 
$\CLieB = (\sigmaB)^2 \CLieGL$ & \eqref{eq:CLieB} & Lemma \ref{lem:twist} \\
\hline
$\Lie{\omega} A$ &  $\CLieA = \CLieB \CtOmegaL \sigmaB + \tfrac{1}{2} (\sigmaB)^2 \CLietOmegaL$
& \eqref{eq:CLieA}$^*$ & Lemma \ref{lem:twist} \\
\hline
$\Lie{\omega}N$ &
$\CLieN = \CLieL \CA + \CL \CLieA + \CLieNO \sigmaB + \CNO \CLieB$ &
\eqref{eq:CLieN} & Lemma \ref{lem:twist} \\
\hline
$\Loper{L}$ & $\CLoperL = d + \cteDXp \delta$ & \eqref{eq:CLoperL} & Lemma \ref{lem:twist} \\
\hline
$\Loper{L}^\top$ & $\CLoperLT = \max\{ 2n  \, , \, \cteDXpT \delta\}$ & \eqref{eq:CLoperLT} & Lemma \ref{lem:twist} \\
\hline
$\Loper{N}$ & {$\!\begin{aligned} 
\CLoperN = 
{} &
\CLoperL \cauxT \CA+\cteDXH \CNO \sigmaB + \CL \CLieA \\
& + \CLieNO \sigmaB + \CNO \CLieB
\end{aligned}$}
$\vphantom{
\left\{\begin{array}{l}
0 \\
0 
\end{array}\right.}$
& \eqref{eq:CLoperN} & Lemma \ref{lem:twist} \\
\hline
$T$ & $\CT = \CNT \cteOmega \CLoperN$ & \eqref{eq:CT} & Lemma \ref{lem:twist} \\
\hline
$\Lie{\omega} \Omega_L$ & $\CLieOmegaL = \max\{ \CLieOmegaK + \cteDpT\,,\, d \cteDp \}$ &
\eqref{eq:CLieOmegaL} & Lemma \ref{lem:reduc} \\
\hline
$\Ered^{1,1}$ & $\Creduu = \CNT \cteOmega \CLoperL$ & \eqref{eq:Creduu} & Lemma \ref{lem:reduc} \\
\hline
$\Ered^{2,1}$ & $\Creddu = \CLT \cteOmega \CLoperL$ & \eqref{eq:Creddu} & Lemma \ref{lem:reduc} \\
\hline
$\Ered^{2,2}$ & 
{$\!\begin{aligned} 
\Creddd =
{} &
 (\CLT \cteDOmega \CN \delta + \CLoperLT \cteOmega \CN + \CLieOmegaL \CA) \gamma \delta^\tau \\
& + \Clag \CLieA
\end{aligned}$}
$\vphantom{
\left\{\begin{array}{l}
0 \\
0 
\end{array}\right.}$
& \eqref{eq:Creddd} & Lemma \ref{lem:reduc} \\
\hline
$\Ered$ &
$\Cred =
\max \{
\Creduu \gamma \delta^{\tau}
\, , \,
\Creddu \gamma \delta^{\tau} + \Creddd
\}$
& \eqref{eq:CEred} & Lemma \ref{lem:reduc} \\
\hline
\end{longtable}}
\egroup

\newpage

\bgroup
{\tiny
\def\arraystretch{1.5}
\begin{longtable}{|l l l|}
\caption{\scriptsize Constants introduced in Lemma \ref{lem:KAM:inter:integral}.
Given an object $X$, we use the notation $\Delta \LieO X = \Lie{\omega} \bar X-\Lie{\omega} X$.
Constants with $^*$ in the label are $0$ in {\bf Case III}. 
} \label{tab:constants:all:2} \\
\hline
Object & Constant & Label \\
\hline
\hline
$\xi_0^N$ & $\CxiNO = \sigmaT(\CNT \cteOmega \gamma \delta^\tau + c_R \CT \CLT \cteOmega)$
& \eqref{eq:CxiNO} \\
\hline
$\xi^N$ &
$\CxiN = \CxiNO + c_R \CLT \cteOmega$ & \eqref{eq:CxiN} \\ 
\hline
$\xi^L$ &
$\CxiL = c_R (\CNT \cteOmega \gamma \delta^\tau + \CT \CxiN)$ & \eqref{eq:CxiL} \\
\hline
$\xi$ & 
$\Cxi = \max \{\CxiL \, , \, \CxiN \gamma \delta^\tau\}$ & \eqref{eq:Cxi} \\
\hline
$\DeltaK$ &
$\CDeltaK = \CL \CxiL + \CN \CxiN \gamma \delta^\tau$ & \eqref{eq:CDeltaK} \\
\hline
$\Lie{\omega} \xi^N$ &
$\CLiexiN=  \CLT \cteOmega$
& \eqref{eq:CLiexiN} \\
\hline
$\Lie{\omega} \xi^L$ &
$\CLiexiL= \CNT \cteOmega \gamma \delta^\tau + \CT \CxiN$
& \eqref{eq:CLiexiL} \\
\hline
$\Lie{\omega} \xi$ &
$\CLiexi = \max \{ \CLiexiL \, , \, \CLiexiN \gamma \delta^\tau\}$
& \eqref{eq:CLiexi} \\
\hline
$\Elin$ &
$\Clin = \Cred \Cxi + \cteOmega \Csym \CLiexi \gamma \delta^\tau$ &
\eqref{eq:Clin} \\
\hline
$\bar E$ &
$\CE = 2(\CL +\CN) \Clin \gamma \delta^{\tau-1} + \tfrac{1}{2} \cteDDXH (\CDeltaK)^2$ & \eqref{eq:CE} \\
\hline
$\DeltaL$ &
$\CDeltaL = (d + \cteDXp\delta ) \CDeltaK$ &
\eqref{eq:CDeltaL} \\
\hline
$\DeltaLT$ &
$\CDeltaLT = \max \{ 2n \, , \cteDXpT \delta \,\} \CDeltaK$ &
\eqref{eq:CDeltaLT} \\
\hline
$\DeltaG$ & $\CDeltaG = \cteDG \CDeltaK$ & \eqref{eq:CDeltaG} \\
\hline
$\DeltaGL$ &
$\CDeltaGL = \CLT \cteG \CDeltaL + \CLT \CDeltaG \CL \delta + \CDeltaLT \cteG \CL$ & \eqref{eq:CDeltaGL} \\
\hline
$\DeltaB$ &
$\CDeltaB = 2 (\sigmaB)^2 \CDeltaGL$ & \eqref{eq:CDeltaB} \\
\hline
$\DeltatOmega$ & $\CDeltatOmega = \cteDtOmega \CDeltaK$ & \eqref{eq:CDeltatOmega} \\
\hline
$\DeltatOmegaL$ &
$\CDeltatOmegaL = \CLT \ctetOmega \CDeltaL + \CLT \CDeltatOmega \CL \delta + \CDeltaLT \ctetOmega \CL$ &
\eqref{eq:CDeltatOmegaL} \\
\hline
$\DeltaA$ &
$\CDeltaA = 
%\left\{\begin{array}{ll}
\sigmaB \CtOmegaL \CDeltaB + \tfrac{1}{2} (\sigmaB)^2 \CDeltatOmegaL
% & \mbox{in \textbf{Case II}}\,,\\
%0 & \mbox{in \textbf{Case III}}\,. 
%\end{array}\right.
$
& \eqref{eq:CDeltaA}$^*$ \\
\hline
$\Delta J$ & $\CDeltaJ = \cteDJ \CDeltaK$ & \eqref{eq:CDeltaJ} \\
\hline
$\Delta J^\top$ & $\CDeltaJT = \cteDJT \CDeltaK$ & \eqref{eq:CDeltaJT} \\
\hline
$\DeltaNO$ &
$\CDeltaNO = \cteJ \CDeltaL + \CDeltaJ \CL \delta$
& \eqref{eq:CDeltaNO} \\
\hline
$\DeltaNOT$ &
$\CDeltaNOT = \CDeltaLT \cteJT + \CLT \CDeltaJT \delta$ & \eqref{eq:CDeltaNOT} \\
\hline
$\DeltaN$ & 
$\CDeltaN = \CL \CDeltaA + \CDeltaL \CA + \CNO \CDeltaB + \CDeltaNO \sigmaB$ & \eqref{eq:CDeltaN} \\
\hline
$\DeltaNT$ &
$\CDeltaNT = \CA \CDeltaLT + \CDeltaA \CLT + \CDeltaB \CNOT + \sigmaB \CDeltaNOT$ & \eqref{eq:CDeltaNT} \\
\hline
$\Lie{\omega}\DeltaK$ &
$\CDeltaLieK = 
\CLieL \CxiL+ ( \CL \CLiexiL +  \CLieN \CxiN ) \gamma \delta^{\tau}
+\CN \CLiexiN\gamma^2 \delta^{2\tau}$
& \eqref{eq:CDeltaLieK} \\
\hline
$\DeltaLieL$ &
$\CDeltaLieL = d \CDeltaLieK + \cteDXp \CDeltaLieK \delta + \cteDDXp \CDeltaK \CLieK \delta$ & \eqref{eq:CDeltaLieL} \\
\hline
$\DeltaLieLT$ &
$\CDeltaLieLT = \max\{ 2n \CDeltaLieK\, , \, \cteDXpT \CDeltaLieK \delta+\cteDDXpT \CDeltaK \CLieK \delta \}$ & \eqref{eq:CDeltaLieLT} \\
\hline
$\DeltaLieG$ & $\CDeltaLieG = \cteDG \CDeltaLieK + \cteDDG \CDeltaK \CLieK$ & \eqref{eq:CDeltaLieG} \\
\hline
$\DeltaLieGL$ & {$\!\begin{aligned} 
\CDeltaLieGL = {} &
\CLieLT \cteG \CDeltaL + \CLieLT \cteDG \CDeltaK \CL \delta + \CDeltaLieLT \cteG \CL \\
& +\CLT \cteDG \CLieK \CDeltaL + \CLT \CDeltaLieG \CL \delta + \CDeltaLT \cteDG \CLieK \CL \\
& + \CLT \cteG \CDeltaLieL + \CLT \cteDG \CDeltaK \CLieL \delta + \CDeltaLT \cteG \CLieL\,, 
\end{aligned}$}
$\vphantom{
\left\{\begin{array}{l}
0 \\
0 \\
0
\end{array}\right.}$
& \eqref{eq:CDeltaLieGL} \\
\hline
$\DeltaLieB$ &
$\CDeltaLieB = 
2 \sigmaB \CLieGL \CDeltaB + (\sigmaB)^2 \CDeltaLieGL$
& \eqref{eq:CDeltaLieB} \\
\hline
$\DeltaLietOmega$ & $\CDeltaLietOmega = \cteDtOmega \CDeltaLieK + \cteDDtOmega \CDeltaK \CLieK$ & \eqref{eq:CDeltaLietOmega} \\
\hline
$\DeltaLietOmegaL$ & {$\!\begin{aligned} 
\CDeltaLietOmegaL = {} &
\CLieLT \ctetOmega \CDeltaL + \CLieLT \cteDtOmega \CDeltaK \CL \delta + \CDeltaLieLT \ctetOmega \CL \\
& +\CLT \cteDtOmega \CLieK \CDeltaL + \CLT \CDeltaLietOmega \CL \delta + \CDeltaLT \cteDtOmega \CLieK \CL \\
& + \CLT \ctetOmega \CDeltaLieL + \CLT \cteDtOmega \CDeltaK \CLieL \delta + \CDeltaLT \ctetOmega \CLieL\,, 
\end{aligned}$}
$\vphantom{
\left\{\begin{array}{l}
0 \\
0 \\
0
\end{array}\right.}$
& \eqref{eq:CDeltaLietOmegaL} \\
\hline
$\DeltaLieA$ & 
{$\!\begin{aligned} 
\CDeltaLieA = {} &
\CLieB \CtOmegaL \CDeltaB + \CLieB \CDeltatOmegaL \sigmaB 
+ \CDeltaLieB \CtOmegaL \sigmaB \\ &~~+ \sigmaB \CLietOmegaL \CDeltaB 
+ \tfrac{1}{2} (\sigmaB)^2 \CDeltaLietOmegaL
\end{aligned}$}
$\vphantom{
\left\{\begin{array}{l}
0 \\
0
\end{array}\right.}$
& \eqref{eq:CDeltaLieA}$^*$ \\
\hline
$\DeltaLieJ$ & 
$\CDeltaLieJ = \cteDJ \CDeltaLieL + \cteDDJ \CDeltaK \CLieK$ &
\eqref{eq:CDeltaLieJ} \\
\hline
$\DeltaLieNO$ & 
$\CDeltaLieNO =
\CLieJ \CDeltaL + \CDeltaLieJ \CL\delta + \cteJ \CDeltaLieL + \CDeltaJ \CLieL \delta$
& \eqref{eq:CDeltaLieNO} \\
\hline
$\DeltaLieN$ & 
{$\!\begin{aligned} 
\CDeltaLieN = {} &
\CLieL \CDeltaA + \CDeltaLieL \CA + \CL \CDeltaLieA + \CDeltaL \CLieA
\\
& + \CLieNO \CDeltaB + \CDeltaLieNO \sigmaB + \CNO \CDeltaLieB + \CDeltaNO \CLieB
\end{aligned}$}
$\vphantom{
\left\{\begin{array}{l}
0 \\
0
\end{array}\right.}$
& \eqref{eq:CDeltaLieN} \\
\hline
$\Delta \Loper{N}$ & 
$\CDeltaLoperN = \cteDXH \CDeltaN + \cteDDXH \CDeltaK \CN \delta + \CDeltaLieN$
& \eqref{eq:CDeltaLoperN} \\
\hline
$\DeltaT$ &
$\CDeltaT = 
\CNT \cteOmega \CDeltaLoperN + \CNT \cteDOmega \CDeltaK \CLoperN \delta + \CDeltaNT \cteOmega \CLoperN$
& \eqref{eq:CDeltaT} \\
\hline
$\DeltaTOI$ &
$\CDeltaTOI = 2(\sigmaT)^2 \CDeltaT$ & \eqref{eq:CDeltaTOI} \\
\hline
\end{longtable}}
\egroup

\bgroup
{\scriptsize
\def\arraystretch{1.5}
\begin{longtable}{|l l|}
\caption{Constants introduced in Section \ref{ssec:proof:KAM} associated
with the convergence of que quasi-Newton method that yields
Theorem \ref{theo:KAM}, i.e., the ordinary formulation.} \label{tab:constants:all:3} \\
\hline
Constant & Label \\
\hline
\hline
$\mathfrak{C}_1 = \max
\left\{
(a_1 a_3)^{4\tau} \CE \, , \, (a_3)^{2\tau+1} \gamma^2 \rho^{2\tau-1} \CDeltatot
\right\}$
&
\eqref{eq:cte:mathfrakC1} \\
\hline
$\CDeltaI = 
\max
\Bigg\{
\frac{d \CDeltaK}{\sigmaDK-\norm{\Dif K}_{\rho}} \, , \,
\frac{2n \CDeltaK}{\sigmaDKT-\norm{\Dif K^\top}_{\rho}} \, , \, 
\frac{\CDeltaB}{\sigmaB-\norm{B}_{\rho}} \, , \,
\frac{\CDeltaTOI}{\sigmaT-\abs{\aver{T}^{-1}}}
\Bigg\}$
& \eqref{eq:CDeltaI} \\
\hline
$\CDeltaII = \frac{\CDeltaK \delta}{\dist(K(\TT^d_{\rho}),\partial \B)}$
& \eqref{eq:CDeltaII} \\
\hline
$\CDeltatot = \max
\Bigg\{
\frac{\gamma^2 \delta^{2\tau}}{\cauxT} \, , \,
2 \Csym \gamma \delta^\tau \, , \,
\frac{\CDeltaI}{1-a_1^{1-2\tau}} \, , \, 
\frac{\CDeltaII}{1-a_1^{-2\tau}}
\Bigg\}$
& \eqref{eq:CDeltatot} \\
\hline
$\mathfrak{C}_2 = a_3^{2\tau} \CDeltaK/(1-a_1^{-2\tau})$ &
\eqref{eq:Cmathfrak2} \\
\hline
$\mathfrak{C}_3 = \cteDc \mathfrak{C}_2$ &
\eqref{eq:Cmathfrak3} \\
\hline
\end{longtable}}
\egroup

\bgroup
{\scriptsize
\def\arraystretch{1.5}
\begin{longtable}{|l l l|}
\caption{Constants introduced in Lemma \ref{lem:KAM:inter:integral:iso}, i.e. in the
Iterative Lemma for the generalized iso-energetic KAM theorem.
In this table, given an object $X$, we use the notation $\Delta \LieO X = \Lie{\bar \omega} \bar X-\Lie{\omega} X$.
} \label{tab:constants:all:2:iso} \\
\hline
Object & Constant & Label \\
\hline
\hline
$\xi_0^N$ & $\CxiNO = \sigmaTc \max \Big\{\CNT \cteOmega {\gamma \delta^\tau} + {c_R}\CT \CLT \cteOmega \, , \, {\gamma \delta^\tau} + {c_R} \cteDc \CN \CLT
\cteOmega \Big\}$ & \eqref{eq:CxiNO:iso} \\
\hline
$\xio$ & $\Cxio= \CxiNO$ & \eqref{eq:CxiNO:iso} \\
\hline
$\omega$ & $\Comega = \sigmao \abs{\omega_*}$ & \eqref{eq:Comega} \\
\hline
$\Deltao$ & $\CDeltao = \Comega \CxiNO$ & \eqref{eq:CDeltao} \\
\hline
$\Lie{\omega} \xi^L$
& $\CLiexiL= \CNT \cteOmega {\gamma \delta^\tau}+ \CT {\CxiN} + \Comega {\Cxio}$ &
 \eqref{eq:CLiexiL:iso}\\
\hline
$\Elin$ &
$\Clin = \Cred \Cxi + \cteOmega \Csym \CLiexi \gamma \delta^\tau +
\max\{ \CA\,,\, 1\} \Clag \Comega \Cxio$ &
\eqref{eq:Clin:iso} \\
\hline
$\Elin^\omega$ &
$\Clin^\omega= d c_R \cteDc \CxiL$ &
\eqref{eq:Clino:iso} \\
\hline
$\Lie{\omega}\DeltaK$ &
$\CLieDeltaK = 
\CLieL \CxiL+ ( \CL \CLiexiL +  \CLieN \CxiN ) \gamma \delta^{\tau}
+\CN \CLiexiN\gamma^2 \delta^{2\tau}$
& \eqref{eq:CLieDeltaK} \\
\hline
$\bar E$ &
$\CE = 2(\CL +\CN) \Clin \gamma \delta^{\tau-1} + \tfrac{1}{2} \cteDDXH (\CDeltaK)^2
+ \Cxio \CLieDeltaK  \gamma \delta^\tau$ & \eqref{eq:CE:iso} \\
\hline
$\bar E^\omega$ & $\CEo = \Clin^\omega \gamma \delta^{\tau-1}  + \tfrac{1}{2} \cteDDc (\CDeltaK)^2$
& \eqref{eq:CEo:iso}
\\
\hline
$\bar E_c$ &
$\CEc = \max \{\CE \,,\,  \CEo \} $ & \eqref{eq:CEc:iso} \\
\hline
$\DeltaLieK$ & 
$\CDeltaLieK = {\sigmao \Cxio \CLieK}{\gamma \delta^\tau} + {\CLieDeltaK}$
& \eqref{eq:CDeltaLieKiso}\\
\hline
$\DeltaTc$ &
$\CDeltaTc = \max\{ 
\CDeltaT+\CDeltao \gamma \delta^{\tau+1}
\, , \, 
\cteDc\CDeltaN + \cteDDc \CDeltaK \CN \delta\}$
& \eqref{eq:CDeltaTc} \\
\hline
$\DeltaTcOI$ &
$\CDeltaTcOI = 2(\sigmaTc)^2 \CDeltaTc$ & \eqref{eq:CDeltaTcOI} \\
\hline
\end{longtable}}
\egroup

\bgroup
{\scriptsize
\def\arraystretch{1.5}
\begin{longtable}{|l l|}
\caption{Constants introduced in Section \ref{ssec:proof:KAM} associated
with the convergence of que quasi-Newton method that yields
Theorem \ref{theo:KAM:iso}, i.e., the iso-energetic formulation.} \label{tab:constants:all:3:iso} \\
\hline
Constant & Label \\
\hline
\hline
$\mathfrak{C}_1 = \max
\left\{
(a_1 a_3)^{4\tau} \CEc \, , \, (a_3)^{2\tau+1} \gamma^2 \rho^{2\tau-1} \CDeltatot
\right\}$
&
\eqref{eq:cte:mathfrakC1:iso} \\
\hline
$\CDeltaI = 
\max
\Bigg\{
\frac{d \CDeltaK}{\sigmaDK-\norm{\Dif K}_{\rho}} \, , \,
\frac{2n \CDeltaK}{\sigmaDKT-\norm{\Dif K^\top}_{\rho}} \, , \, 
\frac{\CDeltaB}{\sigmaB-\norm{B}_{\rho}} \, , \,
\frac{\CDeltaTcOI}{\sigmaTc-\abs{\aver{T_c}^{-1}}}
\Bigg\}$
& \eqref{eq:CDeltaI:iso} \\
\hline
$\CDeltaII = \frac{\CDeltaK \delta}{\dist(K(\TT^d_{\rho}),\partial \B)}$
& \eqref{eq:CDeltaII:iso} \\
\hline
$\CDeltaIII = \frac{\CDeltao \gamma \delta^{\tau+1}}{\dist(\omega,\partial \Theta)}$
& \eqref{eq:CDeltaIII:iso} \\
\hline
$\CDeltatot = \max
\Bigg\{
\frac{\gamma^2 \delta^{2\tau}}{\cauxT} \, , \,
2 \Csym \gamma \delta^\tau \, , \,
\frac{\CDeltaI}{1-a_1^{1-2\tau}} \, , \, 
\frac{\CDeltaII}{1-a_1^{-2\tau}} \, , \,
\frac{\CDeltaIII}{1-a_1^{1-3\tau}}  
\Bigg\}$
& \eqref{eq:CDeltatot:iso} \\
\hline
$\mathfrak{C}_2 = a_3^{2\tau} \CDeltaK/(1-a_1^{-2\tau})$ &
\eqref{eq:Cmathfrak2:iso} \\
\hline
$\mathfrak{C}_3 = a_3^{\tau} \CDeltao/(1-a_1^{-3\tau})$ &
\eqref{eq:Cmathfrak3:iso} \\
\hline
\end{longtable}}
\egroup

\end{document}